\documentclass[11pt]{article}
\usepackage{amscd,amsmath,amsfonts,amssymb,amsthm,cite,mathrsfs}
\usepackage[T1]{fontenc}
\usepackage{lmodern}
\usepackage{appendix}
\usepackage{graphicx}
\usepackage{hyperref}
\usepackage{comment}
\usepackage[left=2.cm, right=2.cm, top=2.5cm, bottom=2.5cm]{geometry}
\usepackage{xcolor}
\usepackage{enumitem}
\usepackage{setspace}
\usepackage{authblk}
\usepackage{tabularx}
\usepackage{booktabs}
\usepackage{multirow}
\usepackage{subcaption}

\allowdisplaybreaks

\numberwithin{equation}{section}

\theoremstyle{plain}
\newtheorem{theorem}{Theorem}[section]
\newtheorem{lemma}[theorem]{Lemma}
\newtheorem{corollary}[theorem]{Corollary}
\newtheorem{proposition}[theorem]{Proposition}

\theoremstyle{definition}
\newtheorem{definition}[theorem]{Definition}
\newtheorem{remark}{Remark}[section]
\newcommand{\tr}{\mathrm{Tr}}

\title{\bf Edge Universality for Inhomogeneous Random Matrices II: Markov Chain Comparison and Critical Statistics
}

\author[1]{Dang-Zheng Liu
\thanks{dzliu@ustc.edu.cn}}
\author[2]{Guangyi Zou
\thanks{zouguangyi2001@gmail.com}}

\affil[1]{School of Mathematical Sciences, University of Science and Technology of China}
\affil[2]{Department of Mathematics, University of California, Irvine}

\date{\today}

\begin{document}

\maketitle

\begin{abstract}
The first paper in this series introduced a \emph{short-to-long mixing} condition that captures mean-field GOE/GUE edge universality in the supercritical sparsity regime, for symmetric/Hermitian random matrices with independent entries and a Markov variance profile. This condition reduces the universality problem to the mixing   properties of the underlying Markov chains.

In this paper, we develop new \emph{short-to-long comparison} conditions that   extend the analysis to the subcritical and critical sparsity regimes.  Specifically, we prove that two inhomogeneous random matrices exhibit the same universal edge statistics whenever their variance-profile Markov chains are comparable,  regardless of the fine details of the matrix entries. To illustrate the power of our Markov chain comparison theorem, we derive the spectral edge statistics for several prototypical models: random band matrices, the Wegner orbital model, and Hankel-profile random matrices. These comparisons uncover a rich landscape of both universal and non-universal phenomena---shaped by geometric structure, spike patterns, and domains of stable attraction---features that lie fundamentally beyond the reach of classical random matrix theory.
\end{abstract}

\tableofcontents

\section{Introduction}

\subsection{Random Matrix Theory: Universality and Beyond}

Classical random matrix theory has centered primarily on mean-field ensembles, such as Wigner matrices and sample covariance matrices, whose entries are   independent up to Hermitian symmetry and   have comparable variances. For such highly homogeneous models, remarkably detailed  spectral properties  have  been established over recent decades,  encompassing global universality (Wigner semicircle and Mar\v{c}enko--Pastur laws)  as well as local universality (sine and Airy point processes). Notably, fluctuations at the spectral edge follow the Tracy--Widom distributions, featuring a characteristic $2/3$ scaling   exponent.
 These results have deep roots and have found profound applications across pure mathematics, statistics, numerical analysis, physics, computer science and  data science\cite{akemann2011oxford}.

Yet, beyond these idealized settings lies a broader and more realistic landscape. A rapidly growing body of research has turned to structured or inhomogeneous random matrices, in which entry variances are allowed to vary across positions. Such inhomogeneous models are not only mathematically richer in structure but also indispensable for modeling real-world data and physical systems.
In functional analysis, the boundedness of infinite random operators naturally imposes that matrix entries be non-identically distributed \cite{davidson2001local}.
In mathematical physics, random band matrices with decaying variance profiles interpolate between Anderson models and mean-field Wigner matrices \cite{Wigner1955Characteristic,wigner1967random,spencer2011random}; see, e.g., \cite{liu2025edge} for a detailed explanation. Meanwhile, in quantum statistical mechanics and many-body systems, the Eigenstate Thermalization Hypothesis requires random matrices with structured off-diagonal elements \cite{d2016quantum,deutsch2018eigenstate,srednicki1999approach}.
In applied mathematics—ranging from   numerical linear algebra to computer and data science—such random matrices are frequently highly inhomogeneous and entirely lack natural symmetries \cite{bandeira2025topics,tropp2015introduction}.
Consequently, developing a general   theory that captures the spectral properties of   arbitrarily structured (inhomogeneous) random  matrices has emerged as a central challenge.
Recently, inhomogeneous random matrices have received considerable attention, with significant progress established in the study of spectral norms \cite{MR3878726,MR3837269}, matrix concentration and deviation inequalities \cite{MR4635836,brailovskaya2024extremal,MR4823211}, and spectral outliers \cite{bandeira2024matrix, geng2024outliers,MR4234995}. Despite these developments,  the universality of local eigenvalue statistics—a central object of interest for both theoretical understanding and practical applications—remains far less well understood in inhomogeneous settings; see, for instance, \cite{liu2025edge}.

Universality, a term originating in statistical mechanics, lies at the heart of random matrix theory.  The universality principle roughly asserts that, up to appropriate moment conditions, eigenvalue statistics become asymptotically independent of the fine details of the matrix entries and fall into the GOE/GUE universality class.
These results have been firmly established for Wigner matrices, with foundational breakthroughs achieved over the past three decades \cite{erdHos2010bulk, erdHos2011universality, johansson2001universality, soshnikov1999universality, tao2010random, tao2011random}; see, e.g., \cite{bourgade2018random} for a comprehensive survey. For random
$d$-regular graphs, bulk and edge universality have been established in \cite{MR3729611, huang2024ramanujan}, while for random band matrices, important contributions include \cite{bourgade2020random, benaych2014largest, erdHos2011quantum, EK11Quantum, erdHos2013delocalization, he2019diffusion, liu2023edge, sodin2010spectral, Shcherbina2021, shcherbina2014second, yang2021delocalization, yang2022delocalization, MR4736267} (this list is far from exhaustive). Very recent breakthrough results have been obtained in \cite{dubova2025delocalization, dubova2025delocalization2, erdHos2025zigzag, fan2025blockreductionmethodrandom, liu2023edge, MR4736267, yang2025delocalization2, yau2025delocalization}. Most of these works focus on specific block variance profiles or complex Gaussian entries, with the notable exception of the one-dimensional case \cite{erdHos2025zigzag}. For state-of-the-art results and further references, see \cite{dubova2025delocalization, erdHos2025zigzag, liu2023edge}.
The desire to assess the applicability of universality results  in random matrix theory has motivated the need to go beyond the classical GOE/GUE  universality setting—in particular, to understand the influence of the variance profile and the effects of entry non-Gaussianity.

However,  one cannot generally expect GOE/GUE universality to persist  for arbitrarily inhomogeneous random matrices; notably, even a diagonal  matrix with i.i.d. entries resides   outside this universality class.
This constraint prompts two fundamental questions regarding the  boundary of the mean-field regime:
\begin{enumerate}[label=(\Roman*), leftmargin=2em]
\item\label{question1} \textit{Under what conditions do inhomogeneous random matrices manifest the universal eigenvalue statistics characteristic of mean-field GOE/GUE ensembles?}
\item\label{question2} \textit{When this universality breaks down, what emergent laws govern the edge statistics, and do these laws constitute a novel, non-mean-field universality class?}
\end{enumerate}
To address these questions, we focus on inhomogeneous random matrices with independent entries (subject to Hermitian symmetry) whose variance profile is given by a symmetric Markov matrix. Nonetheless, the presence of strong inhomogeneity and locality continues to pose substantial analytical challenges. Our primary objective is to establish a Lindeberg-type central limit theorem under suitable conditions on the Markov variance profile. Such a result would uncover a novel universal pattern in random matrix theory---one that is entirely robust against the microscopic details of the entries.

 In \textbf{Part I of this series} \cite{liu2025edge}, we provide a complete characterization of edge universality for arbitrarily inhomogeneous random matrices. This result is achieved through a combination of Markov chain comparison techniques and short-to-long mixing conditions, thereby addressing question \ref{question1}. In the supercritical sparsity regime, we observe   that global mixing washes out microscopic details and restores  the classical Tracy--Widom law. In this paper, which constitutes \textbf{Part II of this series}, we continue our analysis of inhomogeneous random matrices viewed through the lens of Markov transition matrices governed by variance profiles. Our focus turns to the local spectral statistics at the spectral edge (equivalently, the spectral norm) within the subcritical and critical regimes, addressing question \ref{question2} directly.
In stark contrast to the supercritical regime---where inhomogeneity acts merely as a perturbative correction---the interplay between sparsity and the geometric structure of the variance profile here gives rise to a rich variety of novel critical phenomena. These findings elucidate precisely how inhomogeneity deforms spectral behavior beyond the confines of the GOE/GUE universal limit. Crucially, while these phenomena are robust against the microscopic details of the matrix entries, they represent a fundamental departure from the mean-field Tracy--Widom paradigm.

The central contribution of this paper is the discovery of a \textit{Universal Reduction} principle that governs local edge eigenvalue statistics. We demonstrate that the spectral complexity of large random matrices is entirely encoded in the geometric and probabilistic structures of their variance profiles, yielding a strikingly simple correspondence:
\begin{quote}
    \textbf{\textit{One CLT, One Statistics:}} \textit{Every specific Local Central Limit Theorem for the variance-profile Markov chain dictates a corresponding universal pattern of edge statistics.}
\end{quote}
To illustrate the power of this ``\textit{One CLT, One Statistics}'' framework, we establish a Markov chain comparison theorem (Theorem \ref{prop:super_4} below) for pairs of inhomogeneous random matrices governed by distinct transition kernels.
By comparing with carefully chosen canonical models that lie beyond mean-field descriptions, we leverage this theorem to extract universal asymptotic laws.

\subsection{Models and main results}
We begin by defining the class of inhomogeneous random matrices under consideration. This class encompasses several prototypical models of independent interest: generalized Wigner matrices, random band matrices, sparse random matrices with structured variance profiles, inhomogeneous chiral Wishart-type ensembles, and Wegner orbital models.

\begin{definition}[Inhomogeneous  random matrices, IRM] \label{def:inhomo}
  An inhomogeneous  random matrix and its  deformation via  the entrywise product  “\(\circ\)”  are defined as
\begin{equation} \label{HN}
H_N = \Sigma_N \circ W_N,
\end{equation}
and
\begin{equation}\label{eq:deformed-iid}
X_N = H_N + A_N,
\end{equation}
where  the  three $N\times N$ matrices $W_N$, $\Sigma_N$ and $A_N$ satisfy the following conditions.

\begin{enumerate}[leftmargin=*,label=(A\arabic*)]
\item \label{itm:A1} \textbf{(Wigner matrix)}
The entries $\{W_{ij}\}_{1\leq i\le j\leq N}$ of $W_N$ are independent and  symmetrically distributed,  satisfying    the moment assumptions
\begin{equation}
\begin{cases}
\mathbb{E}[W_{ii}^2]=2,\;\mathbb{E}[W_{ij}^2]=1, & \text{real case},\\
\mathbb{E}[W_{ii}^2]=1,\;\mathbb{E}[|W_{ij}|^2]=1,\;\mathbb{E}[W_{ij}^2]=0, & \text{complex case},
\end{cases}
\end{equation}
and  for all $k\ge2$
\begin{equation}
\mathbb{E}\bigl[|W_{ij}\bigr|^{2k}] \;\le\; \theta^{\,k-1}(2k-1)!!.
\end{equation}
Here,  $\theta\ge 1$ is a  constant, which may depend on $N$ if required.

\item \label{itm:A2} \textbf{(Markov variance profile)}
The variance profile  $\Sigma_N=(\sigma_{ij})_{i,j=1}^N$ is a symmetric matrix with   nonnegative entries such that
$P_N:=(\sigma_{ij}^2)_{i,j=1}^N$ is a symmetric Markov transition matrix.

\item \label{itm:A3} \textbf{(Finite-rank deformation)}
The perturbation $A_N$ is symmetric (respectively, Hermitian) with  rank $r$, and admits the spectral decomposition
\begin{equation}\label{SpetralA}
A_N = Q\,\Lambda\,Q^*, \qquad \Lambda = \mathrm{diag}(a_1,\dots,a_r,0,\dots,0),
\end{equation}
for    some  orthogonal  (respectively, unitary) matrix $Q$.
\end{enumerate}

In particular, when $W_N$ is a GOE/GUE matrix, $H_N$ is referred to as an IRM Gaussian matrix.
Similarly, a second IRM matrix is obtained by replacing the profile matrix $\Sigma$ solely with $\widetilde{\Sigma}=(\tilde{\sigma}_{ij})_{i,j=1}^N$. This substitution defines the modified quantities $\widetilde{P}_N=(\tilde{\sigma}_{ij}^2)_{i,j=1}^N$, $\widetilde{H}_N=\widetilde{\Sigma}_N \circ W_N$, and consequently $\widetilde{X}_N = \widetilde{H}_N + A_N$.
\end{definition}

Given the two transition matrices $P_N=(\sigma_{ij}^2)_{i,j=1}^N$ and $\widetilde{P}_N=(\tilde{\sigma}_{ij}^2)_{i,j=1}^N$ defined above, we obtain two Markov chains on the state space $[N]:=\{1,2,\ldots,N\}$, denoted by $([N], P_N)$ and $([N], \widetilde{P}_N)$.
Their respective $n$-step transition probabilities are written as $p_n(x,y)$ and $\tilde{p}_n(x,y)$.
We emphasize that $[N]$ may be replaced by any finite state space $S$ with $|S| = N$, simply by identifying each index $i \in [N]$ with an element of $S$.

To establish edge universality in both the subcritical and critical regimes, we introduce a key notion: the \textit{short-to-long comparison} of two Markov chains. This comparison captures the asymptotic spectral rigidity and the emergence of universal patterns.
\begin{definition}[{\bf Short-to-Long Comparison}] \label{mixingdef}
Two  Markov chains $([N], P_N)$ and $([N], \widetilde{P}_N)$  are
\emph{Short-to-Long Comparable}, if their $n$-step transition probabilities  $p_n(x,y)$ and $\widetilde{p}_n(x,y)$ satisfy the following two conditions:

\begin{enumerate}[leftmargin=*,label=(B\arabic*)]
\item \label{itm:B1} \textbf{(Average upper bound)}
There exists a sequence $\{b_{n}\}$ of positive  numbers,   where $b_n=b_{n,N}$ may depend on $N$, such that
\begin{equation}\max_{x,y\in [N]}\sum_{i=1}^{n}p_i(x, y) \lor \tilde{p}_i(x, y) \le b_n.
\end{equation}

\item \label{itm:B2} \textbf{(Average $\ell^1$-$\ell^{\infty}$ distance)}
There exist  two sequences   $\{\epsilon_n\}$ and $\{\delta_n\}$ such that
\begin{equation}
\max_{x\in [N]}\sum_{y\in [N]}\left| p_n(x,y) - \tilde{p}_n(x,y) \right| \le \epsilon_n,
\end{equation}
and
\begin{equation}
    \max_{x,y\in [N]}|p_n(x,y) - \tilde{p}_n(x,y)|\le \delta_n.
\end{equation}
  Let  $\mathcal{E}_n=\sum_{i=1}^{n}\epsilon_i$ and $\Delta_n=\sum_{i=1}^{n}\delta_i$ the associated cumulative errors.
\end{enumerate}

\end{definition}

\begin{theorem}[\bf Markov Chain   Comparison Theorem] \label{prop:super_4}
Given
fixed integers $s \geq 1$, $k_1,\ldots,k_s>0$ and $r\geq 0$, let   $n_{i,j}$ be  positive  integers and
let $n=\sum_{1\leq i\leq s}\sum_{1\leq j\leq k_i}n_{i,j}$.  For two IRM matrices $X$ and $\widetilde{X}$ defined in Definition \ref{def:inhomo},
suppose that  Markov chains $([N], P_N)$ and $([N], \widetilde{P}_N)$ are
\emph{Short-to-Long Comparable}  as $N\to \infty$, with
\begin{equation}\label{equ:negligible}
    \frac{\mathcal{E}_n}{n}\rightarrow 0, ~~~ \frac{\Delta_n}{b_n}\rightarrow 0.
\end{equation}
Moreover, assume  that    for some  constant $C>0$,
\begin{equation} \label{UPB}
    n^2b_n\le C, \quad  \|A_N\|_{\mathrm{op}} \leq 1 + Cn^{-1},
\end{equation}
  and, in the non-Gaussian case, additionally
\begin{equation}\label{equ:1.11}
    \theta n^2 \cdot  \max_{x,y\in [N]}\sigma_{xy}^2 \ll 1.
\end{equation}
Then
\begin{equation}
    \mathbb{E}\bigg[\prod_{i=1}^s \tr \Big(\prod_{j=1}^{k_i} U_{n_{i,j}}\big(\frac{X}{2}\big)\Big)\bigg]
    =  \mathbb{E}\bigg[\prod_{i=1}^s \tr \Big(\prod_{j=1}^{k_i} U_{n_{i,j}}\big(\frac{\widetilde{X}}{2}\big)\Big)\bigg] + o\big((Nb_n)^{\sum_{i=1}^sk_i}\big).
\end{equation}
\end{theorem}

  Here it is worthwhile to provide a brief remark regarding Condition \ref{itm:B1} and the first assumption in \eqref{UPB}.
 \begin{remark} \label{remarkB1}
From Condition \ref{itm:B1}, we obtain the bound $b_n \geq n/N$, as derived below:
\begin{equation} \begin{aligned}
       \max_{x\in [N]} \max_{y\in [N]}\sum_{i=1}^{n} p_i(x, y)\vee\tilde{p}_i(x, y)
     &\geq \max_{x\in [N]} \frac{1}{N}\sum_{y\in [N]}\sum_{i=1}^{n}p_i(x, y) \lor \tilde{p}_i(x, y)\\
     & \geq \max_{x\in [N]} \frac{1}{N}\sum_{i=1}^{n} \Big(\sum_{y\in [N]}p_i(x, y)\Big) =\frac{n}{N},
 \end{aligned}
\end{equation}
where the final equality follows from the fact that the transition matrices are Markovian. Combining this with the first assumption in \eqref{UPB} yields $n = O(N^{1/3})$. The $n \sim N^{1/3}$ scaling of the total power is necessary to establish Tracy–Widom universality, as demonstrated in \cite{liu2025edge}.
 \end{remark}

Combining the comparison theorem established above with the spectral analysis of random band matrices on the one-dimensional discrete torus, as established in Theorems~\ref{thm:band_main} and~\ref{thm:triple_critical_convergence}, immediately allows us to characterize the edge universality and the phase transition between the subcritical and supercritical sparsity regimes.
\begin{theorem}[Edge Universality and Phase Transition]\label{thm:edge_universality}
 For  $\alpha > 1$, set 
\begin{equation}
     \tilde{p}_n(x, y)=\frac{1}{M}\sum_{k\in \mathbb{Z}}f\Big(\frac{x-y+kN}{W}\Big), \quad M:=\sum_{k\in \mathbb{Z}} f\Big(\frac{k}{W}\Big),
\end{equation}
where the density profile $f$ satisfies the comparison conditions in Definition~\ref{ass:ft} with respect to the $\alpha$-stable density $f_{\alpha}$. 
Under the same assumptions on the Markov chains $([N], P_N)$ and $([N], \widetilde{P}_N)$ as in Theorem~\ref{prop:super_4}, the scaling limits of the correlation measures established in Theorem~\ref{thm:convergence_proof} ($\alpha$-stable case) also hold for the matrix $X$. In particular, a new tricritical point process emerges at the critical bandwidth and critical external source, as described in Theorem~\ref{thm:triple_critical_convergence}.
\end{theorem}

Theorem \ref{thm:edge_universality} serves as a concrete realization of the {\it{Universal Reduction}} principle, illustrating how the diffusion properties of the Markov variance profile dictate the spectral limit. More generally, our framework allows for a systematic analysis of edge statistics across a diverse range of structured ensembles, as summarized below:
\begin{itemize}

    \item \textbf{Tricritical phenomena of random band matrices:}  For deformed random band matrices, we analyze the joint scaling limits of the bandwidth, the spectral edge scale, and finite-rank perturbations (spikes). This analysis yields a unified description of the phase transition intrinsic to band matrices and characterizes the emergent critical point processes in terms of the interplay between the spike parameters and the geometric structure of the variance profile; the complete phase diagram is presented in Table~\ref{tab:final_phase_diagram}.

    \item \textbf{Wegner orbital model:} Consider block-tridiagonal random matrices of the form \(X = \sqrt{1-\lambda}H + \sqrt{\lambda}\Lambda\), where \(H\) is composed of independent GOE/GUE blocks and \(\Lambda\) encodes the coupling between adjacent blocks. As the coupling strength \(\lambda\) varies, the associated Markov chain on the block indices undergoes a sequence of transitions: first a \textbf{frozen regime}, then a discrete \textbf{Skellam regime}, followed by a diffusive \textbf{Gaussian regime}, and ultimately a \textbf{mixing regime}. We demonstrate that the spectral edge statistics undergo an analogous transition, leading to point processes that are characterized by the local central limit theorem of the underlying random walk (Corollary \ref{coro:Wegner_orbital}).

    \item \textbf{Hankel-profile random matrices:}
    Choose   a  variance profile of the form $\sigma_{xy}^2 = f(x+y)$, which are not translation invariant. The associated Markov chain corresponds to a random walk with parity reflection, in contrast to the translation-invariant walk on the torus $\mathbb{T}$. The resulting local central limit theorem gives rise to a distinct class of critical edge statistics, which are determined by the global geometry and boundary symmetries of the variance profile (Section \ref{sec:hankel}).
    \end{itemize}

\subsection{Structure of this paper}

The remainder of this paper is organized as follows.   Section \ref{sec:section3}  reviews the Chebyshev expansion for  inhomogeneous  Gaussian matrices developed in \cite{liu2025edge} and establishes upper bounds and comparison estimates for the associated diagrammatic functions. Section \ref{sec:section4} is devoted to proving the asymptotic equivalence of mixed Chebyshev moments, which culminates in the proof of the Markov chain comparison theorem (Theorem \ref{prop:super_4}).

Section~\ref{sectionrbm} applies our framework to random band matrices, identifying a hierarchy of edge statistics across the subcritical, critical, supercritical, and tricritical regimes, and further utilizes Theorem~\ref{prop:super_4} to extend these results to random matrices with general profile functions. Section~\ref{sec:application} presents further applications of the comparison theorem to the block Wegner orbital model and Hankel-profile random matrices, and derives a non-asymptotic deviation inequality characterizing the fluctuation scale at the spectral edge. Finally, concluding remarks and open questions are discussed in Section~\ref{sec:remarks}. Technical preliminaries and auxiliary results---concerning the combinatorics of Chebyshev polynomials, properties of the Jacobi $\theta_\alpha$ function, and local limit theorems for Markov chains---are deferred to Appendices~\ref{sec:chebyshev} through~\ref{sec:markov_chain}.
  \textbf{Notation.} Throughout this paper, unless indicated otherwise, constants $c, C, C_1, \dots$ denote universal positive constants independent of $N$ whose values may vary between occurrences.
Asymptotic notation:
$f(N) = o(g(N))$ or $f(N) \ll g(N)$ means $f(N)/g(N) \to 0$ as $N \to \infty$;
$h(N) \sim g(N)$ means $h(N)/g(N) \to 1$;
and $h(N) = O(g(N))$ means $h(N)/g(N)$ is bounded.
For real numbers $a, b$, let $a \wedge b = \min\{a, b\}$ and $a \vee b = \max\{a, b\}$.

\section{Analysis of diagram functions}\label{sec:section3}

We begin by reviewing the necessary background on ribbon graphs and diagram functions, as well as the formulas for Chebyshev moments of  Gaussian IRM matrices established in \cite{geng2024outliers, liu2025edge}. Our aim is to study the diagram functions associated with connected diagrams, with particular attention to two central analytic properties: sharp upper-bound estimates and their detailed asymptotic behavior.
\subsection{Ribbon graphs   and diagrams} \label{sec:ribbow_expansion}
The mixed moments \(\prod_{j=1}^{s} \operatorname{tr} (X^{m_j})\) for a Gaussian IRM matrix can be evaluated via Wick's formula in conjunction with the associated ribbon graph expansion. The link between ribbon graphs on Riemann surfaces and classical random matrix ensembles---GUE being the prototypical example---is by now well established. We refer the reader to \cite{harer1986euler,kontsevich1992intersection,okounkov2000random,lando2004graphs,mulase2003duality} for comprehensive treatments of this connection. In \cite{liu2025edge}, Chebyshev polynomials are brought into play, and an exact identity is derived that reduces the enumeration of ribbon graphs to the evaluation of diagram functions. The reader may safely skip this subsection on a first reading.

\begin{figure}[htbp]
    \centering
    \begin{minipage}[t]{0.48\textwidth}
        \centering
        \includegraphics[width=\textwidth]{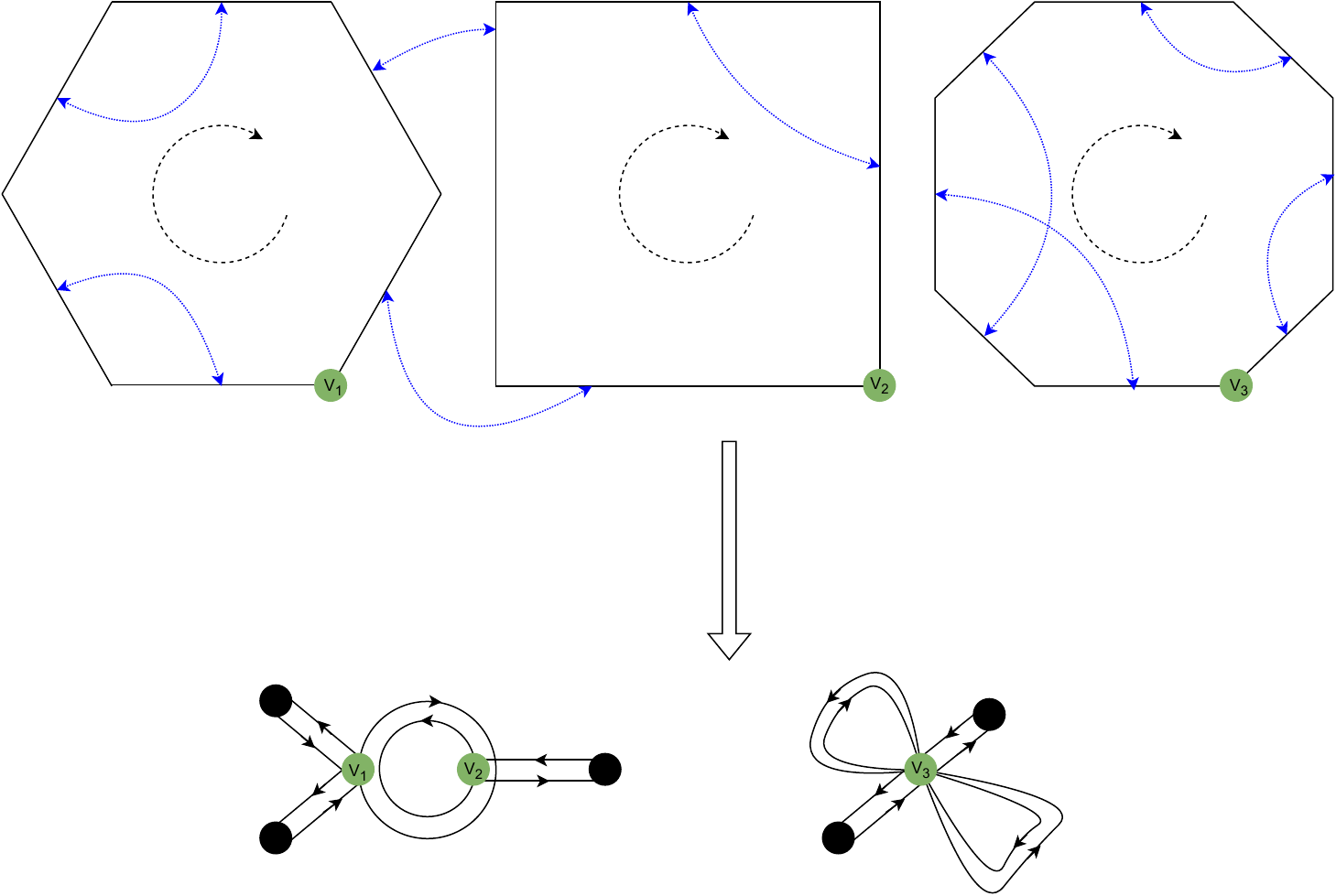}
        \caption{Example of one possible gluing in Hermitian case of $\mathbb{E}[\tr(X^{6})\tr(X^{4})\tr(X^{8})]$. The green vertices are marked vertices.}
        \label{fig:ribbon_graph_example}
    \end{minipage}
    \hfill
    \begin{minipage}[t]{0.48\textwidth}
    \centering
    \includegraphics[width=0.8\textwidth]{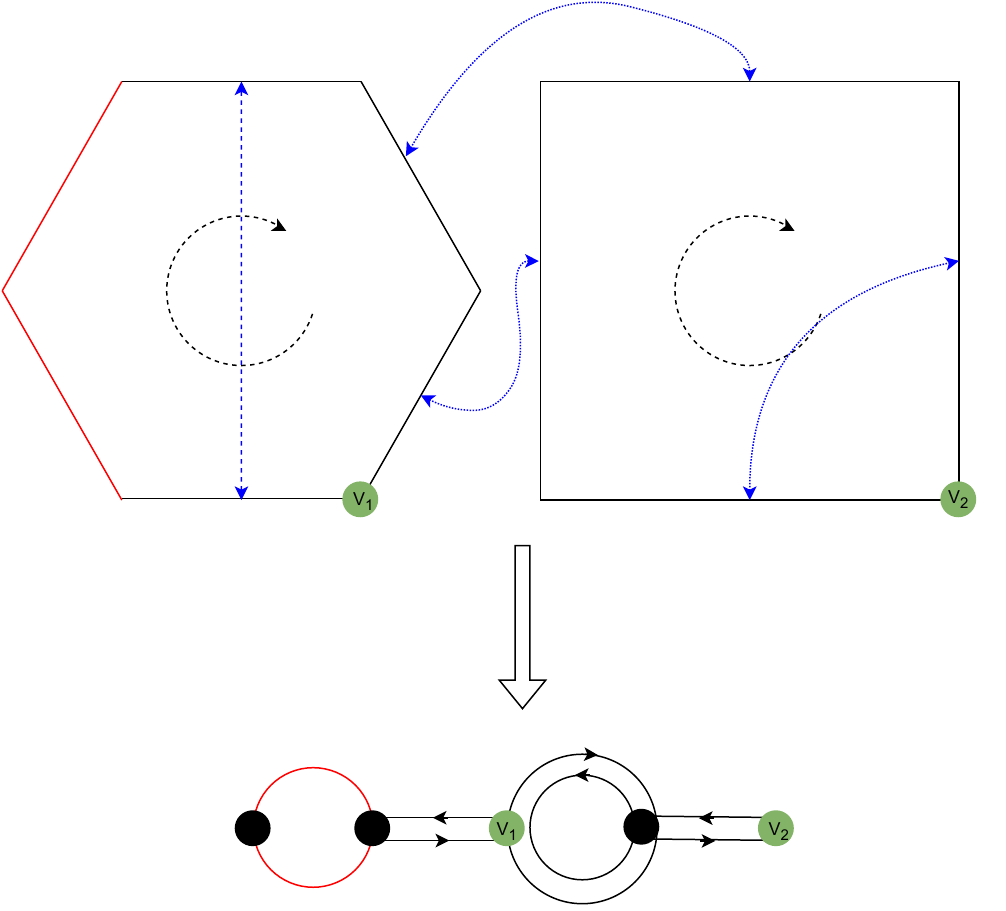}
    \caption{Example of one possible gluing of $\mathbb{E}[\tr(X^{6})\tr(X^{4})]$ with open edges.}
    \label{fig:ribbon_graph_example_open}
    \end{minipage}
\end{figure}

We recall here, for completeness, the definitions of ribbon graphs and diagrams as developed in \cite{geng2024outliers, liu2025edge}.
\begin{definition}[\bf Ribbon graph]
Let $\Sigma$ be a compact surface with or without boundary. A (punctured) ribbon graph $\Upsilon$ with $s$ faces and perimeters $(m_1,\dots,m_s)$ is a quadruple $(\mathcal{V}(\Upsilon), \mathcal{E}(\Upsilon), \iota, \phi)$, where
\begin{enumerate}
    \item[(1)] $(\mathcal{V}, \mathcal{E})$ is a graph;
    \item[(2)] $\iota: (\mathcal{V}, \mathcal{E}) \hookrightarrow \Sigma$ is an embedding;
    \item[(3)] $\phi: [s]=\{1,2,\ldots,s\} \to \mathcal{V}$ assigns a marked vertex to each face;
\end{enumerate}
such that
\begin{itemize}
    \item The boundary of the surface lies in the graph: $\partial \Sigma \subset \iota(\Upsilon)$;
    \item The complement $\Sigma \setminus \iota(\Upsilon) = \bigsqcup_{i=1}^{s} D_i$, where each $D_i$ is an oriented $m_i$-gon and $\phi(i) \in \partial D_i$.
\end{itemize}
\end{definition}

\begin{definition}[\bf Diagram/Reduced ribbon graph]\label{def:diagram}
A \emph{diagram} $\Gamma$ is a ribbon graph satisfying the following two conditions:
\begin{itemize}
    \item[(i)] Every unmarked vertex has degree at least 3;
    \item[(ii)] Every marked vertex has degree at least 2.
\end{itemize}

Let $E(\partial \Sigma), V(\partial\Sigma)$ be the edge and vertex set of $\partial \Sigma$. For the diagram, we will frequently use   the following notation:
\begin{itemize}
    \item  {\bf Edge set}: let \(E(\Gamma)\) denote the set of all edges of \(\Gamma\), with \(E_{\mathrm{int}}(\Gamma)=E(\Gamma)\setminus E(\partial\Sigma)\) and \(E_{\mathrm{b}}(\Gamma)= E(\partial\Sigma)\) denoting the sets of interior and boundary edges, respectively;
    \item {\bf Vertex   set}: let \(V(\Gamma)\), \(V_{\mathrm{int}}(\Gamma)=V(\Gamma)\setminus V(\partial\Sigma)\) and \(V_{\mathrm{b}}(\Gamma)=V(\partial\Sigma)\) denote the sets of all, interior, and boundary vertices, respectively;
    \item  {\bf Boundary  set}: for each \(j \in [s]\), let \(\partial D_j \subset E(\Gamma)\) denote the set of edges forming the boundary of face \(j\).
     \item  {\bf Diagram  set}: let $\mathscr{D}_{s;\beta}$ be the set of all $s$-cell diagrams (on orientable surfaces when $\beta=2$) and  \(\mathscr{D}_{s;\beta}^* \subset \mathscr{D}_{s;\beta}\) be the subset of connected diagrams.
\end{itemize}
\end{definition}

\subsection{Chebyshev moments for Gaussian matrices}\label{sec:section2}

\begin{definition}[Diagram function] \label{defsDF}

For each diagram $\Gamma$, we define its associated diagram function as
\begin{equation}\label{equ:F_formula}
F_{\Gamma}(\{n_j\}) := F_{\Gamma,\Sigma}(\{n_j\}) =
\sum_{\eta: V(\Gamma) \to [N]}
\sum_{\substack{
t_j \ge 0,\; w_e \ge 1 \\
2t_j + \sum\limits_{e \in \partial D_j} w_e = n_j
}}
\Bigg(
\prod_{(x, y) \in E_{\mathrm{int}}}
p_{w_e}\big(\eta(x), \eta(y)\big)
\prod_{(z, w) \in E_{\mathrm{b}}}
\big(A^{w_e}\big)_{\eta(z)\eta(w)}
\Bigg),
\end{equation}
where the second summation runs over all integers \( t_j \ge 0 \) and \( w_e \ge 1 \) satisfying the constraint \( 2t_j + \sum_{e \in \partial D_j} w_e = n_j \) for every $j$.
Analogously, the corresponding diagram function for the variance profile $\widetilde{\Sigma}$ is denoted by $\widetilde{F}_{\Gamma}(\{n_j\})= F_{\Gamma,\widetilde{\Sigma}}(\{n_j\})$.

\end{definition}

Our starting point, and an essential tool throughout this paper, is the graphical expansion for the mixed moments of Chebyshev polynomials of the second kind.
\begin{theorem}[{\cite[Theorem 2.6]{liu2025edge}}]\label{Chebyshevmoment}

Let $X$ be  the deformed Gaussian IRM  matrix in Definition \ref{def:inhomo}, then for any non-negative integers $n_j\ge 0$,
\begin{equation}
\mathbb{E}\Big[\prod_{j=1}^s \tr \big({U}_{n_j}\big(\frac{X}{2}\big)\big)\Big]=\sum_{\Gamma\in \mathscr{D}_{s;\beta}}F_{\Gamma}(\{n_j\}).
\end{equation}
\end{theorem}

In order to convert sums over full diagrams into sums over connected diagrams, we introduce the cluster decomposition of mixed Chebyshev  moments via cumulants.

\begin{definition}[Cumulants]\label{Cumu}
    The joint cumulants of mixed Chebyshev  moments are defined recursively by
    \begin{equation}\label{equ:T}
        \begin{gathered}
            \kappa_X(n_1)=\mathbb{E}\!\Big[\tr\big(U_{n_1}(\tfrac{X}{2})\big)\Big],\\
            \kappa_X(n_1,n_2)=\mathbb{E}\!\Big[\tr\big(U_{n_1}(\tfrac{X}{2})\big)\tr\big(U_{n_2}(\tfrac{X}{2})\big)\Big]-\kappa_X(n_1)\kappa_X(n_2),\\
            \cdots\\
            \kappa_X(n_1,\dots,n_s)=\mathbb{E}\!\Big[\tr\big(U_{n_1}(\tfrac{X}{2})\big)\tr\big(U_{n_2}(\tfrac{X}{2})\big)\cdots\tr\big(U_{n_s}(\tfrac{X}{2})\big)\Big]-\sum_{\Pi}\prod_{P\in\Pi}\kappa_X\big(\{n_j\}_{j\in P}\big),
        \end{gathered}
    \end{equation}
    where the sum runs over all nontrivial partitions $\Pi$ of $[s]$, i.e., excluding the partition $[s]$ itself.
\end{definition}

\begin{lemma}[{\cite[Lemma 2.8]{liu2025edge}}]\label{thm:u=F}
    For all integers $n_j\ge 1$, $j=1,\dots,s$,
    \begin{equation}
        \kappa_X(n_1,\dots,n_s)=\sum_{\Gamma\in\mathscr{D}_{s;\beta}^*} F_{\Gamma}\big(\{n_j\}_{j=1}^s\big),
    \end{equation}
    where $\mathscr{D}_{s;\beta}^*$ denotes the set of all connected $s$-cell diagrams.
\end{lemma}

\subsection{Upper bound estimates}
\begin{lemma}
Let $A$ be a Hermitian matrix with $\|A\|_{\mathrm{op}}\le a$ and $\mathrm{rank}{(A)}\le r$,  then
\begin{equation}\label{equ:3.11}
    \sum_{x_1,\ldots, x_k}
    \left|
        (A^{s_1})_{x_1x_2} (A^{s_2})_{x_2x_3} \cdots (A^{s_k})_{x_kx_1}
    \right|\le a^{s_1 + \cdots + s_k} \, r^k.
\end{equation}
\end{lemma}

\begin{proof}
By the spectral decomposition $A_N=\sum_{t=1}^r a_t \mathbf{q}_{t}^{}\mathbf{q}_{t}^{*}$ with $Q=(\mathbf{q}_{1},\ldots,\mathbf{q}_{N})$ defined in  \eqref{SpetralA},  we have
\begin{align}
    &\sum_{x_1,\ldots, x_k}
    \left|
        (A^{s_1})_{x_1x_2} (A^{s_2})_{x_2x_3} \cdots (A^{s_k})_{x_kx_1}
    \right|
    =
    \sum_{x_1,\ldots, x_k}
    \Big|
        \sum_{1 \le t_1, \ldots, t_k \le r}
        \prod_{i=1}^{k}
        a_{t_i}^{s_i} \,
        \big(\mathbf{q}_{t_i} \mathbf{q}_{t_i}^{*}\big)_{x_i x_{i+1}}
    \Big|\notag \\[4pt]
    &\le
    \sum_{1 \le t_1, \ldots, t_k \le r}
    \sum_{x_1,\ldots, x_k}
        \prod_{i=1}^{k} \left|
        a_{t_i}^{s_i} \,
        \big(\mathbf{q}_{t_i} \mathbf{q}_{t_i}^*\big)_{x_i x_{i+1}}
    \right| \le
    a^{s_1 + \cdots + s_k}
    \sum_{1 \le t_1, \ldots, t_k \le r}
    \sum_{x_1, \ldots, x_k}
    \Big|
        \prod_{i=1}^{k} \big(\mathbf{q}_{t_i} \mathbf{q}_{t_{i+1}}^*\big)_{x_i x_{i}}
    \Big| \notag
    \\[4pt]
    &\le
    a^{s_1 + \cdots + s_k}
    \sum_{1 \le t_1, \ldots, t_k \le r} 1  =
    a^{s_1 + \cdots + s_k} \, r^k,
\end{align}
where  the  Cauchy-Schwarz inequality has been used in the last inequality.
\end{proof}

\begin{proposition}\label{prop:F_upper_bound}
Let $\|A_N\|_{\mathrm{op}} = a$. Under Assumption \ref{itm:B1} in Definition \ref{mixingdef}, the diagram functions defined in \eqref{equ:F_formula} satisfy the following upper bounds.
\begin{enumerate}[label=(\roman*)]
    \item \label{item:upper_noboundary}
        If $\Gamma\in \mathscr{D}_{s;\beta}^*$  is a diagram without open edges,  then
         \begin{equation}\label{equ:3.2}
        \big|F_{\Gamma}(\{n_j\})\big|\le N\big(b_n\big)^{|E|-|V|+1}\frac{n^{|V|-1}}{(|V|-1)!}.
    \end{equation}
    \item \label{item:upper_boundary} If   $\Gamma\in\mathscr{D}_{s;\beta}^{*}$ is a diagram with at least one open edge,  then  \begin{equation}\label{equ:3.3}
        \big|F_{\Gamma}(\{n_j\})\big|\le (1+a^n)r^{|V_{\mathrm{b}}|}\big(b_n\big)^{|E_{\mathrm{int}}|-|V_{\mathrm{int}}|}\frac{n^{|V|}}{(|V|)!}.
    \end{equation}
\end{enumerate}
\end{proposition}

\begin{proof}
Starting from the exact expression  of $F_{\Gamma}$ defined in \eqref{equ:F_formula}, we  take the absolute value and divide vertices into interior and open ones to derive
\begin{align}\label{equ:3.5}
    \left|F_{\Gamma}(\{n_j\})\right|
    &\le
    \sum_{\eta: V_{\mathrm{b}}(\Gamma) \to [N]}
    \sum_{\substack{
        2 t_j + \sum_{e \in\partial D_j} w_e = n_j
    }}
         \prod_{(z,w) \in E_{\mathrm{b}}}  \left|(A^{w_e})_{\eta(z)\eta(w)}
     \right|
    \sum_{\eta: V_{\mathrm{int}}(\Gamma) \to [N]}
    \prod_{(x,y) \in E_{\mathrm{int}}} p_{w_e}(\eta(x), \eta(y)) \notag \\
    &\le
    \sum_{\eta: V_{\mathrm{b}}(\Gamma) \to [N]}
    \sum_{\substack{
        \sum_{e \in E} w_e \le n
    }}
        \prod_{(z,w) \in E_{\mathrm{b}}} \left|(A^{w_e})_{\eta(z)\eta(w)}
    \right|
    \sum_{\eta: V_{\mathrm{int}}(\Gamma) \to [N]}
    \prod_{(x,y) \in E_{\mathrm{int}}} p_{w_e}(\eta(x), \eta(y)).
\end{align} Here $n:=\sum_{j=1}^s n_j$,  we have added  all $s$ linear restrictions together and ignore the evenness of the sum restriction  in the last inequality. That is,
\begin{align} \label{ub-eq1}
    \left|F_{\Gamma}(\{n_j\})\right|
    &\le
    \sum_{ w_e: \sum w_e \le n}
    \sum_{\eta: V(\Gamma) \to [N]}
    \prod_{e \in E_{\mathrm{int}}} p_{w_e}(\eta(x), \eta(y))
    \prod_{(z,w) \in E_{\mathrm{b}}} \left| (A^{w_e})_{\eta(z) \eta(w)} \right|.
    \end{align}

Given the interior edge set $E_{\mathrm{int}}$, which can be  viewed as a graph, we select one spanning forest $\mathcal{F}$ such that each  of its   connected components   contains exactly one open (boundary) vertex, see Figure \ref{fig:forest} for example.   Denote $\mathcal{F}^c=E_{\mathrm{int}}\setminus \mathcal{F}$.
 The assumption  \ref{itm:B1}   in Definition \ref{mixingdef} shows that  \begin{equation}\label{equ:3.6}
 \sum\limits_{w_e=1}^{n}p_{w_e}(\eta(x),\eta(y))\le b_n, \quad \forall e\in \mathcal{F}^c,\end{equation}
  from which summing over $w_e$  from 1 to $n$ (relaxing the sum restriction) on the right-side hand of \eqref{ub-eq1} for every  edge $e\in \mathcal{F}^c$,  we obtain
 \begin{align}
    \left|F_{\Gamma}(\{n_j\})\right|
    &\le
    \sum_{\substack{\sum\limits_{e \in E_{\mathrm{b}} \cup \mathcal{F}} w_e \le n  }}
    (b_n)^{|\mathcal{F}^c|}
    \sum_{\eta: V(\Gamma) \to [N]}
    \prod_{e \in \mathcal{F}} p_{w_e}(\eta(x), \eta(y))
    \prod_{(z,w) \in E_{\mathrm{b}}} \left|(A^{w_e})_{\eta(z) \eta(w)}\right|
    \notag \\[6pt]
    &=
    \sum_{\substack{\sum\limits_{e \in E_{\mathrm{b}} \cup \mathcal{F}} w_e \le n}}
    (b_n)^{|\mathcal{F}^c|}
    \sum_{\eta: V_{\mathrm{b}}(\Gamma) \to [N]}
    \prod_{(z,w) \in E_{\mathrm{b}}} \left|(A^{w_e})_{\eta(z) \eta(w)}\right|
    \notag\\[6pt]
    &\le
    \sum_{\substack{\sum\limits_{e \in E_{\mathrm{b}} \cup \mathcal{F}} w_e \le n}}
    (b_n)^{|\mathcal{F}^c|} (1 + a^n) r^{|V_{\mathrm{b}}|}
      \notag\\[6pt]
    &\le
    (1 + a^n) r^{|V_{\mathrm{b}}|} (b_n)^{|\mathcal{F}^c|}
    \frac{n^{|E_{\mathrm{b}} \cup \mathcal{F}|}}{(|E_{\mathrm{b}} \cup \mathcal{F}|)!}
    \notag \\[6pt]
    &=
    (1 + a^n) r^{|V_{\mathrm{b}}|} (b_n)^{|E_{\mathrm{int}}| - |V_{\mathrm{int}}|}
    \frac{n^{|E_{\mathrm{b}}| + |V_{\mathrm{int}}|}}{(|E_{\mathrm{b}}| + |V_{\mathrm{int}}|)!}.
    \label{equ:3.16}
 \end{align}
 Here  we sum over the evaluation of $\eta$ along the forest    in the second equality  and use  \eqref{equ:3.11} in the third inequality.

Since $|E_{\mathrm{b}}|=|V_{\mathrm{b}}|$, this indeed completes the general case, particularly Case \ref{item:upper_boundary}.

\begin{figure}
    \centering
    \includegraphics[width=\linewidth]{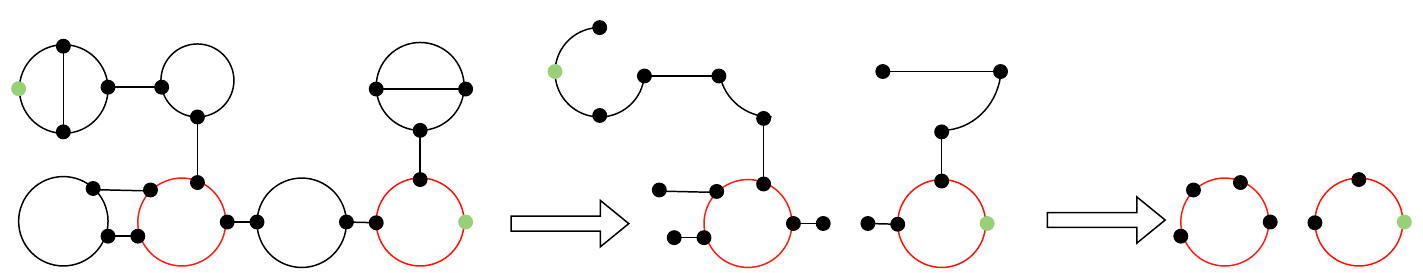}
    \caption{\textbf{Step 1:} Remove all edges in \(\mathcal{F}^c\), leaving a spanning forest \(\mathcal{F}\) where each tree has exactly one ``open’’ root vertex. \textbf{Step 2:} Sum over labels on each tree (collapsing them), so only the open edges and open vertices remain.
    }
    \label{fig:forest}
\end{figure}
For a diagram without boundary, we fix the labeling of a distinguished vertex and classify the remaining vertices as internal, denoted by \( V_{\mathrm{int}} \), with their incident edges denoted  by \( E_{\mathrm{int}} \). Since \( |V_{\mathrm{b}}| = |E_{\mathrm{b}}| = 0 \) in this setting, the configuration corresponds to the special case \( a = 0 \). We proceed as in the derivation of \eqref{equ:3.16} to obtain
\begin{align}
    F_{\Gamma}(\{n_j\})&\le N (b_n)^{|E_{\mathrm{int}}|-|V_{\mathrm{int}}|}\frac{n^{|E_{\mathrm{b}}|+|V_{\mathrm{int}}|}}{(|E_{\mathrm{b}}|+|V_{\mathrm{int}}|)!}\notag\\
    &=N (b_n)^{|E|-|V|+1}\frac{n^{|V|-1}}{(|V|-1)!}\label{equ:no_boundary_case}.
\end{align}

This completes Case \ref{item:upper_noboundary},  and thus   the entire   proof.
\end{proof}
\subsection{Asymptotic equivalence of diagram functions}
We establish asymptotic approximations for pairs of diagram functions corresponding to the variance profile matrices $\Sigma$ and $\widetilde{\Sigma}$.
\begin{theorem}\label{thm:F=tildeF}

Suppose that $\| A_N\|_{\mathrm{op}} \le 1+\tau n^{-1}$, where $n = \sum_{j=1}^s n_j$ and $\tau > 0$ is a fixed constant. Under Assumptions \ref{itm:B1}--\ref{itm:B2} of Definition~\ref{mixingdef}, for any fixed diagram $\Gamma$ the diagram functions in Definition~\ref{defsDF} satisfy the following asymptotic equivalence:

\begin{enumerate}
    \item[(i)] If $\Gamma$ has at least one open edge,  then\begin{equation}\label{equ:F_asy_equal}
F_{\Gamma}(\{n_j\})=\widetilde{F}_{\Gamma}(\{n_j\}_{j=1}^s)+O_{|E|}\Big(\left\{\big(b_n\big)^{|E_{\mathrm{int}}|-|V_{\mathrm{int}}|-1}n^{|V|-1}\right\}\big(n\Delta_n +b_n\mathcal{E}_n \big)\Big).
\end{equation}
\item[(ii)] If $\Gamma$ has no open edge, then
\begin{equation}\label{equ:F_asy_equal_no_open}
F_{\Gamma}(\{n_j\})=\widetilde{F}_{\Gamma}(\{n_j\}_{j=1}^s)+O_{|E|}\Big(\left\{\big(b_n\big)^{|E|-|V|}n^{|V|-2}N\right\}(n\Delta_n +b_n\mathcal{E}_n )\Big).
\end{equation}
\end{enumerate}

\end{theorem}

\begin{proof}
We focus on the case where $\Gamma$ has at least one open edge, since   the case without    open edges can  be handled analogously. In this setting, the   identity $|V| - |E| = |V_{\mathrm{int}}| - |E_{\mathrm{int}}|$ holds.

For convenience, we impose an ordering on $E_{\mathrm{int}}$ by writing $E_{\mathrm{int}}=\{e_1,e_2,\ldots,e_{|E_{\mathrm{int}}|}\}$ and introduce,    for $k=0,1,\ldots, |E_{\mathrm{int}}|$,
\begin{equation}
    F_{\Gamma}^{(k)}(\{n_j\})=\sum_{\eta : V(\Gamma) \to [N]}
    \sum_{\substack{2t_j + \sum_{e \in \partial D_j} w_e = n_j}}
    \prod_{i=1}^{k}\tilde{p}_{w_{e_i}}(x,y)\prod_{i=k+1}^{|E_{\mathrm{int}}|}{p}_{w_{e_i}}(x,y)
    \prod_{(z, w) \in E_{\mathrm{b}}} (A^{w_e})_{\eta(z)\eta(w)}.
\end{equation}
First,  we observe that
\begin{equation}
    F_{\Gamma}(\{n_j\})=F_{\Gamma}^{(0)}(\{n_j\}), ~~~ \widetilde{F}_{\Gamma}(\{n_j\})=F_{\Gamma}^{(|E_{\mathrm{int}}|)}(\{n_j\}),
\end{equation}
and
\begin{equation}
    \widetilde{F}_{\Gamma}(\{n_j\})-F_{\Gamma}(\{n_j\})=\sum_{k=0}^{|E_{\mathrm{int}}|-1}\left( F_{\Gamma}^{(k+1)}(\{n_j\})-F_{\Gamma}^{(k)}(\{n_j\}) \right).
\end{equation}
Next, we analyze the difference between two consecutive terms in the above interpolation.
\begin{equation}
    \begin{aligned}
        &\left|F_{\Gamma}^{(k)}(\{n_j\})-F_{\Gamma}^{(k-1)}(\{n_j\})\right|=\Big|\sum_{\eta : V(\Gamma) \to [N]}
    \sum_{\substack{2t_j + \sum_{e \in \partial D_j} w_e = n_j}}\\
        &
    \Big(\prod_{i=1}^{k-1}\tilde{p}_{w_{e_i}}(x,y)\Big)\Big(\tilde{p}_{w_{e_k}}(x,y)-{p}_{w_{e_k}}(x,y)\Big)\Big(\prod_{i=k+1}^{|E_{\mathrm{int}}|}{p}_{w_{e_i}}(x,y)\Big)
    \prod_{(z, w) \in E_{\mathrm{b}}} (A^{w_e})_{\eta(z)\eta(w)}\Big|\\
    &\le \sum_{\eta : V(\Gamma) \to [N]}
    \sum_{\sum_{e\in E(\Gamma)}w_e\le n}\left(
    \prod_{i=1}^{k-1}\tilde{p}_{w_{e_i}}(x,y)\left|\tilde{p}_{w_{e_k}}(x,y)-{p}_{w_{e_k}}(x,y)\right|\prod_{i=k+1}^{|E_{\mathrm{int}}|}{p}_{w_{e_i}}(x,y)
    \prod_{(z, w) \in E_{\mathrm{b}}} \left|(A^{w_e})_{\eta(z)\eta(w)}\right|\right).
    \end{aligned}
\end{equation}
Now consider a spanning forest $\mathcal{F}$ of $\Gamma$ and carry out  a procedure analogous to that in the derivation of \eqref{equ:3.16}. We first relax
\begin{equation}
    \sum_{\sum_{e\in E(\Gamma)}w_e\le n}\Big(\cdots\Big)\le \sum_{\sum\limits_{e\in E_{\mathrm{b}}\cup \mathcal{F}}w_e\le n}\quad\sum_{\forall e\in \mathcal{F}^c, ~w_e\le n} \Big(\cdots\Big),
\end{equation}
and thus two distinct cases arise.

\begin{enumerate}
    \item[{\bf Case I:}]   $e_k\in \mathcal{F}^c$.
    By the assumption  \ref{itm:B2}   in Definition \ref{mixingdef},
    we obtain for $e_k$,
    \begin{equation}
        \sum\limits_{w_{e}=1}^{n}\big|p_{w_e}(\eta(x),\eta(y))-\tilde{p}_{w_e}(\eta(x),\eta(y))\big|\le \Delta_n.
    \end{equation}
    For other $e\in \mathcal{F}^c$, by assumption \ref{itm:B1}, we have
    \begin{equation}
        \max\left\{\sum_{w_e=1}^{n}p_{w_e}(x, y), \sum_{w_e=1}^{n} \tilde{p}_{w_e}(x, y)\right\} \le b_n.
    \end{equation}
    Thus the total contribution from the summation over $\mathcal{F}^c$ is $b_n^{|E_{\mathrm{int}}|-|V_{\mathrm{int}}|-1}\Delta_n$.

    Then we take the rest summation following steps in \eqref{equ:3.16} and \eqref{equ:no_boundary_case}, we obtain the upper bound
    \begin{equation}\label{equ:not_in_F}
        \begin{aligned}
            \left|F_{\Gamma}^{(k)}(\{n_j\})-F_{\Gamma}^{(k-1)}(\{n_j\})\right|&\le \begin{cases}
                (1 + a^n) r^{|V_{\mathrm{b}}|} (b_n)^{|E_{\mathrm{int}}| - |V_{\mathrm{int}}|-1}\Delta_n
    \frac{n^{|E_{\mathrm{b}}| + |V_{\mathrm{int}}|}}{(|E_{\mathrm{b}}| + |V_{\mathrm{int}}|)!},& E_{\mathrm{b}}\neq\emptyset,\\
    N (b_n)^{|E|-|V|}\Delta_n\frac{n^{|V|-1}}{(|V|-1)!},&E_{\mathrm{b}}=\emptyset.
            \end{cases}
        \end{aligned}
    \end{equation}
    \item[{\bf Case II:}] $e_k\in \mathcal{F}$. In this case, we still sum over all $e\in \mathcal{F}^c$, which leads to a $b_n^{|E_{\mathrm{int}}|-|V_{\mathrm{int}}|}$ factor. We split the sum
    \begin{equation}
        \sum_{\sum\limits_{e\in E_{\mathrm{b}}\cup \mathcal{F}}w_e\le n}\big(\cdots\big)\le \sum_{w_{e_k}=1}^{n}\sum_{\sum\limits_{e\in E_{\mathrm{b}}\cup \mathcal{F}\setminus \{e_k\}}w_e\le n} \big(\cdots\big).
    \end{equation}
    Then, following the steps in \eqref{equ:3.16} yields
    \begin{equation}\label{equ:in_F}
    \small
     \begin{aligned}
     &\sum_{\eta:V(\Gamma)\rightarrow[N]}\sum_{\substack{\sum\limits_{e \in E_{\mathrm{b}} \cup \mathcal{F}} w_e \le n  }}\Big(\prod_{e\in \mathcal{F}\setminus\{e_k\}}p_{w_e}(\eta(x),\eta(y))\Big)\Big|p_{w_{e_k}}(\eta(x),\eta(y))-\tilde{p}_{w_{e_k}}(\eta(x),\eta(y))\Big|\prod_{(z,w) \in E_{\mathrm{b}}} \left|(A^{w_e})_{\eta(z) \eta(w)}\right|
        \\
        &\le \begin{cases}(1+a^n)r^{|V_{\mathrm{b}}|}\sum\limits_{\substack{w_e:\sum\limits_{e \in E_{\mathrm{b}} \cup \mathcal{F}} w_e \le n}}\epsilon_{w_{e_k}}, \quad& E_{\mathrm{b}}(\Gamma)\neq \emptyset,\\
            N\sum\limits_{\substack{w_e:\sum\limits_{e \in  \mathcal{F}} w_e \le n}}\epsilon_{w_{e_k}},\, \quad& E_{\mathrm{b}}(\Gamma)= \emptyset.
        \end{cases}\\
        &\le  \begin{cases}
            (1+a^n)r^{|V_{\mathrm{b}}|}\frac{n^{|E_{\mathrm{b}}| + |V_{\mathrm{int}}|-1}}{{(|E_{\mathrm{b}}| + |V_{\mathrm{int}}|-1)!}}\mathcal{E}_n, \quad& E_{\mathrm{b}}(\Gamma)\neq \emptyset,\\
            N\frac{n^{|V|-2}}{(|V|-2)!}\mathcal{E}_n,\, \quad& E_{\mathrm{b}}(\Gamma)= \emptyset.
        \end{cases}
         \end{aligned}
    \end{equation}
    Here in the first inequality we follow the summing over tree step in \eqref{equ:3.16} and when we sum over one end of $e_k$, we replace the $\sum_{y}p_{w_e}(x,y)=1$ by $\epsilon_{w_{e_k}}$ using assumption \ref{itm:B2}. In the second inequality we use
    \begin{equation}
        \sum\limits_{\substack{w_e:\sum w_e \le n}}\epsilon_{w_{e_k}}\le \sum\limits_{\substack{w_e:\sum w_e \le n, e\ne e_{k}}}\sum_{w_{e_k}\le n}\epsilon_{w_{e_k}}\le \sum\limits_{\substack{w_e:\sum w_e \le n, e\ne e_{k}}}\mathcal{E}_n.
    \end{equation}
    Putting back the $b_n^{|E_{\mathrm{int}}|-|V_{\mathrm{int}}|}$ factor, we obtain the same bound
    \begin{equation}
        \begin{aligned}
            \left|F_{\Gamma}^{(k)}(\{n_j\})-F_{\Gamma}^{(k-1)}(\{n_j\})\right|&\le \begin{cases}
                (1 + a^n) r^{|V_{\mathrm{b}}|} (b_n)^{|E_{\mathrm{int}}| - |V_{\mathrm{int}}|}
    \frac{n^{|E_{\mathrm{b}}| + |V_{\mathrm{int}}|-1}}{(|E_{\mathrm{b}}| + |V_{\mathrm{int}}|-1)!}\mathcal{E}_n,& E_{\mathrm{b}}\neq\emptyset,\\
    N (b_n)^{|E|-|V|+1}\frac{n^{|V|-2}}{(|V|-2)!}\mathcal{E}_n,&E_{\mathrm{b}}=\emptyset.
            \end{cases}
        \end{aligned}
    \end{equation}
\end{enumerate}

Therefore, combining \eqref{equ:not_in_F} and \eqref{equ:in_F}, for a diagram $\Gamma$, the upper bound becomes
\begin{equation}
    \left| F_{\Gamma}^{(k+1)}(\{n_j\}) - F_{\Gamma}^{(k)}(\{n_j\}) \right| \le \begin{cases}
                (1 + a^n) r^{|V_{\mathrm{b}}|} (b_n)^{|E_{\mathrm{int}}| - |V_{\mathrm{int}}|}
    \frac{n^{|E_{\mathrm{b}}| + |V_{\mathrm{int}}|}}{(|E_{\mathrm{b}}| + |V_{\mathrm{int}}|)!}(\frac{\Delta_n}{b_n}+(|E_{\mathrm{b}}| + |V_{\mathrm{int}}|)\frac{\mathcal{E}_n}{n}),& E_{\mathrm{b}}\neq\emptyset,\\
    N (b_n)^{|E|-|V|+1}\frac{n^{|V|-1}}{(|V|-1)!}(\frac{\Delta_n}{b_n}+|V|\frac{\mathcal{E}_n}{n}),&E_{\mathrm{b}}=\emptyset.
            \end{cases}.
\end{equation}
Noting that for any diagram with boundary edges, $|E_{\mathrm{b}}|+|V_{\mathrm{int}}|=|V_{\mathrm{b}}|+|V_{\mathrm{int}}|=|V|$, summing over $e_k\in E$ yields \eqref{equ:F_asy_equal} and \eqref{equ:F_asy_equal_no_open}.
This completes the proof.
\end{proof}

\section{Asymptotic equivalence of mixed moments}\label{sec:section4}

\subsection{Mixed moments in Gaussian case}

Crucially, when both the number of edges $|E|$ and vertices  $|V|$ increase by one, the upper bounds established in \eqref{equ:3.2} and \eqref{equ:3.3} remain non-decreasing.
To reduce arbitrary diagrams $\Gamma\in \mathscr{D}_{s;\beta}^*$ to trivalent forms—in which every vertex has degree at most 3 while preserving the quantity $|E|-|V|$—we introduce a vertex‑splitting map in the following lemma. This construction requires the notion of typical connected diagrams.

\begin{definition}[Typical   diagram]\label{def:typical_diagram}
A \emph{typical (connected) diagram} is a connected diagram
$\Gamma\in \mathscr{D}_{s;\beta}^*$   satisfying    two conditions: (i) exactly $s$  marked vertices have degree 2, and (ii) every other (unmarked) vertex has degree 3.
Introducing   the parameterization
\begin{equation}
  |E| = 3\ell + s,
  \quad
  |V| = 2\ell + s, \quad \ell=0,1,\ldots,
\end{equation}
we
 denote by $\mathcal{T}_{\ell,s}$ the  set of all typical connected diagrams with $\ell=|E(\Gamma)| - |V(\Gamma)|$ and   $s$ marked points:
  \begin{equation}
   \mathcal{T}_{\ell,s}
  := \Bigl\{
     \Gamma \in \mathscr{D}_{s;\beta}^*: \Gamma \  \text{is a typical diagram},
     |E(\Gamma)|=3\ell+s,    |V(\Gamma)| = 2\ell+s \Bigr\}.  \end{equation}
The set of all typical connected diagrams with $s$ marked points is then defined as
\begin{equation}
    \mathcal{T}_{s} :=  \mathcal{T}_{-1,1} \bigcup\Big(\bigcup_{\ell \ge 0} \mathcal{T}_{\ell,s}\Big),
\end{equation}
where $\mathcal{T}_{-1,1}$  only consists of a single-vertex graph  as a graph containing exactly one vertex and no edges.
\end{definition}

\begin{lemma}[Vertex splitting, {\cite[Lemma 4.2]{liu2025edge}}]\label{lem:vertex_split}
There exists a map from the set of diagrams   with all unmarked vertices of degree at least $3$ to the set of typical connected diagrams
\begin{equation}
    \phi : \mathscr{D}_{s;\beta}^* \to \mathcal{T}_{\ell, s}, \quad \Gamma \mapsto \phi(\Gamma)
\end{equation}
  such that  the following three conditions hold:
\begin{itemize}
    \item[(i)] (Preserving $|E|-|V|$)
    \begin{equation}
        |E(\phi(\Gamma))| - |V(\phi(\Gamma))| = |E(\Gamma)| - |V(\Gamma)|,
    \end{equation}

    \item[(ii)] (Monotonic decreasing)
    \begin{equation}
        |E(\phi(\Gamma))| \ge |E(\Gamma)|, \qquad |V(\phi(\Gamma))| \ge |V(\Gamma)|,
    \end{equation}

    \item[(iii)] (Multiplicity controlling)
       \begin{equation}
        |\phi^{-1}(\Gamma)| \le 2^{|E(\Gamma)|},\quad \forall \,\Gamma \in \mathcal{T}_{\ell, s}.
    \end{equation}

\end{itemize}
\end{lemma}

We require an upper bound on the number of typical connected diagrams.

\begin{lemma}[{\cite[Proposition II.2.3]{feldheim2010universality}; \cite[Proposition 3.7]{geng2024outliers}}] \label{diagramno}
There exists a universal constant $C > 0$ such that for any  integers $\ell \ge 0$ and $s \ge 1$,
\begin{equation}
  \left| \mathcal{T}_{\ell,s} \right| \le \frac{(C(\ell+1))^{\ell+1}}{(s-1)!}.
\end{equation}
\end{lemma}

\begin{theorem}\label{prop:U_asy}
For the deformed Gaussian matrix $X$ defined in Definition~\ref{def:inhomo} under Assumptions~\ref{itm:B1}--\ref{itm:B2} of Definition~\ref{mixingdef}, with $n = \sum_{j=1}^s n_j$ and $a=\|A_N\|_{\mathrm{op}}$, the following statements hold:
\begin{enumerate}[label=(\roman*)]
\item \label{item:U_upper} For all integers $n_j \ge 1$, there exists a constant $C > 0$, independent of $n$ and $N$, such that
\begin{equation} \label{eq:Chebyshev_trace_bound}
\mathbb{E}\Big[\prod_{j=1}^s \operatorname{Tr} \Big(U_{n_j}\big(\frac{X}{2}\big)\Big)\Big]
\le C(1 + a^n)\big(rn+Nb_n\big)^s \exp\big\{C(r+1)^2 n^2b_n \big\}.
\end{equation}
  \item \label{item:U_comparison} Let $r$ be fixed. If $a \le 1 + Cn^{-1}$ for some fixed constant $C>0$, $n^2b_n=O(1)$,  $ {\Delta_n}/{b_n}=o(1)$, and $ {\mathcal{E}_n}/{n}=o(1)$ as $N\to \infty$
  , then
  \begin{equation}
   \mathbb{E}\Big[\prod_{j=1}^s \tr \big(U_{n_j}\big(\frac{X}{2}\big)\big)\Big]
    = \mathbb{E}\Big[\prod_{j=1}^s \tr \big( U_{n_j}\big(\frac{\widetilde{X}}{2}\big)\big)\Big]
    + o\big((Nb_n)^s\big).
  \end{equation}
\end{enumerate}
\end{theorem}

\begin{proof}[Proof of Theorem \ref{prop:U_asy}]

We begin by proving Part \ref{item:U_upper}. In view of Lemma~\ref{thm:u=F}, it suffices
    to bound each diagram function.
To this end, let $\ell = |E| - |V|$ and define the auxiliary  function $G_{\Gamma}(n)$ by
\begin{equation} \label{Gn-1}
G_{\Gamma}(n) = \begin{cases}
    \frac{n^{|V|-1}}{(|V|-1)!}\big(b_n\big)^{|E|-|V|+1} N,\quad& \Gamma\text{ has no open edge},\\
    (1+a^n)r^{|V_{\mathrm{b}}|}\frac{n^{|V|}}{(|V|)!}\big(b_n\big)^{|E_{\mathrm{int}}|-|V_{\mathrm{int}}|},\quad& \Gamma\text{ has at least one   open edge.}
\end{cases}
\end{equation}
  Then Proposition \ref{prop:F_upper_bound}  shows  that
    \begin{equation} \label{Gbound}
        F_{\Gamma}(\{n_j\}_{j=1}^s)\le G_{\Gamma}(n).
    \end{equation}
Since $|V|, |E| \le n$ and the quantity  $|E| - |V|=|E_{\mathrm{int}}|-|V_{\mathrm{int}}|$ is preserved under the map $\phi$,
Lemma~\ref{lem:vertex_split} implies
\begin{equation}
G_{\Gamma}(n) \le G_{\phi(\Gamma)}(n).
\end{equation}
Therefore, using the cumulant expansion from Lemma~\ref{thm:u=F}, we have
\begin{align}
  \kappa_X(n_1,\dots,n_s)
  &\le \sum_{\Gamma \in \mathscr{D}_{s;\beta}^*} G_\Gamma(n)
   \le \sum_{\Gamma \in \mathcal{T}_s} |\phi^{-1}(\Gamma)| \cdot G_\Gamma(n) \notag \\
  &\le \sum_{\Gamma \in \mathcal{T}_s} 2^{|E|} \cdot G_\Gamma(n).
\end{align}

Note that
\begin{equation}\label{equ:EVL}
    |E| = 3\ell + s, \quad |V| = 2\ell + s, \quad |E_{\mathrm{int}}| - |V_{\mathrm{int}}| = |E| - |V| = \ell,
\end{equation}
where we have used  the identity   $|E_{\mathrm{b}}|=|V_{\mathrm{b}}|$; see \cite[{Lemma 2.10}]{geng2024outliers}. Rewriting   $G_\Gamma(n): = G(n,\ell)$, a quantity that depends only on $\ell$ and $s$, we apply Lemma~\ref{diagramno} to further obtain
\begin{equation} \label{kappab-1}
     \kappa_X(n_1,\dots,n_s)
    \le \sum_{\ell \ge 0} \frac{(C(\ell+1))^{\ell+1}}{(s-1)!} G(n,\ell).
\end{equation}

On the other hand, for any fixed $r$,   it follows directly from    \eqref{Gn-1} that
\begin{equation}
    G(n,\ell) \le
\begin{cases}
    \frac{n^{2\ell+s-1}}{(2\ell+s-1)!}\big(b_n\big)^{\ell+1} N,\quad& \Gamma\text{ has no open edge},\\
    (1+a^n)\frac{(rn)^{2\ell+s}}{(2\ell+s)!}\big(b_n\big)^{\ell},\quad& \Gamma\text{ has at least one  open edge,}
\end{cases}
\end{equation}
where in the second case we have used the bound $|V_{\mathrm b}|\le |V|\le 2\ell+s$.
Thus, we   see from  \eqref{kappab-1} that
\begin{multline}\label{equ:4.20}
    \kappa_{X}(n_1,\ldots, n_s)
    \le \frac{(1+a^n)}{(s-1)!}
        \sum_{\ell\ge 0} \big(C(\ell+1)\big)^{\ell+1}\bigg(\frac{n^{2\ell+s-1}}{(2\ell+s-1)!}\big(b_n\big)^{\ell+1} N+ \frac{(rn)^{2\ell+s}}{(2\ell+s)!}\big(b_n\big)^{\ell} \bigg)\\
        =:(1+a^n)(\Sigma_1+\Sigma_2).
\end{multline}

For the first sum $\Sigma_{1}$ in \eqref{equ:4.20},  we apply the elementary bound    $(\ell+1)^{\ell+1} \le C^\ell \ell!$ for $\ell\ge 0$  to obtain
\begin{align}
  \Sigma_{1}
    &\le (Cn)^s \frac{Nb_n}{n}\sum_{\ell=0}^\infty \frac{1}{(\ell)!} \big(Cn^2b_n\big)^\ell
    \le (Cn)^s \frac{Nb_n}{n}\exp\big\{C n^2b_n\big\}.
\end{align}
For the second sum in \eqref{equ:4.20},   we have
    \begin{align}\label{equ:4.22}
        \Sigma_2
        &\le (Crn)^s \sum_{\ell=0}^\infty \frac{1}{\ell!}\big( Cr^2 n^2b_n \big)^\ell
        = (Crn)^{s}\exp\big\{Cr^2 n^2b_n \big\}.
    \end{align}

Observing that  $s$ is a fixed positive integer,  we combine   \eqref{equ:4.20}-\eqref{equ:4.22}  to  obtain    \begin{equation} \label{kappab-4}
     \kappa_{X}(n_1,\ldots, n_s)
    \le (1 + a^n)(C n)^s\Big(r^s+\frac{1}{n}Nb_n\Big) \exp\big\{C(r+1)^2 n^2b_n \big\}.
  \end{equation}
  By \eqref{equ:T} of Definition \ref{Cumu},  we  have
\begin{equation} \label{cumprod}
    \mathbb{E}\Big[\tr\big( U_{n_1}\big(\frac{X}{2}\big)\big)\tr \big(U_{n_2}\big(\frac{X}{2}\big)\big)\cdots \tr \big(U_{n_s}(\frac{X}{2})\big)\Big]=\sum_{\Pi}\prod_{P\in \Pi}\kappa_X\big(\{n_j\}_{j\in P}\big)
\end{equation}
where the sum runs over all partitions $\Pi$  of
 $[s]$,  including the trivial one. Noting that  Assumption \ref{itm:B1} of Definition~\ref{mixingdef} (cf. Remark \ref{remarkB1}) implies  $b_n\ge n/N$, every   product on the right-hand side of \eqref{cumprod} can be    bounded  by
\begin{equation}
    (1 + a^n)(Crn+CNb_n)^s \exp\big\{C(r+1)^2 n^2b_n \big\}.
\end{equation}
This completes the proof of Part \ref{item:U_upper}.

Next, we turn to Part \ref{item:U_comparison}. From Remark \ref{remarkB1}, we know $Nb_n\ge n$. Under the assumptions $a \le 1 + O(n^{-1})$ and $n^2b_n = O(1)$, the bounds in \eqref{kappab-4} and \eqref{Gbound} imply that the series
\begin{equation}
    \sum_{\Gamma \in \mathscr{D}_{s;\beta}^*} \frac{1}{(Nb_n)^s} F_{\Gamma}(\{n_j\}_{j=1}^s)
\end{equation}
is absolutely convergent.
 Consequently, for any sufficiently large constant $K > 0$, the contribution from those diagram functions $F_{\Gamma}$  such that   $|E(\Gamma)|\lor |V(\Gamma)| > K$ is negligible.

Since the number of diagrams with $|E|\leq K$ and $|V| \leq K$ is finite, it suffices to analyze the limit for each individual diagram. By Lemma \ref{lem:vertex_split}, the error term in Theorem \ref{thm:F=tildeF} is bounded by the typical case (after map $\phi$, $|E|-|V|$ remains unchanged and $|V|$ increases). For typical diagrams, by Theorem~\ref{thm:F=tildeF}, we obtain
\begin{equation}
    \frac{F_{\Gamma}(\{n_j\}_{j=1}^s)}{(Nb_n)^s}
    = \frac{\widetilde{F}_{\Gamma}(\{n_j\}_{j=1}^s)}{(Nb_n)^s} +O_{|E|}\Big(\frac{\Delta_n}{b_n}+\frac{\mathcal{E}_n}{n}\Big)
    = \frac{\widetilde{F}_{\Gamma}(\{n_j\}_{j=1}^s)}{(Nb_n)^s}  + o(1).
\end{equation}
Hence, under the assumptions of Part \ref{item:U_comparison} we arrive at
\begin{equation}
    \sum_{\Gamma \in \mathscr{D}_{s;\beta}^*} \frac{1}{(Nb_n)^s} F_{\Gamma}(\{n_j\}_{j=1}^s)
    = \sum_{\Gamma \in \mathscr{D}_{s;\beta}^*} \frac{1}{(Nb_n)^s} \widetilde{F}_{\Gamma}(\{n_j\}_{j=1}^s) + o(1),
\end{equation}
which completes the proof of Part \ref{item:U_comparison}.
\end{proof}

\subsection{Proof of Theorem \ref{prop:super_4}}
\begin{proof}[Proof of Theorem \ref{prop:super_4}]
For the Gaussian case, the proof of Theorem \ref{prop:super_4} follows a strategy analogous to that employed in \cite[Theorem 4.5]{liu2025edge}. To ensure the paper remains self-contained, we provide an outline of the primary arguments and the necessary adjustments.

Indeed, we utilize the linearization property of Chebyshev polynomials to derive the following expansion:
\begin{equation} \label{prodexpansion}
    \prod_{j=1}^t U_{m_j}(x) = \sum_{k \geq 0} c(\{m_j\}; k) U_k(x).
\end{equation}
By applying this expansion, the joint expectation of the product of traces in Theorem \ref{prop:super_4}, namely $\mathbb{E}\left[\prod_{i=1}^s \operatorname{Tr} \left(\prod_{j=1}^{k_i} U_{n_{i,j}}\left(\frac{X}{2}\right)\right)\right]$, can be reduced to the analysis of $\mathbb{E}\left[\prod_{j=1}^s \operatorname{Tr} \left(U_{n_j}\left(\frac{X}{2}\right)\right)\right]$, which is established in Theorem \ref{prop:U_asy}. Under our scaling assumptions, the resulting error terms are shown to be negligible, thereby completing the proof for the Gaussian case.

For the extension to the sub-Gaussian case, we adapt the moment comparison framework developed in \cite[Section 5]{liu2025edge} with a few minor modifications to account for the Short-to-Long Comparison. The principal refinement involves the estimation in \cite[Lemma 5.6, (5.22)]{liu2025edge}: specifically, the bound on the return probability $p_2(x,x)$, previously given by $ \gamma t_N/N$, is replaced by the $b_2$ ($b_n$ with $n=2$) in the setting of this  paper. Consequently, the restriction in \cite[Lemma 5.6, (5.25)]{liu2025edge} is updated to
\begin{equation}\label{equ:3.28}
    \theta n^2 b_2 \ll 1.
\end{equation}
Furthermore, in the step 2 within \cite[Lemma 5.6]{liu2025edge}, every instance of the term $ \big((\gamma\,t_N)\lor n\big)/{N}$ is substituted with $b_n$, and the factor ${\theta\,\gamma\,t_N}/{N}$ is replaced by $\theta b_2$. With these parameter substitutions, \cite[Eqn(5.30)]{liu2025edge} (the non-deformed case) becomes
\begin{equation}
    \begin{aligned}
    &\frac{n^{\,\lvert V(\overline\Gamma)\rvert-1}}
     {(\lvert V(\overline\Gamma)\rvert-1)!}
b_n^{\,\lvert E(\overline\Gamma)\rvert-\lvert V(\overline\Gamma)\rvert+1}
(\theta b_2)^{l}\le b_n^l (\theta b_2)^l n^{t-l}
\frac{n^{\,\lvert V(\Gamma)\rvert-1}}{(\lvert V(\Gamma)\rvert-1)!}
b_n^{\,\lvert E(\Gamma)\rvert-\lvert V(\Gamma)\rvert+1}.
    \end{aligned}
\end{equation}
Following (5.31)-(5.36) in \cite{liu2025edge}, we can obtain
\begin{equation}\label{equ:5.29}
\begin{aligned}
    &\sum_{t=1}^n\sum_{\substack{h_i\ge0\\\sum\limits_{i\ge2} i\,h_i=t}}
    b_n^l (\theta b_2)^l n^{t-l}\,2^t\frac{\bigl(t+\lvert E(\Gamma)\rvert\bigr)^{\,\lvert E(\Gamma)\rvert-1}}
     {(\lvert E(\Gamma)\rvert-1)!}\frac{t!}{\prod_{i=2}^\infty h_i!(i!)^{h_i}}\\
     &\le \sum_{t=1}^n
n^{t}2^t
\frac{\bigl(t+\lvert E(\Gamma)\rvert\bigr)^{\,\lvert E(\Gamma)\rvert-1}}
     {(\lvert E(\Gamma)\rvert-1)!}t
3^t\Bigl((\theta b_2 b_n/n)^{t}+(\theta b_2 b_n)^{t/2}\Bigr)\\
&\le \sum_{t=1}^n
C^t e^{t+|E(\Gamma)|}\Bigl((\theta b_2 b_n)^{t}+(n^2\theta b_2 b_n)^{t/2}\Bigr),
\end{aligned}
\end{equation}
which is $o(1)$, provided that $\theta b_2 b_n +n^2\theta b_2 b_n \ll 1.$
 However, this restriction can be deduced straightforward from \eqref{equ:3.28} and \eqref{UPB}. Moreover, we take
\begin{equation}
    b_2=\max_{x,y\in [N]} \{p_1(x,y)+p_2(x,y)\}\le 2\max_{x,y\in [N]}\sigma_{xy}^2.
\end{equation}
Thus \eqref{equ:1.11} implies \eqref{equ:3.28}.  This completes  the proof.
\end{proof}

\section{
Random band matrices
}\label{sectionrbm}

\subsection{Model and limit processes}
We first analyze random band matrices whose variance profile is given by the density of an $\alpha$-stable law, and then use Theorem~\ref{prop:super_4} to extend the results to general profile functions. The $\alpha$-stable case serves as our {role  model}, owing to  its intimate  connection with central limit theorems.
\begin{definition}[Stable theta functions]\label{def:alpha_profile}
   Let  $f_{\alpha}(x)$ denote  the  density of $d$-dimensional   $\alpha$-stable distribution    whose characteristic function is given  by
    \begin{equation}
     \phi(t)=\begin{cases}  e^{-\frac{\sigma^2 }{2}\|t\|^2}, &\alpha=2,\\
     e^{-c_{\alpha}\|t\|^\alpha}, &0<\alpha<2, \end{cases}
 \end{equation}
 for $t\in \mathbb{R}^d$, where $c_{\alpha}>0$ and  $\sigma>0$ are fixed. With   the scaling notation
 \begin{equation}\label{equ:homo_func}
     f_{\alpha}(x, \tau)=\tau^{-\frac{d}{\alpha}}f_\alpha( \tau^{-\frac{1}{\alpha}} {x}),
 \end{equation}
the corresponding $\alpha$-stable theta function is
\begin{equation} \label{theta}
    \theta_\alpha (x,\tau)=\sum_{k\in \mathbb{Z}^d} f_{\alpha}(x+k, \tau).
\end{equation}
\end{definition}

 \begin{definition}[Periodic random band matrices] \label{defmodel} Let $\Lambda_L = (\mathbb{Z}/L\mathbb{Z})^d$  the $d$-dimensional  discrete torus,  and let  $f:\mathbb{R}^d\longrightarrow\mathbb{R}$ be a non‑negative integrable function.
 A symmetric/Hermitian matrix $H=(H_{xy})_{x,y\in \Lambda_{L}}$ is called a periodic random band matrix with bandwidth $W\leq L/2$ and profile  $f$, if its   variance profile  as in Definition \ref{def:inhomo}    is given by
  \begin{equation} \label{VP}
     \sigma_{xy}^2=\frac{1}{M}\sum_{k\in \mathbb{Z}^d}f\big(\frac{x-y+kL}{W}\big), \quad  M:=\sum_{k\in \mathbb{Z}^d} f\big(\frac{k}{W}\big).
 \end{equation}
   In the special case $f = f_{\alpha}(x)$ (as defined above), the model is termed an $\alpha$-stable random band matrix.
 \end{definition}

We will employ the local limit theorem for random walks on the discrete torus associated with the $\alpha$-stable profile.
\begin{theorem}[Local limit theorem on $\mathbb{T}^d$]\label{prop:alpha_llt}
Let $p_n(x,y)$ denote the $n$-step transition probability associated with the transition matrix $P_N$ from Definition~\ref{defmodel}, with profile $f_\alpha(x)$, where $N=L^d$.
  If $ne^{-c_{\alpha}W^\alpha}=o(1)$ as $N\to \infty$, then
    \begin{equation}
        {p}_{n}(0, {x})=\frac{1}{N}\theta_{\alpha}\Big(\frac{x}{L},n\big(\frac{W}{L}\big)^\alpha\Big)+O\big(ne^{-c_{\alpha}W^{\alpha}}\big).
    \end{equation}
\end{theorem}
\begin{proof}
Due to the translation invariance of the transition matrix $P_N$, it suffices to consider $p_n(0,x)$.

We start from the discrete Fourier representation of the $n$-step transition probability $p_n(0,x)$ on the $d$-dimensional torus $\Lambda_L = (\mathbb{Z}\cap (-\frac{L}{2}, \frac{L}{2}])^d$
, where $N = L^d$. Define the discrete Fourier transform of the one step transition probability as
\begin{equation}
    \widehat{p}(k) := \sum_{x\in \Lambda_L} p_{1}(0,x) \, e^{-2\pi i \frac{k\cdot x}{L}}, \qquad k\in\Lambda_L.
\end{equation}
Then the Fourier inversion gives
\begin{equation}\label{eq:p_n_fourier}
    p_n(0,x) = \frac{1}{N} \sum_{k\in\Lambda_L} \bigl(\widehat{p}(k)\bigr)^n e^{2\pi i \frac{k\cdot x}{L}}.
\end{equation}

Using the Poisson summation formula, we obtain
\begin{equation}\label{eq:p_hat_approx}
\begin{aligned}
    \widehat{p}(k) &= \sum_{x\in\Lambda_L} \frac{1}{W^d} \sum_{y\in\mathbb{Z}^d} f\!\Big(\frac{x+yL}{W}\Big) e^{-2\pi i \frac{k\cdot x}{L}} \\
    &= \frac{1}{W^d} \sum_{x\in\mathbb{Z}^d} f\!\big(\frac{x}{W}\big) e^{-2\pi i \frac{k\cdot x}{L}} = \frac{W^d}{M}\sum_{m\in\mathbb{Z}^d} \widehat{f}\!\Big(W\big(\frac{k}{L}+m\big)\Big).
\end{aligned}
\end{equation}

For the $\alpha$-stable profile, $\widehat{f}(\xi) = e^{-c_\alpha \|\xi\|^\alpha}$. Its exponential decay yields
\begin{equation}
    1 = \widehat{p}(0) = \frac{W^d}{M}\sum_{m\in\mathbb{Z}^d} \widehat{f}(mW) =\frac{W^d}{M}\big( 1 + O(e^{-c_\alpha W^\alpha})\big),
\end{equation}
and consequently
\begin{equation}
    \widehat{p}(k) = \widehat{f}\!\big(\frac{Wk}{L}\big) + O(e^{-c_\alpha W^\alpha}).
\end{equation}
Since $n e^{-c_\alpha W^\alpha} \ll 1$, using the binomial formula and the bound $|\widehat{f}(\xi)|\leq 1$, we may expand the $n$-th power:
\begin{equation}
\widehat{p}(k)^n = \exp\!\Big(-c_\alpha n \Bigl\|\frac{Wk}{L}\Bigr\|^\alpha\Big) + O\bigl(n e^{-c_\alpha W^\alpha}\bigr).
\end{equation}
Insert this approximation into the Fourier sum \eqref{eq:p_n_fourier} and we derive
\begin{equation}\label{eq:p_n_sum}
\begin{aligned}
    p_n(0,x) &= \frac{1}{N} \sum_{k\in\Lambda_L} \left\{ \exp\!\left(-c_\alpha n \Bigl\|\frac{Wk}{L}\Bigr\|^\alpha\right) e^{2\pi i \frac{k\cdot x}{L}} + O\bigl(n e^{-c_\alpha W^\alpha}\bigr) \right\} \\
    &= \frac{1}{N} \sum_{k\in\Lambda_L} \exp\!\left(-c_\alpha n \Bigl\|\frac{Wk}{L}\Bigr\|^\alpha\right) e^{2\pi i \frac{k\cdot x}{L}} + O\bigl(n e^{-c_\alpha W^\alpha}\bigr) \\
    &= \frac{1}{N} \sum_{k\in\mathbb{Z}^d} \exp\!\left(-c_\alpha n \Bigl\|\frac{Wk}{L}\Bigr\|^\alpha\right) e^{2\pi i \frac{k\cdot x}{L}} + O\bigl(e^{-c_\alpha n W^\alpha} + n e^{-c_\alpha W^\alpha}\bigr).
\end{aligned}
\end{equation}

Now apply the Poisson summation formula again to the principal sum:
\begin{equation}
    \frac{1}{N} \sum_{k\in\mathbb{Z}^d} \exp\!\left(-c_\alpha n \Bigl\|\frac{Wk}{L}\Bigr\|^\alpha\right) e^{2\pi i \frac{k\cdot x}{L}}
    = \frac{1}{N} \sum_{m\in\mathbb{Z}^d} \frac{1}{\bigl(n \big(\frac{W}{L}\big)^\alpha\bigr)^{\frac{d}{\alpha}}} \; f_\alpha\!\bigg(\frac{\frac{x}{L} + m}{ \bigl(n \big(\frac{W}{L}\big)^\alpha\bigr)^{\frac{1}{\alpha}}}\bigg).
\end{equation}
Here the right-hand side is exactly the $\alpha$-stable $\theta_\alpha$ function defined in \eqref{theta} with scaling parameter $\tau = n(\frac{W}{L})^\alpha$ and $N=L^d$.
Substituting back into \eqref{eq:p_n_sum}, we obtain the final asymptotic:
\begin{equation}
   p_n(0,x) = \frac{1}{N}\,\theta_\alpha\!\Big(\frac{x}{L},\, n \bigl(\tfrac{W}{L}\bigr)^\alpha\Big) + O\bigl(n e^{-c_\alpha W^\alpha}\bigr).
\end{equation}

Thus the proof is complete.
\end{proof}

The following lemma provides the asymptotics of the $\theta_{\alpha}$ function. The proof of this lemma   is postponed to   Appendix~\ref{sec:theta_function}.
\begin{lemma}\label{lem:asy_theta}
The $\alpha$-stable theta function $\theta_{\alpha}(x, \tau)$ defined in \eqref{theta} satisfies the following asymptotic properties for $x \in \mathbb{T}^d$ and $\tau > 0$.
\begin{enumerate}
     \item[(i)] \textbf{(Small $\tau$ regime)} As $\tau \to 0$,
    \begin{equation}
        \theta_{\alpha}(x, \tau) = f_{\alpha}(x, \tau) + O(\tau).
    \end{equation}
    \item[(ii)] \textbf{(Large $\tau$ regime)} As $\tau \to \infty$,
    \begin{equation}
        \theta_{\alpha}(x, \tau) = 1 + O\big(e^{-c_{\alpha} (2\pi)^{\alpha} \tau}\big).
    \end{equation}
    \item[(iii)] \textbf{(Uniform upper bound)} There exists a constant $C > 0$ such that
    \begin{equation}\label{equ:theta_upper_bound}
        \theta_{\alpha}(x, \tau) \le C\big(1 + \tau^{-\frac{d}{\alpha}}\big).
    \end{equation}
\end{enumerate}
\end{lemma}

We now introduce the limiting diagram functions, focusing in particular on the subcritical and critical regimes under the condition $\alpha > d = 1$.

\begin{definition}[Limiting diagram function]\label{def:limiting_function}
Given a typical diagram $\Gamma$, for $t_1,\dots,t_s>0$ we introduce a family of linear constraints for edge–weighted variables $\{\alpha_e\geq 0\}$
\begin{equation}
\mathfrak{C}(\{t_i\})
:\quad
\sum_{e\in E(\Gamma)} c_i(e)\,\alpha_e \;\le\; t_i,
\quad i=1,\dots,s,
\end{equation}
where \(c_i(e)\in \{0,1,2\}\)  counts how many times the edge
$e$ appears in the boundary $\partial D_i$.
   Let \(C_\Gamma>0\) be a combinatorial constant depending only on \(\Gamma\) defined by
\begin{equation}\label{equ:C_Gamma}
    C_\Gamma=\lim_{n_j\rightarrow\infty}\frac{\#\{(w_e)_{e\in E(\Gamma)}|
        m_j, w_e\in \mathbb{N},~m_j \ge 0,\; w_e \ge 1,~
        2m_j + \sum\limits_{e \in \partial D_j} w_e = n_j,~ j=1,\dots,s\}}{\mathrm{Vol}\{(\alpha_e)_{e\in E(\Gamma)}|\alpha_e\in \mathbb{R}^+, \sum_{e\in E(\Gamma)} c_j(e)\,\alpha_e \;\le\; n_j,
\quad j=1,\dots,s\}}.
\end{equation}
The three limiting  diagram functions are then defined separately for each phase regime as follows:
\begin{enumerate}
    \item[(i)] \textbf{(Supercritical case)}
    \begin{equation}\label{eq:cor-F-super}
\mathcal{F}^{\mathrm{(Super)}}_\Gamma(t_1,\dots,t_s)
\;=\frac{C_\Gamma}{\,t_1\cdots t_s\,}\mathrm{Vol}_{\mathfrak{C}}(\{t_i\}),
\end{equation}
where
\begin{equation}
    \mathrm{Vol}_{\mathfrak{C}}(\{t_i\})=\int_{\mathfrak{C}(\{t_i\})}1\,\prod_{e\in E(\Gamma)}d\alpha_e.
\end{equation}
\item[(ii)]  \textbf{(Subcritical case, $\alpha>d=1$)}
\begin{equation}\label{eq:F-sub}
    \mathcal{F}^{\mathrm{(Sub)}}_\Gamma(t_1,\dots,t_s)
    \;=\;
    \frac{C_\Gamma}{\,t_1\cdots t_s\,}\,
    \int_{x_v\in \mathbb{R}^d, v\ne v_0}\int_{\mathfrak{C}(\{t_i\})}
    \prod_{e\in E(\Gamma)}f_{\alpha}(x_{u}-x_{v},\alpha_e ) \prod_{e\in E(\Gamma)}d\alpha_e \prod_{v_0\neq v\in V(\Gamma)} dx_v,
\end{equation} where
the first integral is over all $x_v$ in $\mathbb{R}^d$ except for  $x_{v_0}=0$, with   some $v_0\in V(\Gamma)$.
\item[(iii)] \textbf{(Critical case, $\alpha>d=1$)}
    \begin{equation}\label{eq:F-critical}
    \mathcal{F}^{\mathrm{(Crit)}}_\Gamma(t_1,\dots,t_s;\tau)
    \;=\;
    \frac{C_\Gamma}{\,t_1\cdots t_s\,}\,
    \int_{x_v\in \mathbb{T}^d}\int_{\mathfrak{C}(\{t_i\})}
    \prod_{e\in E(\Gamma)}\theta_{\alpha}(x_{u}-x_{v},\alpha_e\tau) \prod_{e\in E(\Gamma)}d\alpha_e \prod_{v\in V(\Gamma)}dx_v.
\end{equation}
\end{enumerate}
\end{definition}

\subsection{Diagram asymptotics}
\begin{theorem}[Fixed diagram asymptotics]\label{thm:homogeneous_theorem}
For the $\alpha$-stable random band  matrices   in Definition \ref{defmodel}, but with Gaussian entries and  $\alpha>d=1$, assuming that $\sum_{i=1}^s n_i$ is even   and   $t_i\in (0,\infty)$ with   $i=1,\ldots,s$,   for any given {typical} connected diagram $\Gamma\in \mathcal{T}_s$ the following hold:
    \begin{itemize}
        \item[(i)]
({\bf Supercritical case})
    If  $W\gg N^{1-\frac{1}{3\alpha}}$  and all $n_i\sim t_i N^{\frac{1}{3}}$, then
    \begin{equation}\label{equ:f_sup}
        \frac{1}{n_1\cdots n_s}F_\Gamma(\{n_i\}_{i=1}^s)=
        \big(1+o(1)\big)\mathcal{F}^{\mathrm{(Super)}}_\Gamma(t_1,\dots,t_s).
    \end{equation}

 \item[(ii)]({\bf Subcritical case}) If $(\log N)^{\frac{1}{\alpha}}\ll W\ll N^{1-\frac{1}{3\alpha}}$ and all $n_i\sim t_i W^{\frac{\alpha}{3\alpha-1}}$, then
    \begin{equation}\label{equ:f_sub}
    \begin{aligned}
\frac{1}{n_1\cdots n_s} F_\Gamma(\{n_i\}_{i=1}^s)
 &=   \big(1+o(1)\big)NW^{-\frac{3\alpha}{3\alpha-1}}\mathcal{F}^{\mathrm{(Sub)}}_\Gamma(t_1,\dots,t_s).
\end{aligned}
    \end{equation}

\item[(iii)] ({\bf Critical case})
       If $W\sim   (\gamma N)^{1-\frac{1}{3\alpha}}$ for some $\gamma>0$ and all  $n_i\sim t_i W^{\frac{\alpha}{3\alpha-1}}$, then
    \begin{equation}\label{equ:f_critical}
        \frac{1}{n_1\cdots n_s}F_\Gamma(\{n_i\}_{i=1}^k)=   \big(1+o(1)\big)\gamma^{|E|-|V|}\mathcal{F}^{\mathrm{(Crit)}}_\Gamma(t_1,\dots,t_s;\gamma^\alpha).
    \end{equation}
 \end{itemize}
\end{theorem}
\begin{proof}
First, we rewrite the diagram function from Definition~\ref{defsDF}, given in \eqref{equ:F_formula}, in a simpler form:

\begin{equation}
F_{\Gamma}(\{n_j\})=
\sum_{\eta: V(\Gamma) \to [N]}
{\sum_{
  \sum\limits_{e \in \partial D_j} w_e \leq  n_j
}}'
\prod_{(x, y) \in E(\Gamma)}
 {p}_{w_e}(\eta(x), \eta(y)),
\end{equation}
where the second summation ${\sum}'$ is taken over all integers   $w_e \ge 1$
  for which each difference
  $n_j-\sum_{e \in \partial D_j} w_e$ is even.

Using   the  inequality, for all $|a_i|\le 1$ and  $|b_i|\le 1$,
\begin{equation}
    \left|\prod_{i=1}^n a_i - \prod_{i=1}^n b_i\right|\le \left|\sum_{j=1}^{n}\Big(\prod_{i=1}^j a_i\prod_{i=j+1}^n b_i-\prod_{i=1}^{j-1} a_i\prod_{i=j}^n b_i\Big)\right|\le \sum_{i=1}^n |a_i-b_i|,
\end{equation}
combining this   with Theorem \ref{prop:alpha_llt}, under the condition  $W\gg (\log N)^{ {1}/{\alpha}}$, we obtain
\begin{equation}
    \left|\prod_{(x, y) \in E(\Gamma)}{p}_{w_e}(\eta(x), \eta(y))-\prod_{(x, y) \in E(\Gamma)}\frac{1}{N}\theta_{\alpha}\Big(\frac{x-y}{L},w_e\big(\frac{W}{L}\big)^\alpha\Big)\right|= O(|E|ne^{-c_{\alpha}W^{\alpha}}).
\end{equation}
Thus  we may replace $p_{w_e}(x,y)$ by the corresponding values of the $\theta_{\alpha}$ function and
  consequently obtain
\begin{equation}
    F_{\Gamma}(\{n_j\})=\sum_{\eta: V(\Gamma) \to [N]}
{\sum_{
  \sum\limits_{e \in \partial D_j} w_e \leq  n_j
}}'
\prod_{(x, y) \in E(\Gamma)}\frac{1}{N}\theta_{\alpha}\Big(\frac{x-y}{L},w_e\big(\frac{W}{L}\big)^\alpha\Big)+O\big(\sum_{\eta: V(\Gamma) \to [N]}
\sum_{
  \sum\limits_{e \in \partial D_j} w_e \leq  n_j
}|E|ne^{-c_{\alpha}W^{\alpha}}\big).
\end{equation}
For the second summation,  since
\begin{equation}
    \sum_{\eta: V(\Gamma) \to [N]}
{\sum_{
  \sum\limits_{e \in \partial D_j} w_e \leq  n_j
}}'|E|ne^{-c_{\alpha}W^{\alpha}}\le n|E|N^{|V|}e^{-c_{\alpha}W^{\alpha}}\sum_{ w_e \leq  n}1\le |E|n^{|E|+1}N^{|V|}e^{-c_{\alpha}W^{\alpha}},
\end{equation}
 this simplifies to
\begin{equation}
    F_{\Gamma}(\{n_j\})=\sum_{\eta: V(\Gamma) \to [N]}
{\sum_{
  \sum\limits_{e \in \partial D_j} w_e \leq  n_j
}}'
\prod_{(x, y) \in E(\Gamma)}\frac{1}{N}\theta_{\alpha}\Big(\frac{x-y}{L},w_e\big(\frac{W}{L}\big)^\alpha\Big)+ O\big(|E|n^{|E|+1}N^{|V|}e^{-c_{\alpha}W^{\alpha}}\big).
\end{equation}
For any fixed diagram, the error term vanishes as $N\rightarrow \infty$, provided that $n=N^{O(1)}$ and $W\gg (\log N)^{\frac{1}{\alpha}}$.

Using the properties of the theta function from Lemma \ref{lem:asy_theta}, we obtain the following asymptotic estimates:
\begin{enumerate}
    \item If $w_e(\frac{W}{L})^\alpha\gg 1$, then
    \begin{equation}
        \frac{1}{N}\theta_{\alpha}\Big(\frac{x}{L},w_e\big(\frac{W}{L}\big)^\alpha\Big)=\frac{1+o(1)}{N}.
    \end{equation}
    \item If $w_e(\frac{W}{L})^\alpha\ll 1$, setting $w_e=\alpha_e W^{\frac{\alpha}{3\alpha-1}}$,   then
    \begin{equation}
        \frac{1}{N}\theta_{\alpha}\Big(\frac{x}{L},w_e\big(\frac{W}{L}\big)^\alpha\Big)=\frac{(1+o(1))}{N}f_{\alpha}\big(\frac{x}{L},w_e\big(\frac{W}{L}\big)^{\alpha}\big)=\frac{1+o(1)}{W^{\frac{3\alpha}{3\alpha-1}}}f_{\alpha}\big(\frac{x}{W^{\frac{3\alpha}{3\alpha-1}}},\alpha_e\big).
    \end{equation}
    \item If $w_e(\frac{W}{L})^\alpha=\alpha_e W^{\frac{\alpha}{3\alpha-1}}(\frac{W}{L})^{\alpha}\sim\alpha_e\gamma^\alpha$, then
    \begin{equation}
        \frac{1}{N}\theta_{\alpha}\big(\frac{x}{L},w_e\big(\frac{W}{L}\big)^\alpha\big)=\frac{1+o(1)}{N}\theta_{\alpha}\big(\frac{x}{L},\alpha_e\gamma^\alpha\big).
    \end{equation}

\end{enumerate}

 Converting two summations into double integrals gives the following results:
\begin{enumerate}

\item {\bf Supercritical case:} When  $w_e({W}/{L})^\alpha\gg 1$ for all $e\in E(\Gamma)$, we have
        \begin{equation}
        \begin{aligned}
                   &{\sum_{
  \sum\limits_{e \in \partial D_j} w_e \leq  n_j
}}'\sum_{\eta: V(\Gamma) \to [N]}\prod_{e\in E(\Gamma)}\frac{1+o(1)}{N}= (1+o(1)) {\sum_{
  {\sum\limits_{e \in \partial D_j} w_e \leq  n_j
}}}'N^{|V|-|E|}\\
&=(1+o(1))N^{|V|-|E|}    {\sum_{
  \sum\limits_{e \in \partial D_j} w_e \leq  n_j
}}'
1
=(1+o(1))N^{\frac{s}{3}-\frac{1}{3}|E|} C_{\Gamma} \mathrm{Vol}(\{t_iN^{1/3}\})\\
&=(1+o(1))N^{\frac{s}{3}}C_{\Gamma}\mathrm{Vol}(\{t_i\}).
        \end{aligned}
    \end{equation}

    The contribution from configurations in which
 $w_e({W}/{L})^\alpha=O(1)$ for at least one $e\in E(\Gamma)$ is negligible. Dividing by $n_1n_2\ldots n_s$ gives
    \begin{equation}
        \frac{N^{\frac{s}{3}}}{n_1\cdots n_s}=\frac{1+o(1)}{t_1\cdots t_s},
    \end{equation}
    which completes the analysis of the supercritical case.

\item \textbf{Subcritical case:}
The sum can be regarded as a Riemann sum and  approximated by the corresponding integral. Let $w_e=\alpha_eW^{\frac{\alpha}{3\alpha-1}}$, we  have
        \begin{equation}
        \begin{aligned}&F_\Gamma(\{n_i\}_{i=1}^s)={\sum_{
  \sum\limits_{e \in \partial D_j} w_e \leq  n_j
}}'\sum_{\eta: V(\Gamma) \to [N]}\prod_{e\in E(\Gamma)}\frac{1+o(1)}{W^{\frac{3\alpha}{3\alpha-1}}}f_{\alpha}\big(\frac{\eta(x_e)-\eta(y_e)}{W^{\frac{3\alpha}{3\alpha-1}}},\alpha_e\big)\\
&= (1+o(1)) NW^{\frac{3\alpha}{3\alpha-1}(|V|-1-|E|)} {\sum_{
  {\sum\limits_{e \in \partial D_j} w_e \leq  n_j
}}}'\int_{x_v\in \mathbb{R}^d, v\ne v_0}
    \prod_{e\in E(\Gamma)}f_{\alpha}(x_{u}-x_{v},\alpha_e ) \prod_{e\in E(\Gamma)}\prod_{v_0\neq v\in V(\Gamma)} dx_v.
        \end{aligned}
    \end{equation}
    Here $N$ comes from the translation invariance and we can fix $v_0$ to be $0$ with an extra factor $N$. Now we change the sum using $n_i=t_i W^{\frac{\alpha}{3\alpha-1}}$
    \begin{equation}
        {\sum_{
  {\sum\limits_{e \in \partial D_j} w_e \leq  n_j
}}}'(\cdots)\longrightarrow W^{\frac{\alpha}{3\alpha-1}|E|}C_{\Gamma}\int_{\mathfrak{C}(\{t_i\})}(\cdots).
    \end{equation}
Thus
\begin{equation}
\begin{aligned}
    \frac{1}{n_1\cdots n_s} F_\Gamma(\{n_i\}_{i=1}^s)&=NW^{-\frac{3\alpha}{3\alpha-1}}\frac{(1+o(1))C_\Gamma}{\,t_1\cdots t_s\,}\,\\
    &\times
    \int_{x_v\in \mathbb{R}^d, v\ne v_0}\int_{\mathfrak{C}(\{t_i\})}
    \prod_{e\in E(\Gamma)}f_{\alpha}(x_{u}-x_{v},\alpha_e ) \prod_{e\in E(\Gamma)}d\alpha_e \prod_{v_0\neq v\in V(\Gamma)} dx_v,
\end{aligned}
\end{equation}
where in the exponent of $W$ we have used the simple identity  \begin{equation}\frac{3\alpha}{3\alpha-1}(|V|-|E|)+\frac{\alpha}{3\alpha-1}|E|-\frac{s\alpha}{3\alpha-1}=0.\end{equation}
\item \textbf{Critical case:}  The critical case follows a similar argument. Let $w_e=\alpha_eW^{\frac{\alpha}{3\alpha-1}}$, then
 \begin{equation}
        \begin{aligned}F_\Gamma(\{n_i\}_{i=1}^s)&={\sum_{
  \sum\limits_{e \in \partial D_j} w_e \leq  n_j
}}'\sum_{\eta: V(\Gamma) \to [N]}\prod_{e\in E(\Gamma)}\frac{1+o(1)}{N}\theta_{\alpha}\big(\frac{\eta(x_e)-\eta(y_e)}{L},\alpha_e\gamma^\alpha\big)\\
&= (1+o(1))N^{-|E|+|V|}{\sum_{
  {\sum\limits_{e \in \partial D_j} w_e \leq  n_j
}}}'\int_{x_v\in \mathbb{T}^d}
    \prod_{e\in E(\Gamma)}\theta_{\alpha}(x_{u}-x_{v},\alpha_e\gamma^\alpha) \prod_{v\in V(\Gamma)}dx_v.
        \end{aligned}
    \end{equation}
Now we change the sum using $n_i=t_i W^{\frac{\alpha}{3\alpha-1}}$
    \begin{equation}
        {\sum_{
  {\sum\limits_{e \in \partial D_j} w_e \leq  n_j
}}}'(\cdots)\longrightarrow W^{\frac{\alpha}{3\alpha-1}|E|}C_{\Gamma}\int_{\mathfrak{C}(\{t_i\})}(\cdots),
    \end{equation}
and thus obtain
\begin{equation}
\begin{aligned}
     \frac{1}{n_1\cdots n_s} F_\Gamma(\{n_i\}_{i=1}^s)&=W^{\frac{\alpha}{3\alpha-1}|E|-s\frac{\alpha}{3\alpha-1}}N^{-|E|+|V|} \frac{(1+o(1))C_\Gamma}{\,t_1\cdots t_s\,}\\
    &\times \,
    \int_{x_v\in \mathbb{T}^d}\int_{\mathfrak{C}(\{t_i\})}
    \prod_{e\in E(\Gamma)}\theta_{\alpha}(x_{u}-x_{v},\alpha_e\gamma^\alpha) \prod_{e\in E(\Gamma)}d\alpha_e \prod_{v\in V(\Gamma)} dx_v.
\end{aligned}
\end{equation}
This case immediately follows from
\begin{equation}
    W^{\frac{\alpha}{3\alpha-1}|E|-s\frac{\alpha}{3\alpha-1}}N^{-|E|+|V|}=(\gamma N)^{|E|-|V|}N^{-|E|+|V|}=\gamma^{|E|-|V|}.
\end{equation}
\end{enumerate}

   We   thus complete  the proof.
\end{proof}

The limiting diagram functions admit the following uniform upper bounds.
\begin{lemma}\label{lem:limit_function_upper_bound} Let $\Gamma$ be a connected typical diagram. Let   $t=\sum_{i=1}^s t_i$,
   the following upper bounds for limiting diagram functions  hold  for some constant $C>0$.
    \begin{enumerate}
    \item[(i)] \textbf{(Supercritical case)}
    \begin{equation}\label{eq:cor-F-super-upper}
\mathrm{Vol}_{\mathfrak{C}}(\{t_i\})\le \frac{(Ct)^{|E|}}{(|E|)!}.
\end{equation}
\item[(ii)]  \textbf{(Subcritical case, $\alpha>d=1$)}  For    a chosen vertex $v_0$ and   $x_{v_0}=0$,
\begin{equation}\label{eq:F-sub-upper}
    \int_{x_v\in \mathbb{R}^d, v\ne v_0}\int_{\mathfrak{C}(\{t_i\})}
    \prod_{e\in E(\Gamma)}f_{\alpha}(x_{e}-y_{e},\alpha_e ) \prod_{e\in E(\Gamma)}d\alpha_e \prod_{v\in V(\Gamma),v\neq v_0}dx_v\le \frac{(Ct)^{|E|-\frac{1}{\alpha}(|E|-|V|+1)}}{(|E|-\frac{1}{\alpha}(|E|-|V|+1))!}.
\end{equation}
Here, $x_e$ and $y_e$ denote the variables $x_u$ and $x_v$ associated with the endpoints of the edge $e=\{u,v\}$.
\item[(iii)] \textbf{(Critical case, $\alpha>d=1$)}
   For some constant  $C_{\tau}$ depending on  $\tau>0$, \begin{equation}\label{eq:F-critical-upper}
    \begin{aligned}
    &\int_{x_v\in \mathbb{T}^d}\int_{\mathfrak{C}(\{t_i\})}
    \prod_{e\in E(\Gamma)}\theta_{\alpha}(x_{e}-y_{e},\alpha_e \tau) \prod_{e\in E(\Gamma)}d\alpha_e \prod_{v\in V(\Gamma)}dx_v\le \frac{(Ct)^{|E|}}{(|E|)!}+\frac{(C_{\tau}t)^{|E|-\frac{1}{\alpha}(|E|-|V|+1)}}{(|E|-\frac{1}{\alpha}(|E|-|V|+1))!}.
    \end{aligned}
\end{equation}
Here, $x_e$ and $y_e$ denote the variables $x_u$ and $x_v$ associated with the endpoints of the edge $e=\{u,v\}$.
\end{enumerate}
\end{lemma}
\begin{proof}
 Let $T$ be a spanning tree of $\Gamma$, so $|E(T)|=|V|-1$. Choose the spatial dimension   $d=1$ in both the subcritical and critical cases.

\begin{enumerate}
 \item \textbf{(Super-critical case)}
 Since the constraint region $\mathfrak{C}(\{t_i\})$ is defined by $\sum_{e\in E(\Gamma)} c_i(e)\,\alpha_e \;\le\; t_i$ with $c_i(e) \ge 0$, we may relax the constraint to an $|E|$-dimensional simplex:
    \begin{equation}
    \sum_{e\in E(\Gamma)} \alpha_e \le C t,
    \end{equation}
    where $C$ is a constant that
     absorbs the coefficients. The volume of this simplex  is bounded by the standard formula:
    \begin{equation}
    \mathrm{Vol}_{\mathfrak{C}}(\{t_i\}) \le \int_{\alpha_e\ge 0, \sum \alpha_e \le Ct} 1\,\prod_{e\in E(\Gamma)}d\alpha_e = \frac{(Ct)^{|E|}}{|E|!}.
    \end{equation}

 \item \textbf{(Subcritical case)}
Let $I_{\mathrm{Sub}}$ denote  the integral  on the left hand side. We relax  the  constraints on $\alpha_e$ to $\sum_{e} \alpha_e \le C t$. Next, we estimate the integral over the vertex positions    $x_v \in \mathbb{R}^d$. Fixing one vertex at the origin, say   $x_0=0$, the integral over the remaining $|V|-1$ degrees of freedom can be bounded by the convolution properties of $f_{\alpha}$  together with    the spanning
   tree $T$:
    \begin{equation}
    \begin{aligned}
&\int_{x_v\in \mathbb{R}^d,x_{v_0}=0}  \prod_{e\in E(\Gamma)}f_{\alpha}(x_{e}-y_{e},\alpha_e ) \prod_{v\in V(\Gamma)}dx_v \\
 &\le C_{1}^{|E|-|V|+1} \prod_{e \notin T} \alpha_e^{-\frac{d}{\alpha}}\int_{x_v\in \mathbb{R}^d,x_{v_0}=0}  \prod_{e\in E(\Gamma)\setminus T}f_{\alpha}(x_{e}-y_{e},\alpha_e ) \prod_{v\in V(\Gamma)}dx_v\\
 &= C_{1}^{|E|-|V|+1} \prod_{e \notin T} \alpha_e^{-\frac{d}{\alpha}}.
    \end{aligned}
    \end{equation}

   When    $d=1$,
    substituting this bound into the integral over
$\alpha_e$ (a generalized Dirichlet integral over a simplex) gives
    \begin{equation}
    I_{\mathrm{Sub}} \le C_{1}^{|E|-|V|+1} \int_{\alpha_e \ge 0, \sum_e \alpha_e \le Ct} 1\prod_{e \in T} d\alpha_e \prod_{e \notin T} (\alpha_e^{-\frac{1}{\alpha}} d\alpha_e).
    \end{equation}
    Thus we obtain
    \begin{equation}
    I_{\mathrm{Sub}} \le  \frac{C_{1}^{|E|-|V|+1}(Ct)^{|E|-\frac{1}{\alpha}(|E|-|V|+1)}}{\Gamma(|E|-\frac{1}{\alpha}(|E|-|V|+1)+1)} \leq \frac{(C't)^{|E|-\frac{1}{\alpha}(|E|-|V|+1)}}{(|E|-\frac{1}{\alpha}(|E|-|V|+1))!}.
    \end{equation}

 \item \textbf{(Critical case)} Let $I_{\mathrm{Crit}}$ denote the integral on the left hand side. We relax the constraints on $\alpha_e$ to $\sum_{e} \alpha_e \le C t$ again.  By \eqref{equ:theta_upper_bound} in Lemma \ref{lem:asy_theta}, we have

    \begin{equation}
        \theta_{\alpha}(x_{e}-y_{e},\alpha_e \tau)\le C(1+(\alpha_e\tau)^{-\frac{d}{\alpha}}).
    \end{equation}
    By translation invariance, we obtain
    \begin{equation}
    \begin{aligned}
&\int_{x_v\in \mathbb{T}^d}  \prod_{e\in E(\Gamma)}\theta_{\alpha}(x_{e}-y_{e},\alpha_e ) \prod_{v\in V(\Gamma)}dx_v\\
&=\int_{x_v\in \mathbb{T}^d,v_0=0}  \prod_{e\in E(\Gamma)}\theta_{\alpha}(x_{e}-y_{e},\alpha_e ) \prod_{v\in V(\Gamma)}dx_v \\
 &\le C_{1}^{|E|-|V|+1} \prod_{e \notin T} (1+(\alpha_e\tau)^{-\frac{d}{\alpha}})\int_{x_v\in \mathbb{T}^d, x_{v_0}=0}  \prod_{e\in E(\Gamma)\setminus T}\theta_{\alpha}(x_{e}-y_{e},\alpha_e ) \prod_{v\in V(\Gamma)}dx_v\\
 &= C_{1}^{|E|-|V|+1} \prod_{e \notin T} (1+(\alpha_e\tau)^{-\frac{d}{\alpha}}).
    \end{aligned}
    \end{equation}
When $d=1$, substituting this bound into the integral over
$\alpha_e$ gives
    \begin{equation}
    I_{\mathrm{Crit}} \le C_{1}^{|E|-|V|+1} \int_{\alpha_e \ge 0, \sum_e \alpha_e \le Ct} 1\prod_{e \in T} d\alpha_e \prod_{e \notin T} (1+(\alpha_e\tau)^{-\frac{1}{\alpha}}) d\alpha_e.
    \end{equation}

    Let $k = |E|-|V|+1$ be the number of cycles in the connected diagram $\Gamma$, then
    \begin{equation}
    \begin{aligned}
        I_{\mathrm{Crit}}&\le C_{1}^{|E|-|V|+1}\sum_{j=1}^k \binom{k}{j} \int_{\alpha_i \ge 0, \sum_i^{|E|} \alpha_i \le Ct} 1\prod_{i=1}^{j}(\alpha_i\tau)^{-\frac{1}{\alpha}} \prod_{i=1}^{|E|}d\alpha_i\\
        &=C_{1}^{|E|-|V|+1}\sum_{j=1}^k \binom{k}{j} \frac{\left[\Gamma\left(1 - \frac{1}{\alpha}\right)\right]^j\tau^{-\frac{j}{\alpha}}}{\Gamma\left(1 + |E| - \frac{j}{\alpha}\right)} (Ct)^{|E| - \frac{j}{\alpha}}.
    \end{aligned}
    \end{equation}

    Moreover, for a typical diagram, we have $|E| - k \ge  |E|/10$ and $k \ge |E|/10$. The binomial coefficient satisfies the elementary bound $\binom{k}{j} \le 2^k$. Consequently,
    \begin{equation}
        \begin{aligned}
            I_{\mathrm{Crit}}&\le (2C_1)^{|E|-|V|+1}\sum_{j=1}^k  \frac{\left[\Gamma\left(1 - \frac{1}{\alpha}\right)\right]^j\tau^{-\frac{j}{\alpha}}}{\Gamma\left(1 + |E| - \frac{j}{\alpha}\right)} (Ct)^{|E| - \frac{j}{\alpha}}\\
            &\le (C_2)^{|E|-|V|+1}\sum_{j=1}^{k}\frac{(Ct)^{|E|}(\tau t)^{-\frac{j}{\alpha}}}{(|E|)^{|E|-\frac{j}{\alpha}}}\\
            &\le (C_2)^{|E|-|V|+1}\frac{(Ct)^{|E|}}{(|E|)^{|E|}}\sum_{j=1}^{k}(\frac{\tau t}{|E|})^{-\frac{j}{\alpha}}.
        \end{aligned}
    \end{equation}
    The last sum is dominated by the first and last term, which corresponding to the supercritical and subcritical cases.
    \end{enumerate}

    Thus, we finish the proof.

\end{proof}

\begin{theorem}\label{thm:limit_cumulant}
Consider the mixed cumulant $\kappa_X(n_1,\dots,n_s)$ as   in  Definition \ref{Cumu}   for the $\alpha$-stable random band Gaussian matrices (Definition \ref{defmodel}) with $ \alpha>d=1$. Assume that $n = \sum_{i=1}^s n_i$ is even and $t_i \in (0,\infty)$ for $i=1,\dots,s$. Suppose further that the parameters $(\{n_i\}, W, N)$ satisfy the following scalings:
\begin{itemize}
    \item \textbf{Supercritical regime:} $n_i \sim t_i N^{1/3}$;
    \item \textbf{Subcritical/Critical regimes:} $n_i \sim t_i W^{\alpha/(3\alpha-1)}$.
\end{itemize}
Then, as $N \to \infty$,
\begin{equation}\label{equ:asy_kappa}
\frac{1}{n_1\cdots n_s}\,\kappa_X(n_1,\dots,n_s) = \big(1+o(1)\big)\sum_{\Gamma\in \mathcal{T}_{s}} \mathcal{L}_\Gamma \bigl(\{t_i\}, (W,N)\bigr),
\end{equation}
where $\mathcal{T}_s$ is the set of typical connected diagrams on $s$ marked points (Definition \ref{def:typical_diagram}), and
\begin{equation} \label{Ldef}
\mathcal{L}_\Gamma (\{t_i\}, (W,N)): = \begin{cases}
    \mathcal{F}_{\Gamma}^{\mathrm{(Super)}}(t_1,\dots,t_s), & \mathrm{Supercritical ~case,} \\
    N W^{-\frac{3\alpha}{3\alpha-1}}\mathcal{F}_{\Gamma}^{\mathrm{(Sub)}}(t_1,\dots,t_s), & \mathrm{Subcritical ~case,} \\
    \gamma^{|E|-|V|}\mathcal{F}_{\Gamma}^{\mathrm{(Crit)}}(t_1,\dots,t_s;\gamma^{\alpha}), & \mathrm{Critical ~case ~} (W \sim (\gamma N)^{1-\frac{1}{3\alpha}}).
\end{cases}
\end{equation}
Furthermore, there exist  constants $C, C_{\gamma} > 0$ such that the sum is uniformly bounded by
\begin{equation}\label{equ:limit_upper_bound}
\sum_{\Gamma\in \mathcal{T}_{s}} \mathcal{L}_\Gamma \le \begin{cases}
    e^{Ct^{3/2}}, & \mathrm{Supercritical ~case,} \\
    N W^{-\frac{3\alpha}{3\alpha-1}}e^{Ct^{\frac{3\alpha-1}{2\alpha-1}}}, & \mathrm{Subcritical ~case,} \\
    e^{Ct^{3/2}} + e^{C_{\gamma}t^{\frac{3\alpha-1}{2\alpha-1}}}, & \mathrm{Critical ~case.}
\end{cases}
\end{equation}
\end{theorem}

\begin{proof}
We first verify   \eqref{equ:limit_upper_bound}. Set
\begin{equation}
    \mathcal{G}_{\Gamma}(t):=\begin{cases}
        \frac{(Ct)^{|E|}}{(|E|)!}, &\text{Supercritical case},\\
        NW^{-\frac{3\alpha}{3\alpha-1}}\frac{(Ct)^{|E|-\frac{1}{\alpha}(|E|-|V|+1)}}{(|E|-\frac{1}{\alpha}(|E|-|V|+1))!}, &\text{Subcritical case},\\
        \frac{(Ct)^{|E|}}{(|E|)!}+\frac{(C_{\gamma}t)^{|E|-\frac{1}{\alpha}(|E|-|V|+1)}}{(|E|-\frac{1}{\alpha}(|E|-|V|+1))!}, &\text{Critical case},
    \end{cases}
\end{equation}
we see from  Definition \ref{def:limiting_function} and Lemma \ref{lem:limit_function_upper_bound} that
\begin{equation}
    \sum_{\Gamma\in \mathcal{T}_{s}} \mathcal{L}_\Gamma (\{t_i\}, (W,N))\le \sum_{\Gamma\in \mathcal{T}_{s}} \mathcal{G}_{\Gamma}(t).
\end{equation}
Then following the lines of \eqref{kappab-1}-\eqref{kappab-4}, we obtain
\begin{equation}
\sum_{\Gamma\in \mathcal{T}_{s}} \mathcal{G}_{\Gamma}(t)\le \begin{cases}
        e^{Ct^{\frac{3}{2}}},&\text{Supercritical case}\\
        NW^{-\frac{3\alpha}{3\alpha-1}}e^{Ct^{\frac{3\alpha-1}{2\alpha-1}}},&\text{Subcritical case},\\
        e^{Ct^{\frac{3}{2}}} + e^{C_{\gamma}t^{\frac{3\alpha-1}{2\alpha-1}}}&\text{Critical case}.
    \end{cases}
\end{equation}

Next for \eqref{equ:asy_kappa}, we need to consider the cumulant $\kappa$, which corresponds to connected diagrams. By \eqref{equ:4.20}, \eqref{equ:4.22} and \eqref{kappab-4}, we know that
\begin{equation}
\sum_{\Gamma \in \mathscr{D}_{s;\beta}^*}
\frac{1}{n^s\big(1+\frac{1}{n}Nb_n\big)}
F_{\Gamma}(\{n_j\}_{j=1}^s)
\end{equation}
is absolutely convergent. Hence it suffices to consider the termwise limit of the diagram functions, which is given termwise by $\mathcal{L}_\Gamma \bigl(\{t_i\}; (W,N)\bigr)$ according to Theorem~\ref{thm:homogeneous_theorem}   and the established result \eqref{equ:limit_upper_bound}. Here, for the subcritical case, we  have used
\begin{equation}
    \big(1+\frac{1}{n}Nb_n\big)\sim \frac{1}{n}N \frac{1}{W}\sum_{i=1}^n n^{-\frac{1}{\alpha}}\sim \frac{N}{n^{\frac{1}{\alpha}}W}\sim N W^{-\frac{3\alpha}{3\alpha-1}}.
\end{equation}
This completes   the proof of \eqref{equ:asy_kappa}.
\end{proof}

To derive the limit correlation measure, we need the relationship between cumulants and mixed moments.
\begin{theorem}\label{prop:mix_prod}
For   polynomials $P_i(X)$ defined via Chebyshev polynomials of  second kind
\begin{equation}
P_i(X) :=
\begin{cases}
\displaystyle \Bigl(\frac{1}{n_i+1}\, U_{n_i}\bigl(\tfrac{X}{2}\bigr)\Bigr)^{\!m}, & i = 1,\ldots, q, \\[8pt]
\displaystyle \Bigl(\frac{1}{n_i+1}\, U_{n_i}\bigl(\tfrac{X}{2}\bigr)\Bigr)^{\!m-1} \frac{1}{n_i+2}\, U_{n_i+1}\bigl(\tfrac{X}{2}\bigr), & i = q+1,\ldots, s,
\end{cases}
\end{equation} we define $\tau(I)$ recursively  for any subset $I \subset [s]$ by
\begin{equation}\label{equ:tau}
\tau(I) := \mathbb{E}\!\Big[\,\prod_{i \in I} \operatorname{Tr} P_i(X)\Big] \;-\; \sum_{\Pi \in \mathscr{P}(I)} \prod_{\pi \in \Pi} \tau(\pi),
\end{equation}
where the sum runs over all \emph{proper} partitions of $I$.
\begin{enumerate}
    \item[(i)] {\bf Even case:}  If   $\displaystyle S(I) := \sum_{i \in I \cap [q]} m n_i \;+\!\! \sum_{i \in I \setminus [q]} (m n_i + 1)$ is even, and $t_i \in (0,\infty)$ for $i=1,\dots,s$ such that
    \begin{itemize}
    \item Supercritical regime: $n_i \sim t_i N^{1/3}$;
    \item Subcritical/Critical regimes: $n_i \sim t_i W^{\alpha/(3\alpha-1)}$.
    \end{itemize}
    Then
    \begin{equation}\label{equ:asy_tau}
    \tau(I) = \bigl(1+o(1)\bigr) \sum_{\Gamma \in \mathcal{T}_{s}}
    \int_{\xi_1}\! \cdots\! \int_{\xi_s}
    \Bigl[\,\prod_{i \in I} \mathcal{Q}_m(\xi_i)\Bigr]\,
    \mathcal{L}_\Gamma\bigl(\{\xi_i t_i\}_{i\in I}; (W,N)\bigr)
    \,d\xi_1 \cdots d\xi_{s},
    \end{equation}
     where
     \begin{equation}\label{eq:Pt_def}
        \mathcal{Q}_m(\xi) := \xi \, \mathcal{P}_m(\xi), \quad \mathcal{P}_m(\xi) = \frac{1}{2^{m-1}(m-2)!} \sum_{j=0}^{m} (-1)^j \binom{m}{j} \left( m - 2j - \xi \right)_+^{m-2}.
    \end{equation}

    \item[(ii)] {\bf Odd case:} If $S(I)$ is odd, then
    $
    \tau(I) = 0.$
\end{enumerate}

\end{theorem}
\begin{proof}
We prove only the simplified case \(q = s\) and \(I = [s]\) since the general case is very similar. By Lemmas \ref{lem:c_asymptotics} and \ref{lem:c_perturbed}, we can expand each
\begin{equation}
    \bigl( U_{n_i}(x) \bigr)^{m} = \sum_{k \ge 0} c_m(n_i; k) \, U_k(x).
\end{equation}

Hence, we obtain
\begin{equation}
    \mathbb{E}\!\Big[\prod_{i=1}^s\operatorname{Tr} P_i(X)\Big]
= \sum_{k_1,\dots,k_s} \,\prod_{i=1}^s \frac{c_m(n_i; k_i)}{(n_i+1)^m}
\; \mathbb{E}\!\Big[\prod_{i=1}^s\frac{1}{k_i+1}U_{k_i}(X)\Big].
\end{equation}
For each inner term, the recurrence in \eqref{equ:tau} is identical to that  in \eqref{equ:T}. Consequently,
\begin{equation}
\tau([s]) = \sum_{k_1,\dots,k_s} \,\bigg(\prod_{i=1}^s \frac{c_m(n_i; k_i)}{(n_i+1)^m}\bigg)
\; \kappa(k_1,\ldots,k_s).
\end{equation}
By Lemma \ref{lem:c_asymptotics}, the product \(\prod_{i=1}^s c_m(n_i; k_i)\) vanishes unless \(k_i \equiv m n_i \pmod 2\) for every \(i\in [s]\), and the following asymptotic relation holds:
\begin{equation}
\prod_{i=1}^s\frac{c_m(n_i; k_i)}{(n_i+1)^m}
\sim \prod_{i=1}^s\frac{2\mathcal{Q}_m(\xi_i)}{n_i+1}.
\end{equation}

Finally, expanding \(\kappa(k_1,\ldots,k_s)\)
via \eqref{equ:asy_kappa} and  interchanging integration and summation completes the proof.
\end{proof}
\begin{definition}[Limit $k$-point correlation measure in critical regime] \label{def:Rcrit}
For a fixed $\gamma > 0$, the \textbf{critical correlation measure} $R_{k}^{\mathrm{(Crit)}}(\gamma; dx_1, \dots, dx_k)$ is the unique Radon measure on $\mathbb{R}^k$ whose mixed moments are determined by its sinc transform. Specifically, for any $t_1, \dots, t_k > 0$ and $m \in \{4, 8, 10\}$, the measure satisfies
\begin{equation}\label{equ:def_Rcrit}
    \int_{\mathbb{R}^k} \prod_{j=1}^k \frac{\sin^{m} (t_j\sqrt{-x_j})}{(t_j\sqrt{-x_j})^{m}} \, R_{k}^{\mathrm{(Crit)}}(\gamma; dx_1, \dots, dx_k) = \sum_{\Pi \in \mathscr{P}([k])} \prod_{\pi \in \Pi} \tau_{\gamma}(\pi),
\end{equation}
where the right-hand side is expressed via the partition sum of limiting cumulants, with the critical cumulant limit $\tau_{\gamma}$ given by the sum over connected typical diagrams
\begin{equation*}
    \tau_{\gamma}(\{t_j\}_{j \in \pi}) = \sum_{\Gamma \in \mathcal{T}_{|\pi|}} \gamma^{|E|-|V|} \int \left[ \prod_{j \in \pi} \mathcal{Q}_m(\xi_j) \right] \mathcal{F}^{\mathrm{(Crit)}}_{\Gamma}( \{\xi_j t_j\}_{j\in \pi};\gamma^\alpha) \prod d\xi_j.
\end{equation*}
\end{definition}

\begin{definition}[Limit $1$-point measure in subcritical regime] \label{def:Rsub}
The \textbf{subcritical 1-point measure} $R_{1}^{\mathrm{(Sub)}}(dx)$ is the unique Radon measure on $\mathbb{R}$ determined by its sinc transform: for any $t > 0$,
\begin{equation}\label{equ:def_Rsub}
    \int_{-\infty}^{\infty} \frac{\sin^{m} (t\sqrt{-x})}{(t\sqrt{-x})^{m}} \, R_{1}^{\mathrm{(Sub)}}(dx) = \sum_{\Gamma \in \mathcal{T}_{1}} \int_{0}^{1} \mathcal{Q}_m(\xi) \mathcal{F}^{\mathrm{(Sub)}}_{\Gamma}(\xi t) \, d\xi.
\end{equation}
\end{definition}
The rescaled $k$-point correlation measure is defined as the pushforward of $R_{N;k}$ under the map $T_{s_N}(x_1,\ldots,x_k) = (2 + s_N x_1, \ldots, 2 + s_N x_k)$, denoted by $(T_{s_N})_* R_{N;k}(dx_1, \dots, dx_k)$. Specifically, for any continuous function $g$ with compact support,
\begin{equation}
\int_{\mathbb{R}^k} g(y) \, (T_{s_N})_* R_{N;k}(dy_1, \ldots, dy_k) = \int_{\mathbb{R}^k} g(T_{s_N}(x)) \, R_{N;k}(dx_1, \ldots, dx_k).
\end{equation}
\begin{theorem}[Edge statistics of $\alpha$-stable band matrices] \label{thm:convergence_proof}
For the $\alpha$-stable Gaussian band matrices defined in Definition \ref{defmodel} with $\alpha>d=1$, the scaling limits hold as $N \to \infty$:
    \begin{enumerate}
        \item[(i)] \textbf{(Supercritical regime)} When  $W \gg N^{1-\frac{1}{3\alpha}}$, with $s_N=N^{-2/3}$
        the rescaled $k$-point correlation   measure
        \begin{equation}
       (T_{s_N})_*R_{N;k}(dx_1, \ldots, dx_k)
        \end{equation}

        converges  vaguely to those of the GOE/GUE  Airy point  process.

        \item[(ii)] \textbf{(Critical regime)} When  $W \sim \gamma N^{1-\frac{1}{3\alpha}}$ with some   constant $\gamma > 0$, with  $s_N = W^{-\frac{2\alpha}{3\alpha-1}}$ the rescaled $k$-point correlation measure
        \begin{equation}
            (T_{s_N})_*R_{N;k}(dx_1, \ldots, dx_k)
        \end{equation}

        converges vaguely  to $R^{\mathrm{(Crit)}}_{k}(\gamma;d y_1,\ldots,d y_k)$.

\item[(iii)] \textbf{(Subcritical regime)} When  $(\log N)^{1/\alpha} \ll W \ll N^{1-\frac{1}{3\alpha}}$, let $s_N = W^{-\frac{2\alpha}{3\alpha-1}}$ and $\gamma_N = N W^{-\frac{3\alpha}{3\alpha-1}}$. As $N \to \infty$, for any $k \ge 1$, the rescaled $k$-point correlation measures satisfy the  factorization property:
\begin{equation}
    \gamma_N^{-k}  (T_{s_N})_* R_{N;k}(dx_1, \dots, dx_k) \rightarrow \prod_{j=1}^k R^{\mathrm{(Sub)}}_1(dy_j).
\end{equation}
    \end{enumerate}
\end{theorem}
We remark that in the subcritical regime this product limit indicates that the local eigenvalue statistics in the subcritical regime become asymptotically independent within the spectral window of scale $s_N$. Consequently, the rescaled $k$-point correlation measures decouple into the $k$-fold tensor product of the limiting $1$-point measure $R^{\mathrm{(Sub)}}_1$.
\begin{proof}

We prove the \textbf{critical case} first. We can proceed along the same lines  as in the critical one to complete the proof of   the \textbf{supercritical} case,  therefore, we omit the details.

Starting from the recursive definition \eqref{equ:tau},  for $m \in \{4,8,10\}$ we have
\begin{equation}
    \mathbb{E}\!\left[\,\prod_{i=1}^{s} \operatorname{Tr} P_{i}(X)\right]
    = \sum_{\Pi \in \mathscr{P}([s])} \prod_{\pi \in \Pi} \tau(\pi),
\end{equation}
where the sum runs over all partitions of $[s]$.

\textbf{Base case: $s=1$.}
Take $n_1 \sim t_1 W^{\frac{\alpha}{3\alpha-1}}$. In this case, $\mathcal{L}_\Gamma\bigl(\{\xi_1 t_1\}; (W,N)\bigr)$ is independent of $W$ and $N$, we derive from the asymptotic formula \eqref{equ:asy_tau}  that
\begin{equation}\label{equ:limit_P}
    \lim_{N \to \infty} \mathbb{E}\bigl[\operatorname{Tr} P_{1}(X)\bigr]
    = \sum_{\Gamma \in \mathcal{T}_{1}} \int_{\xi_1} \mathcal{Q}_m(\xi_1) \,
    \mathcal{L}_\Gamma\bigl(\{\xi_1 t_1\}; (W,N)\bigr) \, d\xi_1 .
\end{equation}
On the other hand, by \cite[Lemma 4.8]{liu2025edge},
\begin{equation}
    \lim_{n \to \infty} \frac{1}{n+1} \, U_n\!\Bigl(1 + \frac{t^2 y}{2 n^2}\Bigr)
    = \frac{\sin\bigl(t\sqrt{-y}\,\bigr)}{t\sqrt{-y}} .
\end{equation}
Consequently, the rescaled moment can be written as
\begin{equation}\label{eq:moment_integral}
    \widehat{\sigma}_{N}^{(m)}(t)
    := \int_{-\infty}^{\infty} \frac{\sin^{m}\!\bigl(t\sqrt{-\lambda}\,\bigr)}
           {\bigl(t\sqrt{-\lambda}\,\bigr)^{m}} \,
      dR_{N;1}(\lambda)
    = \bigl(1+o(1)\bigr) \,
      \mathbb{E}\bigl[\operatorname{Tr} P_{1}(X)\bigr].
\end{equation}
Applying the continuity theorem \cite[Theorem B.10]{liu2023edge}, we deduce the existence of a unique limit measure $R_{1}(\gamma;\cdot)$ such that
\begin{equation}\label{eq:limit_measure}
    \int_{-\infty}^{\infty} \frac{\sin^{m}\!\bigl(t\sqrt{-\lambda}\,\bigr)}
           {\bigl(t\sqrt{-\lambda}\,\bigr)^{m}} \,
      R_{1}(\gamma;d\lambda)
    = \sum_{\Gamma \in \mathcal{T}_{1}} \int_{\xi_1} \mathcal{Q}_m(\xi_1) \,
      \mathcal{L}_\Gamma\bigl(\{\xi_1 t_1\}; (W,N)\bigr) \, d\xi_1,
\end{equation}
which implies  the convergence of the  one-point measure.

 \textbf{Induction to $k\ge 2$.}
For $k=2$, by the definition of the two‑point correlation measure we have
\begin{equation}\label{equ:R_2}
\begin{aligned}
\int f(x_1)g(x_2) \, dR_{N;2}(x_1,x_2)
&= \mathbb{E}\bigl[\operatorname{Tr} f(H)\,\operatorname{Tr} g(H)\bigr]
   - \int f(x)g(x) \, dR_{N;1}(x)
\end{aligned}
\end{equation}
for  nice functions, say  polynomials.
The right‑hand side is precisely a mixed moment of Chebyshev polynomials, which can be expressed through the partition sums $\tau(I)$ already controlled in the $s=1$ case. Iterating this argument for higher $k$,  using the recurrence \eqref{equ:tau},  yields the convergence of the full $k$-point correlation measure  $R_{N;k}$ to  $R^{\mathrm{(Crit)}}_{k}$.
For the \textbf{subcritical case},
we see from \eqref{Ldef} that
\begin{equation}
    \mathcal{L}_\Gamma\bigl(\{\xi_1 t_1\}; (W,N)\bigr)=O(NW^{-\frac{3\alpha}{3\alpha-1}})=O(\gamma_N).
\end{equation}
Thus for the rescaled quantity \begin{equation}
    \gamma_N^{-s}\mathbb{E}\!\left[\,\prod_{i=1}^{s} \operatorname{Tr} P_{i}(X)\right]
    = \gamma_N^{-s}\sum_{\Pi \in \mathscr{P}([s])} \prod_{\pi \in \Pi} \tau(\pi),
\end{equation}
 in the large $N$ limit the only non-vanishing term is $\gamma_N^{-s} \prod_{j=1}^s \tau(\{j\}).$ That is,
\begin{equation}
    \lim\limits_{N\rightarrow\infty}
    \int_{\mathbb{R}^k}\prod_{i=1}^s\frac{\sin^{m}\!\bigl(t_i\sqrt{-x_i}\,\bigr)}
           {\bigl(t_i\sqrt{-x_i}\,\bigr)^{m}} \gamma_N^{-k} R_{N;k}(2 + s_N x_1, \dots, 2 + s_N x_k) = \prod_{i=1}^k \int_{\mathbb{R}}\frac{\sin^{m}\!\bigl(t_i\sqrt{-x_i}\,\bigr)}
           {\bigl(t_i\sqrt{-x_i}\,\bigr)^{m}} R^{\mathrm{(Sub)}}_1(d x_i).
\end{equation}

Hence, by the continuity theorem \cite[Theorem B.10]{liu2023edge} again, we finish the proof in  the subcritical case.

This finally completes   the proof.
\end{proof}

By invoking Theorem \ref{prop:super_4}, we extend the edge statistics results to random matrices with general profile functions.
\begin{theorem}[Universality for general power-law profiles] \label{thm:band_main}
For $d=1$ and $\alpha > 1$, consider the random band matrix $H_N$ defined in Definition \ref{defmodel} with profile function $f$. Assume $f$ satisfies the comparison conditions in Definition \ref{ass:ft} with respect to the $\alpha$-stable density $f_{\alpha}$ in Definition \ref{def:alpha_profile}. Let $R_{N;k}$ denote the $k$-point correlation measure of $H_N$, and let $(T_{s_N})_* R_{N;k}$ be its rescaling through the map $T_{s_N}(x_i) = 2 + s_N x_i$. Then, the following scaling limits hold as $N \to \infty$:
    \begin{enumerate}
        \item[(i)] \textbf{(Supercritical regime)} When  $W \gg N^{1-\frac{1}{3\alpha}}$, with $s_N=N^{-2/3}$ the rescaled $k$-point correlation   measure
        \begin{equation}
       (T_{s_N})_*R_{N;k}(dx_1, \ldots, dx_k)
        \end{equation}
        converges vaguely to those of the GOE/GUE  Airy point  process.

        \item[(ii)] \textbf{(Critical regime)} When  $W \sim \gamma N^{1-\frac{1}{3\alpha}}$ with some   constant $\gamma > 0$, with  $s_N = W^{-\frac{2\alpha}{3\alpha-1}}$ the rescaled $k$-point correlation measure
        \begin{equation}
            (T_{s_N})_*R_{N;k}(dx_1, \ldots, dx_k)
        \end{equation}
        converges vaguely  to $R^{\mathrm{(Crit)}}_{k}(\gamma;d y_1,\ldots,d y_k)$ defined in Definition \ref{def:Rcrit}.

        \item[(iii)] \textbf{(Subcritical regime)} When  $(\log N)^{1/\alpha} \ll W \ll N^{1-\frac{1}{3\alpha}}$, let $s_N = W^{-\frac{2\alpha}{3\alpha-1}}$ and $\gamma_N = N W^{-\frac{3\alpha}{3\alpha-1}}$. As $N \to \infty$, for any $k \ge 1$, the rescaled $k$-point correlation measures satisfy the  factorization property:
        \begin{equation}
            \gamma_N^{-k}  (T_{s_N})_* R_{N;k}(dx_1, \dots, dx_k) \rightarrow \prod_{j=1}^k R^{\mathrm{(Sub)}}_1(dy_j).
        \end{equation}
        Here $R^{\mathrm{(Sub)}}_1(dy_j)$ is given in Definition \ref{def:Rsub}.
    \end{enumerate}
\end{theorem}

\begin{proof}
This result is a direct application of the Short-to-Long Comparison principle established in Theorem \ref{prop:super_4}. As demonstrated in Appendix \ref{sec:band_comparison}, for any profile function $f$ satisfying the conditions in Definition \ref{ass:ft}, the associated Markov transition probabilities $p_n(x, y)$ satisfy the $\ell^\infty$ and $\ell^1$ distance constraints required by the conditions \ref{itm:B1} and \ref{itm:B2} relative to the $\alpha$-stable  model, with
\begin{equation}
    \epsilon_n=\sum_{y\in [N]}\left| p_n(x,y) - \tilde{p}_n(x,y) \right|=o(1),
\end{equation}
and
\begin{equation}
    \delta_n=o\big(W^{-d} n^{-\tfrac{d}{\alpha}}\big) +\; O( n W^{-K}),
\end{equation}
respectively due to Proposition \ref{prop:L1_difference} and Proposition \ref{thm:clt-upper}.

Consequently, the two asymptotic conditions in \eqref{equ:negligible} are satisfied for $n$ with the same scale in Theorem \ref{thm:limit_cumulant} and the diagram functions $F_\Gamma$ for the general power-law matrix converge to the same limits $\mathcal{L}_\Gamma$ as those derived for the $\alpha$-stable case in Theorem \ref{thm:homogeneous_theorem}. This implies that Theorem \ref{prop:mix_prod} also holds.
     Following the same steps as in the proof of Theorem~\ref{thm:convergence_proof}, we conclude that the \(k\)-point correlation measures  admit identical universal limits in  the  three different   regimes.
\end{proof}
\subsection{
Poisson-Airy transition}\label{sec:poisson}
Now we examine the crossover phenomenon that emerges from the critical limits associated with the critical bandwidth. The limit diagram functions in the  supercritical and subcritical regimes, defined in Definition \ref{def:limiting_function},    are recovered when  $\gamma$ tends to $\infty$ or $0$.
\begin{lemma}\label{lem:crossover_analysis} For any diagram  $\Gamma=(V,E)\in \mathcal{T}_s$,  the critical  diagram functions   admit a transition
\begin{equation}\label{equ:critical=sup}
    \lim\limits_{\gamma\rightarrow\infty}\gamma^{|E|-|V|}\mathcal{F}_{\Gamma}^{\mathrm{(Crit)}}(\gamma^{-\frac{1}{3}}t_1,\ldots,\gamma^{-\frac{1}{3}}t_s;\gamma^{\alpha})=\mathcal{F}_{\Gamma}^{\mathrm{(Super)}}(t_1,\ldots,t_s),
\end{equation}
and
\begin{equation}\label{equ:critical=sub}
    \lim\limits_{\gamma\rightarrow0}\gamma^{|E|-|V|+1}\mathcal{F}_{\Gamma}^{\mathrm{(Crit)}}(t_1,\ldots,t_s;\gamma^{\alpha})=\mathcal{F}_{\Gamma}^{\mathrm{(Sub)}}(t_1,\ldots,t_s).
\end{equation}
\end{lemma}
\begin{proof}

We begin directly  with  Definition \ref{def:limiting_function}    and  use Lemma \ref{lem:asy_theta} for small and large $\tau=\gamma^\alpha$ respectively.  As $\gamma\rightarrow\infty$, we have
\begin{equation}
\begin{aligned}
    &\gamma^{|E|-|V|}\mathcal{F}_{\Gamma}^{\mathrm{(Crit)}}(t_1,\ldots,t_s;\gamma^{\alpha})\\
    &=\gamma^{|E|-|V|+\frac{s}{3}}\frac{C_\Gamma}{\,t_1\cdots t_s\,}\int_{x_v\in \mathbb{T}^d}\int_{\mathfrak{C}(\{t_i\})}
    \prod_{e\in E(\Gamma)}\theta_{\alpha}(x_{u}-x_{v},\alpha_e\gamma^{\alpha}) \prod_{e\in E(\Gamma)}d\alpha_e \prod_{v\in V(\Gamma)}dx_v\\
    &=(1+o(1)) \gamma^{|E|-|V|+\frac{s}{3}}\frac{C_\Gamma}{\,t_1\cdots t_s\,}\int_{x_v\in \mathbb{T}^d}\int_{\mathfrak{C}(\{\gamma^{-\frac{1}{3}}t_i\})}
    1 \prod_{e\in E(\Gamma)}d\alpha_e \prod_{v\in V(\Gamma)}dx_v\\
    &=(1+o(1)) \gamma^{|E|-|V|+\frac{s}{3}-\frac{|E|}{3}} \frac{C_\Gamma}{\,t_1\cdots t_s\,}\int_{\mathfrak{C}(\{t_i\})}
    1 \prod_{e\in E(\Gamma)}d\alpha_e\\
    &=(1+o(1))\mathcal{F}_{\Gamma}^{\mathrm{(Super)}}(t_1,\ldots,t_s).
\end{aligned}
\end{equation}
Thus we obtain \eqref{equ:critical=sup}.

As $\gamma\rightarrow 0$, we have
\begin{equation}
\begin{aligned}
    &\gamma^{|E|-|V|+1}\mathcal{F}_{\Gamma}^{\mathrm{(Crit)}}(\gamma^{-\frac{1}{3}}t_1,\ldots,\gamma^{-\frac{1}{3}}t_s;\gamma^{\alpha})\\
    &=\gamma^{|E|-|V|+1}\frac{C_\Gamma}{\,t_1\cdots t_s\,}\int_{x_v\in \mathbb{T}^d}\int_{\mathfrak{C}(\{t_i\})}
    \prod_{e\in E(\Gamma)}\theta_{\alpha}(x_{u}-x_{v},\alpha_e\gamma^{\alpha}) \prod_{e\in E(\Gamma)}d\alpha_e \prod_{v\in V(\Gamma)}dx_v.\\
\end{aligned}
\end{equation}
Now we treat
\begin{equation}
    \int_{x_v\in \mathbb{T}^d}\prod_{e\in E(\Gamma)}\theta_{\alpha}(x_{u}-x_{v},\alpha_e\gamma^{\alpha})\prod_{v\in V(\Gamma)}dx_v.
\end{equation}
Due to the translation invariance, we can fix $v_0=0$. Then we replace
\begin{equation}
    \theta_{\alpha}(x_{u}-x_{v},\alpha_e\gamma^{\alpha})\longrightarrow \alpha_e^{-\frac{1}{\alpha}}\gamma^{-1}f_\alpha( \alpha_e^{-\frac{1}{\alpha}}\gamma^{-1} {(x_{u}-x_{v})}),
\end{equation}
and
\begin{equation}
    \gamma^{-1} (x_{u}-x_{v}) \rightarrow y_{u}-y_{v}.
\end{equation}
The integral reduces to
\begin{equation}
\begin{aligned}
    &\int_{x_v\in \mathbb{T}^d,x_{v_0}=0}\prod_{e\in E(\Gamma)}\alpha_e^{-\frac{1}{\alpha}}\gamma^{-1}f_\alpha( \alpha_e^{-\frac{1}{\alpha}}{(y_{u}-y_{v})})\prod_{v\in V(\Gamma),v\ne v_0}\gamma dy_v\\
    &=\gamma^{|V|-1-|E|}\int_{x_v\in \mathbb{T}^d,x_{v_0}=0}\prod_{e\in E(\Gamma)}\alpha_e^{-\frac{1}{\alpha}}f_\alpha( \alpha_e^{-\frac{1}{\alpha}}{(y_{u}-y_{v})})\prod_{v\in V(\Gamma),v\ne v_0}dy_v\\
    &=\gamma^{|V|-1-|E|}\int_{x_v\in \mathbb{T}^d,x_{v_0}=0}\prod_{e\in E(\Gamma)}f_\alpha(y_{u}-y_{v},\alpha_e)\prod_{v\in V(\Gamma),v\ne v_0}dy_v.
\end{aligned}
\end{equation}
Thus we get
\begin{equation}
    \lim\limits_{\gamma\rightarrow0}\gamma^{|E|-|V|+1}\mathcal{F}_{\Gamma}^{\mathrm{(Crit)}}(t_1,\ldots,t_s;\gamma^{\alpha})=\mathcal{F}_{\Gamma}^{\mathrm{(Sub)}}(t_1,\ldots,t_s).
\end{equation}

Hence we finish the proof.
\end{proof}
Lemma~\ref{lem:crossover_analysis} indeed shows that the critical measures \( R_{k}^{\mathrm{(Crit)}}(\gamma;x_1,\dots,x_k) \) interpolate between the correlation function of the Airy point process and a Poisson point process with intensity \( R^{(\mathrm{Sub})}_{1}(x) \).

\begin{corollary}\label{coro:poisson_statistics}
Let $R_{k}^{\mathrm{(Crit)}}(\gamma;dx_1,\ldots,dx_k)$ be the critical $k$-point correlation measure in Definition \ref{def:Rcrit}, then
\begin{equation} \label{supertransition}
   \lim_{\gamma\rightarrow\infty} R^{\mathrm{(Crit)}}_{k}\big(\gamma; d\gamma^{-\frac{2}{3}}x_1,\ldots,d \gamma^{-\frac{2}{3}}x_k\big){=} R^{\mathrm{Airy}}_{k}(dx_1,\ldots, dx_k),\quad\text{vaguely,}
\end{equation}
and  \begin{equation} \label{subtransition}
     \lim_{\gamma\rightarrow 0}  \gamma^{k} R^{\mathrm{(Crit)}}_{k}(\gamma;dx_1,\ldots,dx_k){=} \prod_{i=1}^k  R^{(\mathrm{Sub})}_{1}(dx_i),\quad\text{vaguely.}
\end{equation}
\end{corollary}
\begin{proof}

Recalling from Definition \ref{def:Rcrit}, we have
\begin{equation}
    \int_{\mathbb{R}^k} \prod_{j=1}^k \frac{\sin^{m} (t_j\sqrt{-x_j})}{(t_j\sqrt{-x_j})^{m}} \, R_{k}^{\mathrm{(Crit)}}(\gamma; dx_1, \dots, dx_k) = \sum_{\Pi \in \mathscr{P}([k])} \prod_{\pi \in \Pi} \tau_{\gamma}(\pi),
\end{equation}
with
\begin{equation}
    \tau_{\gamma}(\{t_j\}_{j \in \pi}) = \sum_{\Gamma \in \mathcal{T}_{|\pi|}} \gamma^{|E|-|V|} \int \left[ \prod_{j \in \pi} \mathcal{Q}_m(\xi_j) \right] \mathcal{F}^{\mathrm{(Crit)}}_{\Gamma}( \{\xi_j t_j\}_{j\in \pi};\gamma^\alpha) \prod d\xi_j.
\end{equation}
From Lemma \ref{lem:crossover_analysis}, we have
\begin{equation}
    \lim\limits_{\gamma\rightarrow\infty}\gamma^{|E|-|V|}\mathcal{F}_{\Gamma}^{\mathrm{(Crit)}}(\gamma^{-\frac{1}{3}}t_1,\ldots,\gamma^{-\frac{1}{3}}t_s;\gamma^{\alpha})=\mathcal{F}_{\Gamma}^{\mathrm{(Super)}}(t_1,\ldots,t_s).
\end{equation}
Therefore, it is easy to see that
\begin{equation}
    \begin{aligned}
         &\lim_{\gamma\rightarrow\infty}\int_{\mathbb{R}^k} \prod_{j=1}^k \frac{\sin^{m} (t_j\sqrt{-x_j})}{(t_j\sqrt{-x_j})^{m}} \, R_{k}^{\mathrm{(Crit)}}(\gamma; d\gamma^{-2/3}x_1, \dots, d\gamma^{-2/3}x_k)
        \\
        &=\lim_{\gamma\rightarrow\infty}\int_{\mathbb{R}^k} \prod_{j=1}^k \frac{\sin^{m} (\gamma^{-1/3}t_j\sqrt{-x_j})}{(\gamma^{-1/3}t_j\sqrt{-x_j})^{m}} \, R_{k}^{\mathrm{(Crit)}}(\gamma; dx_1, \dots, dx_k)\\
        &=\lim_{\gamma\rightarrow\infty}\int_{\mathbb{R}^k} \prod_{j=1}^k \frac{\sin^{m} (t_j\sqrt{-x_j})}{(t_j\sqrt{-x_j})^{m}} \, R_{k}^{\mathrm{Airy}}( dx_1, \dots, dx_k).
    \end{aligned}
\end{equation}
Hence combining the continuity theorem \cite[Theorem B.10]{liu2023edge} yields \eqref{supertransition}.

Now consider the limit \(\gamma \to 0\). By Lemma \ref{lem:crossover_analysis},
\begin{equation}
    \lim\limits_{\gamma\rightarrow0}\gamma^{|E|-|V|+1}\mathcal{F}_{\Gamma}^{\mathrm{(Crit)}}(t_1,\ldots,t_s;\gamma^{\alpha})=\mathcal{F}_{\Gamma}^{\mathrm{(Sub)}}(t_1,\ldots,t_s).
\end{equation}
Thus we have
\begin{equation}
    \begin{aligned}
        &\lim\limits_{\gamma\rightarrow 0}\int_{\mathbb{R}^k} \prod_{j=1}^k \frac{\sin^{m} (t_j\sqrt{-x_j})}{(t_j\sqrt{-x_j})^{m}} \,\gamma^k R_{k}^{\mathrm{(Crit)}}(\gamma; dx_1, \dots, dx_k)\\
        &=\lim\limits_{\gamma\rightarrow 0}\gamma^k\sum_{\Pi \in \mathscr{P}([k])} \prod_{\pi \in \Pi} \tau_{\gamma}(\pi).
    \end{aligned}
\end{equation}
By Lemma \ref{lem:crossover_analysis}, we know that only the term of partition $\pi=\{\{1\},\{2\},\ldots,\{k\}\}$ does not vanish. This implies
\begin{equation}
    \lim\limits_{\gamma\rightarrow 0}\int_{\mathbb{R}^k} \prod_{j=1}^k \frac{\sin^{m} (t_j\sqrt{-x_j})}{(t_j\sqrt{-x_j})^{m}} \,\gamma^k R_{k}^{\mathrm{(Crit)}}(\gamma; dx_1, \dots, dx_k)=\prod_{j=1}^k\int_{\mathbb{R}} \frac{\sin^{m} (t_j\sqrt{-x_j})}{(t_j\sqrt{-x_j})^{m}}R_{1}^{\mathrm{(Sub)}}(dx_j).
\end{equation}
Hence combining the continuity theorem \cite[Theorem B.10]{liu2023edge} yields \eqref{subtransition}.

This finally completes   the proof.

\end{proof}

\subsection{Tricritical analysis}\label{sec:doubly_critical}

In this subsection, we consider deformed random band matrices in the triply critical regime, where the bandwidth scaling, the perturbation strength, and the position of the perturbation all compete. For simplicity, we consider a finite-rank perturbation $A_{\mathcal{M}}$, supported on a set of coordinates
 $\mathcal{M}=\{m_1, m_2, \dots, m_r\} \subset [N]$,
\begin{equation}
    A_{\mathcal{M}} = \sum_{i,j=1}^r a_{ij} E_{m_i m_j},
\end{equation}
where $E_{m_im_j}$ is the matrix with entry $1$ at $(m_i,m_j)$ and   zeros elsewhere.
We also denote the reduced $r \times r$ matrix and its associated spectral decomposition by
\begin{equation}
    \widetilde{A}=\sum_{i,j=1}^r a_{ij} E_{ij}=\sum_{i=1}^r a_i \boldsymbol{v}_i \boldsymbol{v}_i^T,
\end{equation}
 where $\boldsymbol{v}_i$ is a column eigenvector of size $r$ associated with eigenvalue $a_i$. We assume the top $q$ eigenvalues of $\widetilde{A}$ are critically scaled near the critical value and the remaining eigenvalues stay away,    i.e.,
\begin{equation} \label{perturbationscaling}
    a_i = 1 + \tau_i W^{-\frac{\alpha}{3\alpha-1}},  i=1, \dots, q,  \quad \mathrm{and} \quad|a_j| \le 1 - \epsilon_0, j=q+1, \dots, r,
\end{equation}
where $\epsilon_0 \in (0,1)$  and all $\tau_i\in  \mathbb{R}$   are   fixed numbers.

  With the above notations,  we introduce the concept of deformed limiting Diagram Function in the triply critical regime.
\begin{definition}[Deformed Limiting Diagram Function]\label{def:F-double-critical}
For    an  $r\times r$ matrix
\begin{equation}
    \mathfrak{A}(t) := \sum_{i=1}^q e^{\tau_i t} \boldsymbol{v}_i \boldsymbol{v}_i^T, \quad t \ge 0,
\end{equation}
we   define a quadratic form  in real variables $\{z_i\}$ by
\begin{equation}
    \mathfrak{A}[t;\{z_i\}_{i=1}^r](x,y)=\sum_{i,j=1}^r (\mathfrak{A}(t))_{ij}\sqrt{\delta(x-z_i)}\sqrt{\delta(y-z_j)}.
\end{equation}
For a typical connected diagram $\Gamma = (V, E)$,  we define the weight function  $\mathcal{W}_\Gamma$ as
\begin{equation}
    \mathcal{W}_\Gamma(x, \{\alpha_e\}) = \prod_{e=(u,v) \in E_{\mathrm{int}}} \theta_{\alpha}(x_u - x_v, \alpha_e \gamma^\alpha) \prod_{e=(u,v) \in E_b} \mathfrak{A}[\alpha_e;\{z_i\}_{i=1}^r](x_u,x_v),
\end{equation}
and the corresponding limiting diagram function $\mathcal{F}^{\mathfrak{A}}_\Gamma$   for  $t_1, \dots, t_s > 0$ as the integral
\begin{equation}\label{eq:F-double-critical_refined}
    \mathcal{F}^{\mathfrak{A}}_\Gamma(t_1, \dots, t_s; \{z_i\}_{i=1}^r;\gamma^\alpha) =
    \frac{C_\Gamma}{\prod_{j=1}^s t_j}  \int_{\mathbb{T}^{|V|}} \int_{\mathfrak{C}(\{t_i\})} \mathcal{W}_\Gamma(x, \{\alpha_e\}) \prod_{e \in E} d\alpha_e \prod_{v \in V} dx_v.
\end{equation}
Here in the product $\prod_{e=(u,v) \in E_b}$, each vertex  $u\in V_b$ carries   the variable $x_u$   exactly twice and   the square-root delta functions are interpreted according to the rule
\begin{equation}
    \sqrt{\delta(x_u-z_1)}\sqrt{\delta(x_u-z_2)}=\begin{cases}
        \delta(x_u-z_1),&z_1=z_2,\\
        0,&z_1\ne z_2.
    \end{cases}
\end{equation}
  \end{definition}

\begin{definition}[Tricritical point process]\label{def:triple_crit_process}
For  any fixed number  $\gamma > 0$ and  a perturbation operator $\mathfrak{A}$, we define  a \textbf{tricritical point process}  as  the random point process on $\mathbb{R}$ whose $k$-point correlation measures $R_{k}^{\mathfrak{A}}(\gamma; dx_1, \dots, dx_k)$ are the unique Radon measures determined by the  sinc-transform identity: for any $t_1, \dots, t_k > 0$ and $m \in \{4, 8, 10\}$,
\begin{equation}\label{equ:def_R_triple}
    \int_{\mathbb{R}^k} \prod_{j=1}^k \frac{\sin^{m}(t_j \sqrt{-x_j})}{(t_j \sqrt{-x_j})^{m}} \, R_{k}^{\mathfrak{A}}(\gamma; dx_1, \dots, dx_k) = \sum_{\Pi \in \mathscr{P}([k])} \prod_{\pi \in \Pi} \tau_{\mathfrak{A}}(\{t_j\}_{j \in \pi}).
\end{equation}
Here the deformed limiting cumulants $\tau_{\mathfrak{A}}$ are defined by the sum over connected typical diagrams:
\begin{equation}
    \tau_{\mathfrak{A}}(\{t_j\}_{j \in \pi}) = \sum_{\Gamma \in \mathcal{T}_{|\pi|}} \gamma^{|E|-|V|} \int_{0}^1 \cdots \int_{0}^1 \left[ \prod_{j \in \pi} \mathcal{Q}_m(\xi_j) \right] \mathcal{F}^{\mathfrak{A}}_{\Gamma}\bigl(\{\xi_j t_j\}_{j \in \pi}; \{z_i\}_{i=1}^r; \gamma^\alpha\bigr) \prod_{j \in \pi} d\xi_j,
\end{equation} with $\mathcal{Q}_m$  defined in \eqref{eq:Pt_def}.
\end{definition}

\begin{remark}
The term \textbf{tricritical point} emphasizes the confluence of three distinct physical scales that govern the local eigenvalue statistics:
\begin{enumerate}
    \item[(i)] The \textbf{spectral edge} scale, establishing the base for local density analysis;
    \item[(ii)] The \textbf{critical bandwidth} $W \sim N^{1-\frac{1}{3\alpha}}$, which interpolates the transition between Airy (determinantal) and Poisson (independent) statistics;
    \item[(iii)] The \textbf{critical spike strength} $a_i = 1 + \tau_i W^{-\frac{\alpha}{3\alpha-1}}$ ($i=1,\ldots,q$), which determines the birth or absorption of outliers relative to the bulk edge.
\end{enumerate}
This process provides a unified framework that interpolates between the BBP transition and the pure variance-profile transition statistics.
\end{remark}

We are now ready to state our result in the tricritical regime.
\begin{theorem}[Tricritical edge statistics] \label{thm:triple_critical_convergence}
For the $\alpha$-stable random band  matrices   in Definition \ref{defmodel} but with Gaussian entries, subject to a finite-rank perturbation $A_{\mathcal{M}}$ with $\mathcal{M}=\{m_1,\ldots,m_r\}$ and with critically scaled eigenvalues  in \eqref{perturbationscaling}.
Assume  that $\sum_{i=1}^s n_i$ is even
and the parameters $(n_i, W, N)$ satisfy
\begin{itemize}
    \item $W \sim (\gamma N)^{1-\frac{1}{3\alpha}}$ for some fixed constant $\gamma > 0$;
    \item $n_i \sim t_i W^{\frac{\alpha}{3\alpha-1}}$ for all  $t_i \in (0, \infty)$.
    \item $m_i\sim z_iN$ for all $z_i\in [0,1)$.
\end{itemize}
Then, as $N \to \infty$, the following asymptotic properties hold:
\begin{enumerate}
    \item[(i)] \textbf{(Deformed diagram asymptotics)} For any given typical connected diagram $\Gamma \in \mathcal{T}_s$,
    \begin{equation}\label{equ:f_triple_critical}
        \frac{1}{n_1\cdots n_s}F^{\mathfrak{A}}_\Gamma(\{n_i\}_{i=1}^s) = (1+o(1)) \gamma^{|E|-|V|} \mathcal{F}^{\mathfrak{A}}_\Gamma(t_1, \dots, t_s; \{z_i\}_{i=1}^r; \gamma^\alpha),
    \end{equation}
    where $\mathcal{F}^{\mathfrak{A}}_\Gamma$ is given  in Definition \ref{def:F-double-critical}.
  \item[(ii)] \textbf{(Convergence to tricritical point process)} With   $s_N = W^{-\frac{2\alpha}{3\alpha-1}}$, the rescaled $k$-point correlation measure
    \begin{equation}
        (T_{s_N})_*R_{N;k}\bigl(dx_1, \ldots, dx_k\bigr)
    \end{equation}
    converges vaguely to the \textbf{tricritical point process} $R_{k}^{\mathfrak{A}}(\gamma; dy_1, \dots, dy_k)$.

    \item[(iii)] \textbf{(Crossover and Consistency)} The limiting measure $R_{k}^{\mathfrak{A}}$
    provides a unified interpolation: as $\tau_i \to -\infty$, it recovers the non-deformed critical measure $R_{k}^{\mathrm{(Crit)}}(\gamma; \cdot)$ of Theorem \ref{thm:convergence_proof}; as $\gamma \to \infty$, it converges to the universal BBP transition statistics under appropriate scaling of the parameters $\tau_i$; as $\gamma\rightarrow-\infty$, $\gamma^kR_{k}^{\mathfrak{A}}$ converges to $\prod_{j=1}^k R_1^{\mathrm{(Sub)}}(dx_j)$.
\end{enumerate}
\end{theorem}

\begin{proof}

    Part (i). The proof of (i) is structurally analogous to that of the critical case in Theorem \ref{thm:homogeneous_theorem}, with necessary modifications to accommodate the boundary edges and vertices introduced by the perturbation. We break down the derivations of the new terms as follows:

    First, regarding the origin of the exponential factor $e^{\tau_i \alpha_e}$: In the expansion of the discrete diagram function, each boundary edge $e \in E_b$ contributes a matrix weight governed by $\widetilde{A}^{w_e}$. Under the critical scaling, the eigenvalues of the perturbation are $a_i = 1 + \tau_i W^{-\frac{\alpha}{3\alpha-1}}$, and the edge summation variable scales as $w_e = \alpha_e W^{\frac{\alpha}{3\alpha-1}}$ (identically to the substitution made in the critical regime proof of Theorem \ref{thm:homogeneous_theorem}). Taking the limit $W \to \infty$, the spectral contribution of the $i$-th critical spike asymptotically behaves as:
    \begin{equation}
        a_i^{w_e} = \left(1 + \tau_i W^{-\frac{\alpha}{3\alpha-1}}\right)^{\alpha_e W^{\frac{\alpha}{3\alpha-1}}} = e^{\tau_i \alpha_e} \big(1 + o(1)\big).
    \end{equation}
    This continuous exponential limit naturally yields the continuous-time operator $\mathfrak{A}(\alpha_e) = \sum_{i=1}^q e^{\tau_i \alpha_e} \boldsymbol{v}_i \boldsymbol{v}_i^T$.

    Second, regarding the appearance of the delta functions and the condition $m_i \sim z_i N$: In the discrete sum, the spatial coordinates of any boundary vertex $u \in V_b$ are strictly restricted to the perturbation index set $\mathcal{M} = \{m_1, \dots, m_r\}$. When we transition from the discrete Riemann sum to the continuous integral over the torus $\mathbb{T}$ (as executed in the critical case of Theorem \ref{thm:homogeneous_theorem}), the sum over $\eta(u)$ is replaced by an integral over $x_u = \eta(u)/N$. Because the discrete spike positions scale macroscopically as $m_i \sim z_i N$, their continuous coordinates on the torus tightly concentrate at $z_i \in [0,1)$. Consequently, the discrete spatial restriction to $\mathcal{M}$ transforms rigorously into an integration against the Dirac delta distributions $\delta(x_u - z_i)$ in the continuum limit. This justifies the integral kernel $\mathfrak{A}[t;\{z_i\}_{i=1}^r](x_u,x_v)$ defined via $\sqrt{\delta(x_u-z_i)}\sqrt{\delta(x_v-z_j)}$.

    Finally, separating the graph into interior and boundary components replaces the global scaling exponent $\gamma^{|E|-|V|}$ with $\gamma^{|E_{\mathrm{int}}|-|V_{\mathrm{int}}|}$. Since the boundary structure consists of tree-like attachments to the interior vertices, the topological identity $|E_{\mathrm{int}}|-|V_{\mathrm{int}}| = |E|-|V|$ holds, preserving the exact $\gamma$-scaling.

    Part (ii). The convergence to the tricritical point process follows the same recursive mixed-moment argument as established in the proof of Theorem \ref{thm:convergence_proof}, replacing the critical diagram function with the deformed diagram function $\mathcal{F}^{\mathfrak{A}}_\Gamma$.

    Part (iii). The crossover consistency follows the identical continuity and decoupling logic applied in the proof of Corollary \ref{coro:poisson_statistics}. When $\tau_i \to -\infty$, the spike operator $\mathfrak{A}$ vanishes, degenerating to the unperturbed critical measure.
\end{proof}

Finally, by combining the results of \cite{geng2024outliers, liu2025edge}, we can now present the phase diagram for the edge statistics of deformed random band matrices, as summarized in Table~\ref{table:phase_diagram}.

\begin{table}[htbp]
\centering
\caption{Edge statistics for spiked $1$D random band matrices
}
\label{table:phase_diagram}
\label{tab:final_phase_diagram}
\renewcommand{\arraystretch}{1.8}
\begin{tabularx}{\textwidth}{@{} >{\bfseries}l >{\centering\arraybackslash}X >{\centering\arraybackslash}X @{} p{4cm} @{}}
\toprule
Bandwidth $\backslash$ Spike & Weak signal
$a < 1$ & Intermediate signal $a = 1$ & \centering\arraybackslash Strong signal
$a > 1$ \cr
\midrule

\textbf{Subcritical} ~ $W \ll N^{1-\frac{1}{3\alpha}}$ &
\multicolumn{2}{c}{Poisson Statistics, Theorem  \ref{thm:convergence_proof} and \ref{thm:triple_critical_convergence}
}  &
 \cr
\cmidrule(lr){1-3}

\textbf{Critical} ~ $W = (\gamma N)^{1-\frac{1}{3\alpha}}$ &
{Transition Statistics,} Theorem \ref{thm:convergence_proof} & \textbf{Tricritical phase,
} Theorem \ref{thm:triple_critical_convergence}
& \multirow{3}{4cm}{\centering \textbf{Non-universal Outliers} \cite{geng2024outliers} \\ \vspace{2pt}
 }\cr
\cmidrule(lr){1-3}

\textbf{Supercritical} ~ $W \gg N^{1-\frac{1}{3\alpha}}$ &  Airy Statistics \cite{liu2025edge}&   BBP Transition \cite{liu2025edge}& \cr
\bottomrule
\end{tabularx}
\end{table}

\section{Several applications}\label{sec:application}
In this section, we present several applications of Theorem~\ref{prop:super_4}, focusing primarily on the identification and analysis of prototypical toy models. Unless otherwise indicated, we restrict our attention to the non-deformed setting.

\subsection{Wegner orbital model}\label{sec:wegner_model}

We consider a generalization of the Wegner orbital model from one dimension
to an arbitrary dimension \(d \ge 1\) on the discrete torus
\(\mathbb{T}_D^d = (\mathbb{Z}/D\mathbb{Z})^d\), where \(D \ge 2\) is an integer.
This model interpolates between localized and delocalized phases
\cite{wegner1979disordered,MR3915294}.
Similar high-dimensional block models have been studied in
\cite{fan2025localization,stone2025random,yang2025delocalization,khang2025localization}.

\begin{definition}[Wegner orbital model] \label{wegner}
Define a block random Hermitian matrix by
\begin{equation}\label{equ:Wegner_model}
X = \sqrt{1-\lambda}\,H + \sqrt{\lambda}\,\Lambda,
\end{equation}
where $\lambda\in[0,1)$ may depend on $M$,
for any $i, j \in \mathbb{T}_D^d$, the $(i, j)$-th block of size $M\times M$ is given by
\begin{equation}
    (H)_{i,j} = \delta_{i,j} H_i,
\end{equation}
and
\begin{equation}
    (\Lambda)_{i,j} = \begin{cases}
        \frac{1}{\sqrt{2d}}A_{i,j}, & \text{if } |i-j|_1 = 1, \\
        0, & \text{otherwise}.
    \end{cases}
\end{equation}
Here $|i-j|_1$ denotes the $\ell^1$ distance on the discrete torus (nearest neighbors), and  all non-zero block matrices are independent up to  symmetry  such that diagonal blocks $\{H_i\}$ are   rescaled $M\times M$ GUE (or GOE) matrices and
 off-diagonal blocks $A_{i,j}$ are  rescaled $M\times M$  Ginibre matrices.
\end{definition}

For any two indices $x = (i, a)$ and $y = (j, b)$, where $i, j \in \mathbb{T}_D^d$ represent the block coordinates on the torus and $a, b \in \{1, \dots, M\}$ denote the internal indices within each block, the variance entries are given by:
\begin{equation}
\sigma_{(i,a), (j,b)}^2 = \mathbb{E}[|X_{(i,a), (j,b)}|^2] =
\begin{cases}
\frac{1-\lambda}{M}, & \text{if } i = j, \\
\frac{\lambda}{2dM}, & \text{if } |i-j|_1 = 1, \\
0, & \text{otherwise}.
\end{cases}
\end{equation}
Consequently, the associated   transition matrix $P_N = (\sigma_{xy}^2)$ can be decomposed as $P_N = \mathcal{P}_D \otimes \frac{1}{M}\mathbf{J}_M$, where $\mathbf{J}_M$ is the all-ones matrix (in the sense of entry-wise variance) and $\mathcal{P}_D$ is the transition matrix of a random walk on $\mathbb{T}_D^d$. Specifically, this random walk stays at its current block with probability $1-\lambda$ and jumps to any of its $2d$ nearest neighbors with probability $\frac{\lambda}{2d}$.

The parameter $\lambda$ governs the coupling strength between blocks, leading to
the following structural interpretations:
\begin{itemize}
    \item $\lambda = 0$: The matrix $X$ consists of  $D^d$   independent GUE/GOE blocks.
    \item $\lambda = 1/2$: The model reduces to a block-structured random band matrix over the $d$-dimensional lattice (with unit bandwidth $W = 1$ in block units).
    \item $\lambda = 1$: The model coincides with the block adjacency matrix of the $d$-dimensional torus, interpretable as an Anderson-type model featuring block disorder and vanishing on-site potential.
\end{itemize}

To describe   the limiting diagram functions in the critical regime, we introduce, for any $\tau > 0$, the periodic \textbf{Skellam transition probability} on the torus $\mathbb{T}_D^d$, defined by
\begin{equation}\label{equ:skellam}
p^{\mathrm{(Skellam)}}(\boldsymbol{k},\tau)
:= \sum_{\boldsymbol{n} = (n_1, \ldots, n_d) \in \mathbb{Z}^d} \prod_{j=1}^d \left( e^{-\frac{2\tau}{d}} I_{|k_j + n_j D|}\!\left(\frac{2\tau}{d}\right) \right),
\end{equation}
where $\boldsymbol{k} = (k_1, \ldots, k_d) \in \mathbb{Z}^d$ and $I_{\nu}(x)$ denotes the modified Bessel function of the first kind.
In the one-dimensional case ($d = 1$), this expression reduces to the periodic version of the probability mass function of the \textbf{Skellam distribution}---the distribution of the difference of two independent Poisson random variables each with mean $\tau$; see, e.g., \cite{MR20750}.
\begin{theorem} \label{Wegnerlimits}
For any fixed  diagram   $\Gamma\in \mathcal{T}_s$,
        the  diagram function associated with   the Wegner model in Definition \ref{wegner},  denoted by $F^{(\mathrm{Wegner})}_{\Gamma}(\{n_i\};\lambda)$,  admits the following asymptotics      as $M\rightarrow\infty$.
        \begin{enumerate}
        \item[(i)] \textbf{Frozen regime: $M^{\frac{1}{3}}\lambda \ll 1$}. Given   any  fixed integer $D\geq 2$,  if all $n_i\le CM$ for any fixed $C>0$,  then
        \begin{equation}
            \frac{1}{n_1\cdots n_s}F^{(\mathrm{Wegner})}_{\Gamma}(\{n_i\};\lambda)=\frac{1+o(1)}{n_1\cdots n_s}F^{(\mathrm{Wegner})}_{\Gamma}(\{n_i\};0)=\frac{1+o(1)}{n_1\cdots n_s}D^dF^{(\mathrm{G\beta E})}_{\Gamma}(\{n_i\},M),
        \end{equation} where   $F^{(\mathrm{G\beta E})}_{\Gamma}(\{n_i\},M)$ denotes  the diagram function of G$\beta$E  matrix of size $M$.
        \item[(ii)] \textbf{Skellam regime: $M^{\frac{1}{3}}\lambda \to \mu \in (0,\infty)$}. Given   any  fixed integer $D\geq 2$,  if all $n_i\lambda\rightarrow t_i>0$,
        \begin{equation} \label{SkellamPP}
            \frac{1}{n_1\cdots n_s}F^{(\mathrm{Wegner})}_{\Gamma}(\{n_i\})=\frac{(1+o(1))C_\Gamma}{\,t_1\cdots t_s\,}\,\mu^{-|E|+s}\sum_{\alpha_v\in \mathbb{T}_D^d}\int_{\lambda_e:\mathfrak{C}(\{t_i\})}
    \prod_{e\in E(\Gamma)}p^{\mathrm{(Skellam)}}(\alpha_{u}-\alpha_{v},\lambda_e) \prod_{e\in E(\Gamma)}d\lambda_e.
        \end{equation}

        \item[(iii)] \textbf{Diffusive regime: $M^{\frac{1}{3}}\lambda \gg 1$}: When   $d=1$,  with the same notations, assumptions and scalings on $n_i$ as in Theorem \ref{thm:limit_cumulant} where    $\alpha=2$, $N=DM$ and the effective bandwidth  $W_{\mathrm{eff}}=\sqrt{\lambda}M$, then
\begin{equation}
    \frac{1}{n_1\cdots n_s}F^{(\mathrm{Wegner})}_{\Gamma}(\{n_i\};\lambda)=(1+o(1))\frac{1}{n_1\cdots n_s}\mathcal{L}_\Gamma (\{t_i\}; (W_{\mathrm{eff}},N)).
\end{equation}
\end{enumerate}
\end{theorem}
\begin{proof}
Proofs for the frozen and diffusive regimes follow immediately from Theorem~\ref{thm:F=tildeF} combined with the Markov-chain comparison established in Proposition~\ref{prop:wegner_asymptotics}.The proof for the Skellam regime proceeds analogously to that of Theorem~\ref{thm:homogeneous_theorem}, except that the local limit theorem (Proposition~\ref{prop:alpha_llt}) is replaced by Proposition~\ref{prop:wegner_asymptotics}.
\end{proof}
\begin{definition}[Skellam point process] \label{def:skellam_process}
The \textbf{Skellam point process} is the random point process on $\mathbb{R}$ whose $k$-point correlation measures $R^{\mathrm{(Skellam)}}_{\mu;k}(dx_1,\ldots,dx_k)$ are the unique Radon measures determined by the following Sine-transform identity. For any $t_1, \dots, t_k > 0$ and $m \in \{4, 8, 10\}$,
\begin{equation} \label{equ:def_R_Skellam}
    \int_{\mathbb{R}^k} \prod_{j=1}^k \frac{\sin^{m} (t_j\sqrt{-x_j})}{(t_j\sqrt{-x_j})^{m}} \, R^{\mathrm{(Skellam)}}_{\mu;k}(dx_1,\ldots,dx_k) = \sum_{\Pi \in \mathscr{P}([k])} \prod_{\pi \in \Pi} \tau_{\mathrm{(Skellam)}}(\{t_j\}_{j \in \pi}),
\end{equation}
where the deformed limiting cumulants $\tau_{\mathfrak{A}}$ are defined by the sum over connected typical diagrams:
\begin{equation*}
    \tau_{\mathrm{(Skellam)}}(\{t_j\}_{j \in \pi}) = \sum_{\Gamma \in \mathcal{T}_{|\pi|}} \int_{0}^1 \cdots \int_{0}^1 \left[ \prod_{j \in \pi} \mathcal{Q}_m(\xi_j) \right] \mathcal{F}^{\mathrm{(Skellam)}}_{\Gamma}(\{\xi_j t_j\}_{j \in \pi}; \lambda) \prod_{j \in \pi} d\xi_j.
\end{equation*}
Here $\mathcal{F}^{\mathrm{(Skellam)}}_{\Gamma}$ is the deformed limiting diagram function defined as
\begin{equation}
    \mathcal{F}^{\mathrm{(Skellam)}}_{\Gamma}(\{t_j\}_{j \in \pi}; \mu)=\mu^{-|E|+s} \sum_{\alpha_v\in \mathbb{T}_D^d}\int_{\lambda_e:\mathfrak{C}(\{t_i\})}
    \prod_{e\in E(\Gamma)}p^{\mathrm{(Skellam)}}(\alpha_{u}-\alpha_{v},\lambda_e) \prod_{e\in E(\Gamma)}d\lambda_e.
\end{equation}
\end{definition}

The analogue of Theorem \ref{thm:convergence_proof} for the block Wegner orbital model is stated as follows.
\begin{corollary}\label{coro:Wegner_orbital}
\begin{enumerate}
    \item[(i)] \textbf{Frozen regime:} Given   any  fixed integer $D\geq 2$,  if
         \(M^{\frac13}\lambda \ll 1\), then for any fixed integer  $k\geq 1$ with $s_N=M^{-2/3}$,
       $ (T_{s_N})_*  R_{N,\lambda}\bigl(dx_1,\dots,dx_k\bigr)$
        converges vaguely to the corresponding limits  of $D$ independent GUE/GOE matrices.

    \item[(ii)] \textbf{Skellam regime:} Given   any  fixed integer $D\geq 2$,  if
         \(M^{\frac13}\lambda \to \mu \in (0,\infty)\) with $s_N=M^{-2/3}$, then
       $ (T_{s_N})_*  R_{N,\lambda}\bigl(dx_1,\dots,dx_k\bigr)$converges vaguely to a new point process \(R^{\mathrm{(Skellam)}}_{\mu;k}(dy_1,\ldots,dy_k)\) defined by definition \ref{def:skellam_process} that interpolates between Airy and Poisson  statistics.

    \item[(iii)] \textbf{Diffusive regime:}  If $d=1$, \(M^{\frac13}\lambda \gg 1\),
       then  the $k$-point correlation measure  converges
       to those of the corresponding $1d$ 2-stable random band matrix with effective bandwidth $W=\sqrt{\lambda}M$ and $N=DM$.
\end{enumerate}
\end{corollary}

\begin{proof}
The proof follows the same line of argument as that of Theorem~\ref{thm:convergence_proof}.
The finiteness of
$D$  is required to guarantee that the limit  corresponding to   \eqref{equ:limit_P} is finite.
\end{proof}

\subsection{Hankel-profile random matrices}
\label{sec:hankel}

 The model of random band matrices studied in Section~\ref{sectionrbm} is indeed an inhomogeneous random matrix   whose variance profile $\Sigma_N = (\sigma_{xy}^2)$ possesses a Toeplitz structure: $\sigma_{xy}^2$ depends only on the relative distance between the indices $x$ and $y$. In contrast, we define a \textbf{Hankel-profile random matrix} as an IRM matrix  whose variance profile matrix exhibits a Hankel structure; equivalently, $\sigma_{xy}^2$ depends only on the sum of the indices $x$ and $y$ modulo $L$. Such a structure arises naturally in systems with reflection symmetry or under specific spatial constraints on the torus.

\begin{definition}[Hankel-profile random matrices]
\label{def:hankel_irm}
Let $\Lambda_L = \mathbb{T}_L^d$ be the $d$-dimensional discrete torus. An inhomogeneous random matrix $X_N$ is said to be of  \textbf{Hankel-type profile} if its variance profile $\Sigma_N = (\sigma_{xy})$ satisfies
\begin{equation}
\sigma_{xy}^2 = \frac{1}{M} \sum_{k \in \mathbb{Z}^d} f\left(\frac{x+y-x_0+kL}{W}\right), \quad  M:=\sum_{k\in \mathbb{Z}^d} f\big(\frac{k}{W}\big),
\end{equation}
where $f$ is a non-negative integrable and  symmetric function, and  $x_0 \in \mathbb{T}_L^d$ is a fixed center of reflection.
\end{definition}

The fundamental distinction between Toeplitz and Hankel profiles lies in the dynamics of the underlying Markov chain $([N], P_N)$. While a Toeplitz profile induces a standard random walk, recursively defined by $Y_{n+1} = Y_n + \xi_{n+1} \pmod L$, the Hankel-type transition $p(x, y)$
induces an \textbf{alternating walk}:
\begin{equation}
\label{eq:alternating_walk}
Y_{n+1} = -Y_n + x_0 + \xi_{n+1} \pmod L,
\end{equation}
where $\{\xi_i\}$ are i.i.d. random variables governed by the density $f$. Iterating \eqref{eq:alternating_walk}, the state after $n$ steps starting from $Y_0$ is given by
\begin{equation}
Y_n = (-1)^n Y_0 + \frac{1-(-1)^n}{2}x_0 + \sum_{j=1}^n (-1)^{n-j} \xi_j \pmod L.
\end{equation}
For an even number of steps $n=2k$, the state is $Y_{2k} = Y_0 + \sum_{j=1}^{k} (\xi_{2j} - \xi_{2j-1})$, which reduces to a standard random walk with i.i.d. symmetric increments $\zeta_j = \xi_{2j} - \xi_{2j-1}$. Consequently, the local central limit theorem (Theorem \ref{prop:alpha_llt}, Lemma \ref{lem:asy_theta} and Proposition~\ref{thm:clt-upper}) directly applies.

For odd steps $n=2k+1$, the distribution of $Y_{2k+1}$ is centered around $x_0 - Y_0$, and its behavior depends more sensitively on the potential asymmetry of $f$. For simplicity, we assume $f$ is an $\alpha$-stable profile centered at the origin.

\begin{lemma}[LCLT for Alternating Walks]\label{lem:hankel_clt} For
  a Hankel-profile random matrix  $X_N$ with an $\alpha$-stable profile $f$ as in Definition~\ref{def:alpha_profile},  let $p_n(x,y)$ be the $n$-step transition probability on $\mathbb{T}_L^d$. Suppose that  $W \ll L$ and   $n (W/L)^\alpha \to \tau$ as $N \to \infty$, the following asymptotics hold.
\begin{enumerate}
    \item[(i)] \textbf{Even steps ($n=2k$):} The transition probability is concentrated around $y = x$,
    \begin{equation}
        p_{2k}(x, y) = \frac{1}{L^d} \theta_\alpha\Big( \frac{y-x}{L}, 2k \big(\frac{W}{L}\big)^\alpha \Big) + O(k e^{-c W^\alpha}).
    \end{equation}
    \item[(ii)] \textbf{Odd steps ($n=2k+1$):} The transition probability is concentrated around $y = x_0 - x$,
    \begin{equation}
        p_{2k+1}(x, y) = \frac{1}{L^d} \theta_\alpha\Big( \frac{y+x-x_0}{L}, (2k+1) \big(\frac{W}{L}\big)^\alpha \Big) + O(k e^{-c W^\alpha}).
    \end{equation}
\end{enumerate}
\end{lemma}

\begin{proof}  Write the discrete Fourier transform as
  \begin{equation}\widehat{\psi}(k) = \sum_{x\in\Lambda_L} \psi(x) e^{-\frac{2\pi i k\cdot x}{L}},\quad k\in\Lambda_L,\end{equation}  we will employ Fourier analysis on the torus $\Lambda_L$. The Hankel operator $P$ acts as $(P\psi)(x) = \sum_{y\in\Lambda_L} p(x,y)\psi(y)$, then    the Hankel structure at  $x+y-x_0$ gives
\begin{equation}
\widehat{P\psi}(k) = \widehat{p}(k) e^{-\frac{2\pi i k\cdot x_0}{L}} \widehat{\psi}(-k).
\end{equation}
Here we denote
\begin{equation}
    \widehat{p}(k):=\sum_{y\in\Lambda_L} p(x,y) e^{- \frac{2\pi i k\cdot (x+y-x_0)}{L}}=\sum_{y\in\Lambda_L} p(0,y) e^{-\frac{2\pi i k\cdot (y-x_0)}{L}}.
\end{equation}
Note that  $\widehat{p}(k) = \frac{W^d}{M} \widehat{f}\bigl(\frac{Wk}{L}\bigr) + O(e^{-c W^\alpha})$, with $\widehat{f}$ the continuous Fourier transform of $f$.  Using    the frequency reflection $k \to -k$ and  applying $P$ twice yields
\begin{equation}
\widehat{P^2\psi}(k) = \widehat{p}(k) e^{-\frac{2\pi i k\cdot x_0}{L}} \widehat{P\psi}(-k) = \widehat{p}(k)\widehat{p}(-k) \widehat{\psi}(k).
\end{equation}
Since we assume that $f$ is symmetric, $\widehat{p}(k)=\widehat{p}(-k)$,   $\widehat{P^{2m}\psi}(k) = (\widehat{p}(k))^{2m} \widehat{\psi}(k)$, proving part  (i) via the local central limit theorem for normal random walks (Theorem \ref{prop:alpha_llt}).

For odd steps, the $(2m+1)$-step transition probability can be expressed via inverse Fourier transform as
\begin{equation}
p_{2m+1}(x,y) = \frac{1}{L^d} \sum_{\ell\in\Lambda_L} (\widehat{p}(\ell))^{2m+1} e^{-\frac{2\pi i \ell\cdot x_0}{L}} e^{\frac{2\pi i \ell\cdot (x+y)}{L}}.
\end{equation}
Under the $\alpha$-stable scaling $\widehat{f}(\xi) \sim e^{-c|\xi|^\alpha}$, the sum converges to the $\theta_\alpha$ function evaluated at the rescaled coordinate $(x+y-x_0)/L$. This yields concentration around the reflection point $y = x_0 - x$, completing the proof of part (ii).
\end{proof}
By applying the local central limit theorem (Lemma \ref{lem:hankel_clt}), the corresponding limit for the $k$-point correlation measures in Hankel-profile case can be established through an argument analogous to that of Theorem \ref{thm:convergence_proof}.

\subsection{Deviation inequalities}
In the classical mean-field setting, the edge fluctuations of random matrices obey the Tracy--Widom law with scaling $N^{-2/3}$.
For inhomogeneous Gaussian matrices, Brailovskaya and  van Handel~\cite{brailovskaya2024extremal} established a deviation scale of $(\max_{ij}\sigma_{ij})^{4/3}$,
whereas our previous paper \cite{liu2025edge} demonstrated that a \emph{short-time average mixing} condition restores the $N^{-2/3}$ scaling.
Nevertheless, as observed by Sodin~\cite{sodin2010spectral} for random band matrices with unimodular entries,  the fluctuation scale may deviate substantially from $N^{-2/3}$, depending on the geometry of the variance profile.

The following theorem provides a general deviation bound for inhomogeneous random matrices, which serves as the foundation for exploring non-standard fluctuation scales.

\begin{theorem}[Deviation inequalities]\label{thm:deviation}
Let $X$ be a Gaussian IRM   matrix with mean matrix  $A_N=0$. Under Assumption~\ref{itm:B1}, there exist absolute constants $C, c > 0$ such that for any $t > 0$ and any integer  $n \geq 1$,
\begin{equation}
    \mathbb{P}\big(\|X\|\ge 2 + t\big) \le C n N b_{n} \,e^{C n^2 b_n - c n \sqrt{t}}.
\end{equation}
\end{theorem}

\begin{proof}
By applying the Chebyshev-type trace bound and the properties of Chebyshev polynomials (see e.g., \cite[Lemma 4.6]{liu2025edge}), we have
\begin{equation}
    \mathbb{E}\Big[e^{cn\sqrt{\|X\|-2}}\Big] \le \mathbb{E}\Big[\operatorname{Tr} \big(U_{2n}\big(\frac{1}{2}X\big) + 4n \mathbf{I}\big)\Big] \le C n N b_{n} e^{C n^2 b_n}.
\end{equation}
  Applying    Markov's inequality, we get
\begin{equation}
    \mathbb{P}\big(\|X\|\ge 2 + t\big) \le e^{-cn\sqrt{t}} \mathbb{E}\Big[e^{cn\sqrt{\|X\|-2}}\Big]  \le C n N b_n \,e^{C n^2 b_n - c n \sqrt{t}}.
\end{equation}
This indeed completes the proof.
\end{proof}

Applying Theorem~\ref{thm:deviation} to the power-law band matrix model, we observe a fluctuation scale determined by the power-law index $\alpha$, which differs from the standard $N^{-2/3}$ or $(\max_{ij}\sigma_{ij})^{4/3}=W^{-2/3}$ scales.
\begin{corollary}[Power-law deviation scale]\label{cor:power_law_dev}
Let $\alpha>1$. Consider the one-dimensional random band matrix $H$  in definition \ref{defmodel} with Gaussian entries and the bandwidth $W \ll N^{1-\frac{1}{3\alpha}}$. There exist absolute constants $C_1, C_2, C_3,C_4 > 0$,  when   $y>C_4$ and $W$ satisfy the constraint $y^{\frac{\alpha}{4\alpha-2}} W^{\frac{3\alpha}{3\alpha-1}} \le C_3 N$, then
\begin{equation}\label{eq:PL_scaling_bound}
    \mathbb{P}\left(\|H\| \ge 2 + W^{-\frac{2\alpha}{3\alpha-1}}y\right) \le C_1 y^{\frac{\alpha-1}{4\alpha-2}} N W^{-\frac{2\alpha}{3\alpha-1}} e^{-C_2 y^{\frac{3\alpha-1}{4\alpha-2}}}.
\end{equation}
\end{corollary}

\begin{proof}
For power-law band matrices, the diagrammatic control function $b_n$  is bounded,
\begin{equation}
    b_n \le C \sum_{j=1}^{n} (W^{-1} j^{-\frac{1}{\alpha}} + N^{-1}) \le C' n^{1-\frac{1}{\alpha}} W^{-1} + C n N^{-1}.
\end{equation}
Under the condition $n^{1/\alpha}W \le N$, the term $W^{-1}n^{-1/\alpha}$ dominates $N^{-1}$, rendering the latter negligible. Substituting the simplified $b_n \sim n^{1-1/\alpha}W^{-1}$ into the general bound from Theorem~\ref{thm:deviation}, we obtain
\begin{equation}
    \mathbb{P}(\|H\| \ge 2 + t) \le C' n^{1-\frac{1}{\alpha}} N W^{-1} \exp\left( C n^{3-\frac{1}{\alpha}} W^{-1} - c n \sqrt{t} \right).
\end{equation}

To identify the optimal fluctuation scale, we set $t = W^{- {2\alpha}/{(3\alpha-1)}} y$.
We choose $n = \lambda W^{{\alpha}/{(3\alpha-1)}}$ and the exponent then scales as
\begin{equation}
    C n^{3-\frac{1}{\alpha}} W^{-1} - c n \sqrt{t} = C \lambda^{\frac{3\alpha-1}{\alpha}} - c \lambda \sqrt{y},
\end{equation}
where the bandwidth $W$ cancels out exactly. Optimizing over $\lambda$ yields $\lambda \sim y^{{\alpha}/{(4\alpha-2)}}$, which balances the two terms at the scale $y^{(3\alpha-1)/(4\alpha-2)}$.

Substituting this optimal $\lambda$ back into the pre-exponential factor $n^{1-1/\alpha} N W^{-1}$ gives the coefficient $C_1 y^{(\alpha-1)/(4\alpha-2)} N W^{-2\alpha/(3\alpha-1)}$. Finally, the initial assumption $n^{1/\alpha}W \le N$ translates directly into the condition $y^{\alpha/(4\alpha-2)} W^{3\alpha/(3\alpha-1)} \le C_3 N$, ensuring the validity of the power-law approximation throughout the deviation range.
\end{proof}
\begin{remark}\label{rmk:band_tail}
It is worth stressing that \eqref{eq:PL_scaling_bound} does not provide the optimal exponent for the tail probability.
For instance, in the case of random band matrices with $y>0$, \eqref{equ:limit_upper_bound} yields
\begin{equation}\label{equ:band_decay}
    \mathbb{P}\left(\|X\| \ge 2 + W^{-\frac{2\alpha}{3\alpha-1}}y\right) \le C_1 N e^{-C_2 y^{\frac{3\alpha-1}{2\alpha}}},
\end{equation}
which is conjectured to be the correct exponent.
\end{remark}

\section{Concluding remarks and further questions}\label{sec:remarks}
This paper, together with Part I \cite{liu2025edge}, presents a systematic investigation of the spectral properties of general inhomogeneous random matrix models via Markov chain comparison theorems. By incorporating the established spectral properties of several proximal models—including Toeplitz (band) profile, Hankel profile, and block  matrix ensembles—we have obtained a reasonably comprehensive understanding of the local universality and phase transition phenomena at the spectral edge. Nevertheless, deeper mechanisms governing the behavior of eigenvalues and eigenvectors in the presence of inhomogeneity remain largely unresolved. We conclude this final section with a selection of remarks and open questions.

\begin{enumerate}
    \item \textbf{A zoo of edge statistics.} Theorem \ref{prop:super_4} reveals that different Markov chains with similar limit behaviors ensure the same edge point processes of their corresponding random matrices. Our analysis   in this paper and in \cite{liu2025edge} suggests that the classical Tracy--Widom distribution, which corresponds to rapidly mixing Markov chains, is but one "attractor" in a significantly larger zoo of distributions.

    For slowly mixing Markov chains, various types of central limit theorems may apply. We propose the “One CLT, One Statistics” correspondence: \textit{Every specific Local Central Limit Theorem for the variance-profile Markov chain dictates a corresponding universal pattern  of edge statistics}, as suggested in this paper. Examples include (a) random band matrices (Section \ref{sectionrbm}); (b) the block Wegner orbital model (Section \ref{sec:wegner_model}); and (c) Hankel-profile random matrices (Section \ref{sec:hankel}).

    It is now clear how the different phases of the underlying random walk affect the limiting edge statistics. Since the central limit theorem is inherently universal, these statistics are also expected to be universal under very mild assumptions.   The remaining challenge is to characterize the central limit theorem for general Markov chains.
     \item  \textbf{Understanding the new point processes.} Using Feynman diagram expansions and the continuity theorem for the sinc transform, we have obtained many new statistics. The critical processes interpolate between Poisson and Airy point processes. However, beyond these identifications, much remains unknown. A key open question is whether these critical point processes admit more intrinsic descriptions or a deeper characterization, analogous to the Tracy–Widom law.

    \item \textbf{
    Dyson Brownian motion and  relaxation time.} Dyson Brownian motion (DBM) plays a crucial role in establishing the universality of random matrices \cite{erdos2017dynamical}. In essence, we just need to consider matrix flow defined by
     \begin{equation}
        H_{\lambda}=\sqrt{1-\lambda}H_0+\sqrt{\lambda}H_1,
    \end{equation}
     where $H_1$ represents a GOE/GUE matrix. We evaluate our comparison results against these two natural reference objects: $H_0$ and $H_1$, with   transition   matrices $P_0$ and $P_1$, respectively.
Then the transition probability matrix of $H_{\lambda}$ is given by the convex combination
    \begin{equation}
        P_{\lambda} := \lambda P_1 + (1-\lambda)P_0, \quad P_1 := J = \frac{1}{N}\mathbf{1}\mathbf{1}^T.
    \end{equation}
    Since $JP_0=P_0J=J$, the $n$-step transition matrix satisfies the exact identity:
    \begin{equation}
        P_{\lambda}^n = (1-\lambda)^n P_0^n + (1 - (1-\lambda)^n) P_1.
    \end{equation}
    Consequently, the deviations from our two reference objects are explicitly given by
    \begin{align}
        P_{\lambda}^n - P_0^n = \big(1 - (1-\lambda)^n\big) (P_1 - P_0^n),  \quad P_{\lambda}^n - P_1^n = -(1-\lambda)^n (P_1 - P_0^n).
    \end{align}

This exact algebraic decomposition beautifully captures the critical relaxation time of DBM at the spectral edge. To detect edge statistics using the moment method, the length of the random walks must scale as $n \sim N^{1/3}$ \cite{feldheim2010universality,sodin2010spectral,liu2023edge,liu2025edge}. Therefore, the phase transition between the $H_0$ regime and the GOE/GUE regime is entirely governed by the decay factor $(1-\lambda)^n \approx e^{-\lambda N^{1/3}}$, which naturally identifies $\lambda_c \sim N^{-1/3}$ as the critical time:

\begin{itemize}
    \item \textbf{Universality regime ($\lambda \gg N^{-1/3}$).} In this regime, $P_{\lambda}^n$ converges to $P_1^n$ at an exponential rate, so the memory of the initial profile is lost. The comparison with GOE/GUE using the moment method was derived as a Short-to-Long mixing condition in \cite{liu2025edge}; the GOE/GUE universality for $\lambda \gg N^{-1/3+\epsilon}$ can be established by verifying the mixing properties of the variance profile of $H_{\lambda}$. See \cite[Theorem 7.1]{liu2025edge} with $c_N=N^{-1/3+\epsilon}$ and $C_N=o(N^{\epsilon}/\log N)$. This aligns with the continuous DBM result of \cite{landon2017edge}, which shows that $\lambda=N^{-1/3+\epsilon}$ is sufficient to yield edge universality for highly general initial conditions $H_0$.

    \item \textbf{$H_0$ regime ($\lambda \ll N^{-1/3}$).} In this regime, the perturbation is too weak to induce mixing before the critical step count $n \sim N^{1/3}$ is reached. Under the assumption of \eqref{UPB}, where $n^2\frac{n}{N}\le n^2b_n=O(1)$, we naturally have $n\lambda \ll 1$. Consequently, $P_{\lambda}^n$ remains asymptotically close to $P_0^n$. By Theorem \ref{prop:super_4}, the mixed moments of $H_{\lambda}$ are equivalent to those of the unperturbed matrix $H_{0}$. To illustrate the sharpness of this transition, if we take $H_0$ to be random band matrices or block Wegner matrices in the critical regime, their distinct, non-universal edge statistics will persist as long as $\lambda \ll N^{-1/3}$. This explicitly demonstrates that $\lambda \sim N^{-1/3}$ is precisely the critical relaxation time required for Dyson Brownian motion to fully mix and erase the initial spatial structure at the spectral edge.
\end{itemize}

    \item

 \textbf{From Gumbel to Tracy-Widom.} Under the same assumption in Theorem \ref{thm:convergence_proof}, the limit distribution of  the largest eigenvalue $\lambda_{\max}$  in the  critical regime can be defined  via the
Fredholm series expression  (see e.g. \cite[Example 5.4(a)]{MR1950431})
    \begin{equation}
F^{\mathrm{(Crit)}}(\gamma;\xi):= 1
  +\sum_{k=1}^{\infty} \frac{(-1)^k}{k!}
  \int_{\xi}^{\infty} \cdots  \int_{\xi}^{\infty}
  R_{k}^{\mathrm{(Crit)}}(\gamma; dx_1, \dots, dx_k).
    \end{equation}
    According to Corollary \ref{coro:poisson_statistics}, if we could  interchange the summation and the limit (which is one of the most difficult steps), as $\gamma\rightarrow0$ we expect
    \begin{equation}
       F^{\mathrm{(Crit)}}(\gamma;\xi)\sim e^{-\frac{1}{\gamma}\int_{\xi}^{\infty}R^{(\mathrm{Sub})}_{1}(dx)}.
    \end{equation}

This indeed indicates that the largest eigenvalue follows a type of extreme-value distribution.
Due to the exponential decay, $F^{\mathrm{(Crit)}}(\gamma;\xi)$ is expected to converge to the Gumbel distribution after an appropriate rescaling of $\xi$.
Conversely, as $\gamma\rightarrow\infty$, $F^{\mathrm{(Crit)}}(\gamma;\xi)$ converges to the Tracy--Widom law.

Thus, $F^{\mathrm{(Crit)}}(\gamma;\xi)$  may interpolate between the Gumbel distribution and the Tracy--Widom distribution.
We note, however, that a more refined understanding of the critical correlation measure is required to rigorously establish the Gumbel convergence.

\item \textbf{Numerical simulations.} For symmetric random matrices, setting $\alpha=2$ and $d=1$,   we conduct numerical simulations for random band matrices. Figure \ref{fig:eigenvalue_transition} illustrates the transition of the largest eigenvalue  from the Gumbel  law to the Tracy--Widom  law as the bandwidth increases. Furthermore, Figure \ref{fig:eigenvector_simulation} demonstrates the localization--delocalization transition of the eigenvector associated with the largest eigenvalue, where the Inverse Participation Ratio (IPR) characterizes the crossover between regimes.
\begin{figure}[h]
     \centering
     \begin{subfigure}[b]{0.32\textwidth}
         \centering
         \includegraphics[width=\textwidth]{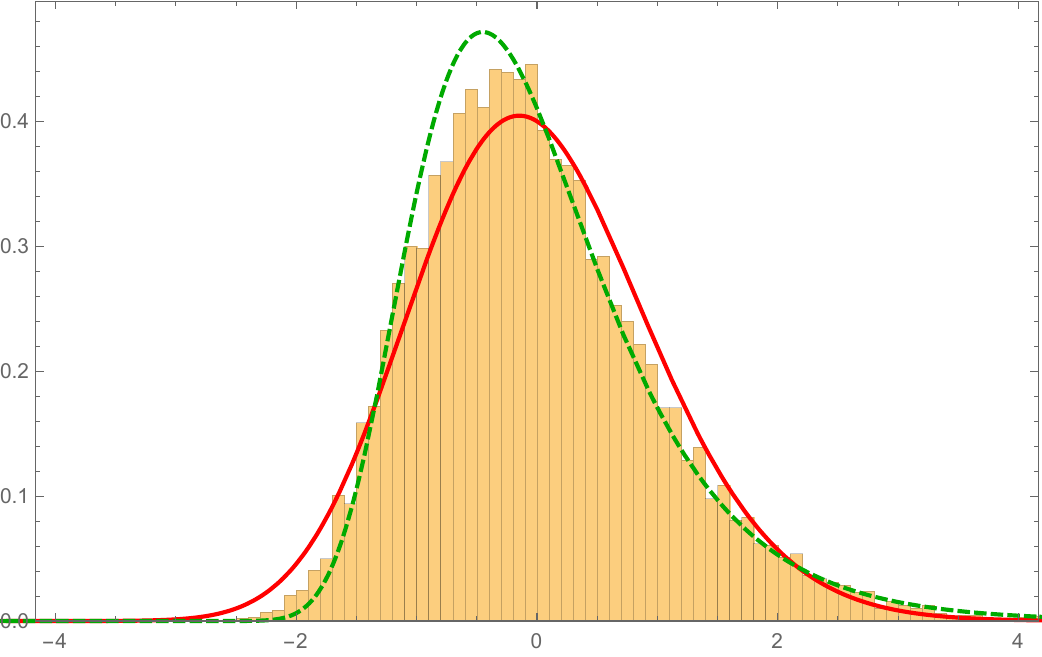}
         \caption{$W=5$ (Subcritical)}
     \end{subfigure}
     \hfill
     \begin{subfigure}[b]{0.32\textwidth}
         \centering
         \includegraphics[width=\textwidth]{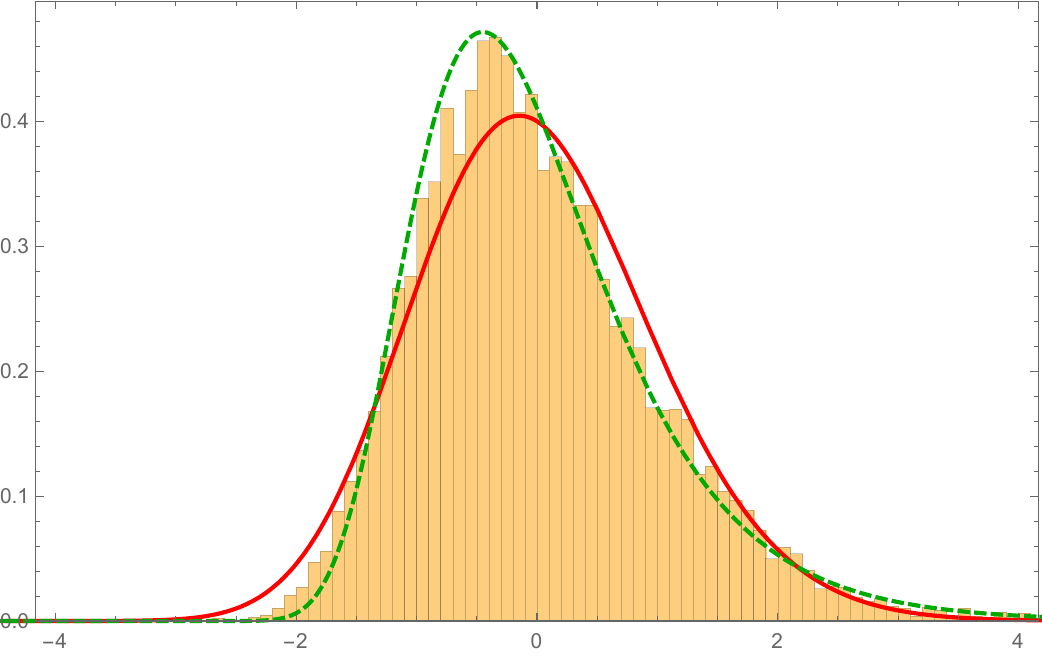}
         \caption{$W=10$ (Subcritical)}
     \end{subfigure}
     \hfill
     \begin{subfigure}[b]{0.32\textwidth}
         \centering
         \includegraphics[width=\textwidth]{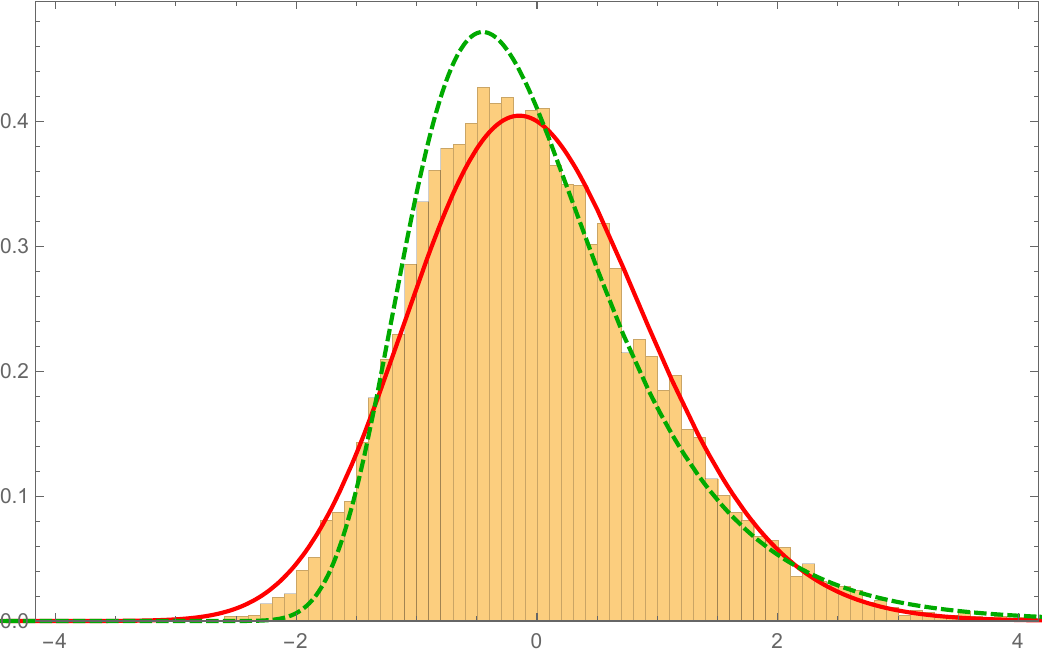}
         \caption{$W=50$ (Near Subcritical)}
     \end{subfigure}

     \vspace{0.5cm}
     \begin{subfigure}[b]{0.32\textwidth}
         \centering
         \includegraphics[width=\textwidth]{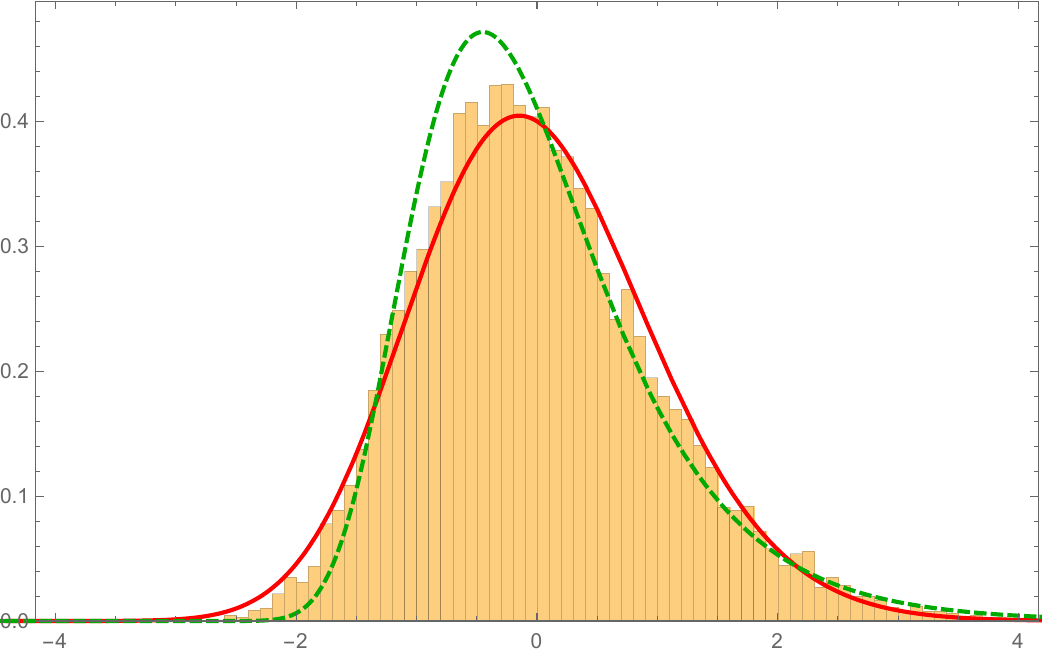}
         \caption{$W=100$ (Transition)}
     \end{subfigure}
     \hfill
     \begin{subfigure}[b]{0.32\textwidth}
         \centering
         \includegraphics[width=\textwidth]{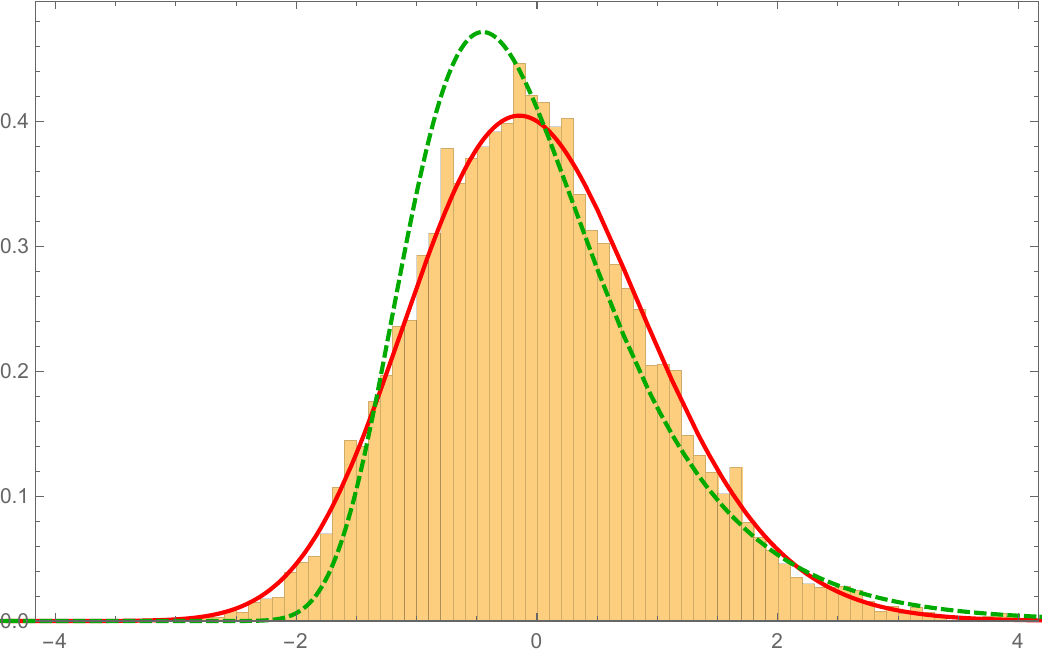}
         \caption{$W=200$ (Near Supercritical)}
     \end{subfigure}
     \hfill
     \begin{subfigure}[b]{0.32\textwidth}
         \centering
         \includegraphics[width=\textwidth]{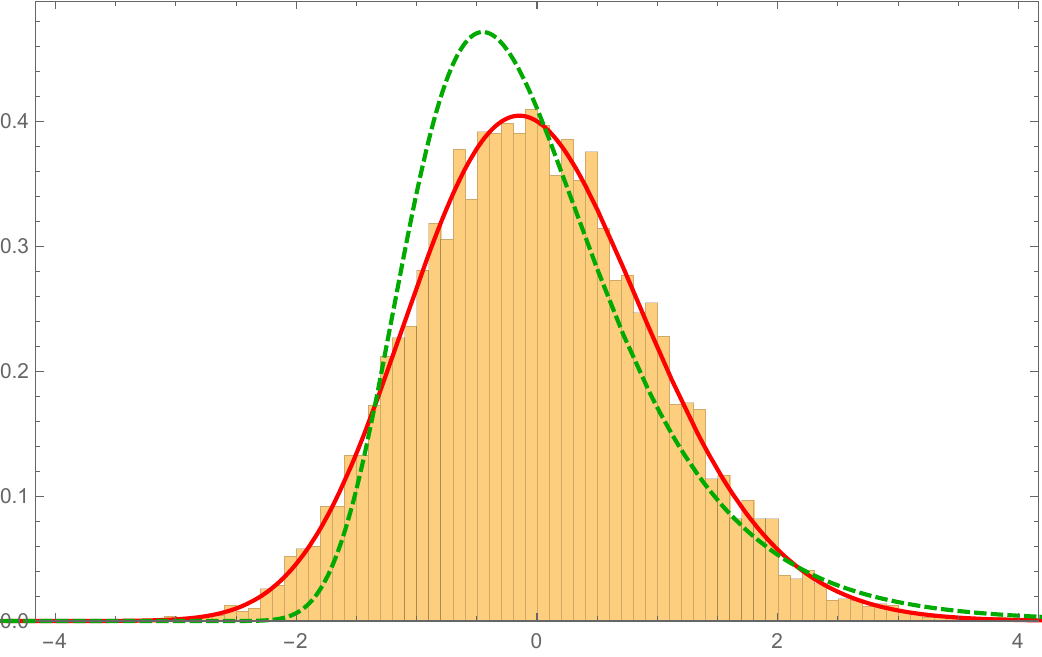}
         \caption{$W=500$ (Supercritical)}
     \end{subfigure}

     \caption{
     Comparison of the largest eigenvalue distributions for $N = 2000$ with varying bandwidth $W$. Each value of $W$ is based on $10000$ trials. The red curve represents the standard Tracy--Widom (TW1) distribution, and the green curve represents the Gumbel distribution. All data and distributions are normalized to have mean zero and unit variance. As $W$ increases, the empirical distribution shifts from the Gumbel law toward the Tracy--Widom law.
     }
     \label{fig:eigenvalue_transition}
\end{figure}

\begin{figure}[htbp]
    \centering
    \includegraphics[width=\textwidth]{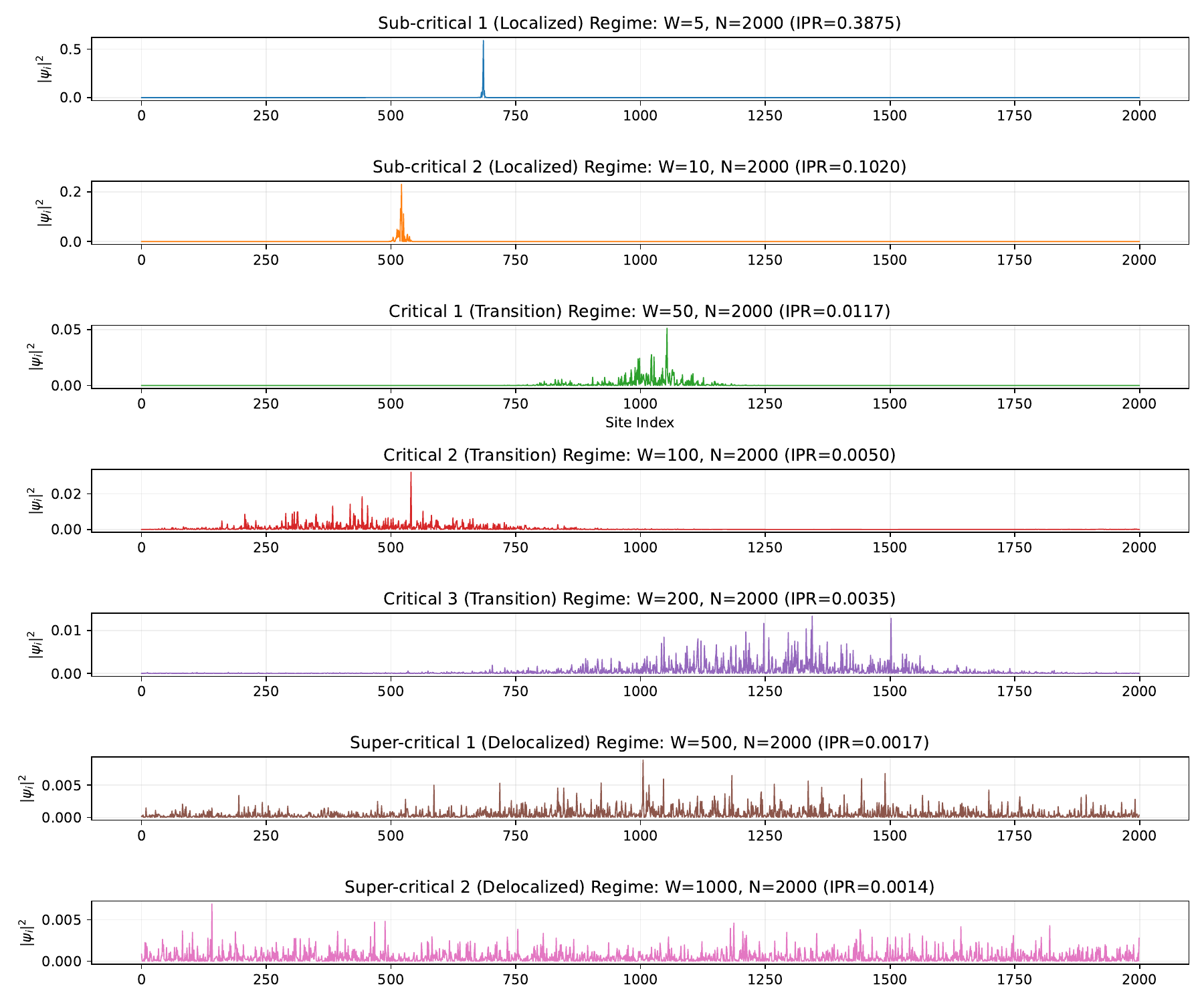}
    \caption{{Localization–delocalization transition of the eigenvector $\psi$ corresponding to the largest eigenvalue.}
The horizontal axis represents the coordinate $i$, and the vertical axis shows the squared amplitude $|\psi(i)|^2$ at each coordinate. As the bandwidth $W$ increases from $5$ to $1000$, the inverse participation ratio (IPR) drops significantly, marking the crossover from the subcritical (localized) regime to the universal supercritical (delocalized) regime.
    }
    \label{fig:eigenvector_simulation}
\end{figure}

\end{enumerate}

\section*{Acknowledgments} The authors are supported by the National Natural Science Foundation of China (Grant No. 12371157) and are grateful to Jiaqi Fan for valuable discussions. G. Zou is greatly indebted to his advisor, Roman Vershynin, for his support and for the excellent research environment provided at UC Irvine.

\let\oldthebibliography\thebibliography
\let\endoldthebibliography\endthebibliography
\renewenvironment{thebibliography}[1]{
\begin{oldthebibliography}{#1}
\setlength{\itemsep}{0.5em}
\setlength{\parskip}{0em}
}
{
\end{oldthebibliography}
}

\bibliographystyle{alpha}
\begin{spacing}{0}
\small
\newcommand{\etalchar}[1]{$^{#1}$}

\end{spacing}

\appendix

\section{  Chebyshev polynomial products}\label{sec:chebyshev}
\begin{lemma}[Linearization coefficients for Chebyshev products]\label{lem:c_asymptotics}
   For any given integer  $t \ge 2$,   the product of $t$   Chebyshev polynomials of the second kind admits a linearization
    \begin{equation}
        \Big(\frac{1}{m+1} U_m(x)\Big)^t = \sum_{k \ge 0} c_t(m; k) \, \frac{1}{k+1}U_k(x),
    \end{equation}
    where
    $c_t(m; k) = 0$ unless $k \equiv tm \pmod 2$.
    As $m \to \infty$, if   $k \equiv tm \pmod 2$ and
    $k = \lfloor \xi m \rfloor$ with $\xi \in (0, t)$,  then  the coefficients satisfy the asymptotic
    \begin{equation}
        c_t(m; k) \sim \frac{2}{m}   \, \xi \, \mathcal{P}_t(\xi),
    \end{equation}
    where   $\mathcal{P}_t(\xi)$ is defined  by \eqref{eq:Pt_def}.
\end{lemma}

\begin{proof}
    The coefficients are extracted via orthogonality,
    \begin{equation}\label{eq:exact_c}
        c_t(m; k) = \frac{2(k+1)}{\pi (m+1)^t} \int_0^\pi \frac{\sin^t((m+1)\theta) \sin((k+1)\theta)}{(\sin\theta)^{t-1}} \, d\theta.
    \end{equation}
    First, observe the parity symmetry: the integrand is invariant under $\theta \to \pi - \theta$ only if $k$ and $tm$ have the same parity; otherwise, it is antisymmetric around $\pi/2$, making the integral exactly zero.
    So in the non-zero case ($k \equiv tm \pmod 2$),  it suffices to consider the integral  over $ (0,\pi/2)$. Moreover,  as $m \to \infty$
    the leading asymptotic contribution comes from the region \(\theta \approx 0\).
         As $m \to \infty$, rescaling $\theta = z/m$, we obtain
    \begin{align}
        c_t(m; k)  &\sim \frac{4\xi m}{\pi m^t} \int_0^\infty \frac{(\sin z)^t \sin(\xi z)}{(z/m)^{t-1}} \frac{dz}{m} =  \frac{2\xi}{m} \left( \frac{2}{\pi} \int_0^\infty \frac{\sin^t z \sin(\xi z)}{z^{t-1}} dz \right) = \frac{2\xi}{m} \mathcal{P}_t(\xi).
    \end{align}
    This completes the proof.
\end{proof}

Similarly, we have the following lemma.
\begin{lemma}
\label{lem:c_perturbed}
    \begin{equation}
        \left(\frac{1}{m} U_{m-1}(x)\right) \left(\frac{1}{m+1} U_m(x)\right)^{t-1} = \sum_{k \ge 0} \tilde{c}_t(m; k) \, \frac{1}{k+1}U_k(x),
    \end{equation}
    where $\tilde{c}_t(m; k) = 0$ unless $k \equiv tm-1 \pmod 2$.
    As $m \to \infty$, for $k \equiv tm-1 \pmod 2$ and $k = \lfloor \xi m \rfloor$,
    \begin{equation}
        \tilde{c}_t(m; k) \sim \frac{2}{m}   \xi \, \mathcal{P}_t(\xi).
    \end{equation}
\end{lemma}
\section{Asymptotics of Jacobi $\theta_{\alpha}$ function}\label{sec:theta_function}

\begin{proof}[Proof of Lemma \ref{lem:asy_theta}]
We use  the dual representations of $\theta_{\alpha}$ provided by the {Poisson Summation Formula}. First, recall that the characteristic function of the $\alpha$-stable distribution is $ e^{-c_{\alpha}\|t\|^{\alpha}}$. For the scaled density $f_{\alpha}(x, \tau)$, its Fourier transform is given by
\begin{equation}
    \hat{f}_{\alpha}(\xi, \tau) = \int_{\mathbb{R}^d} e^{i \xi \cdot x} f_{\alpha}(x, \tau) dx = e^{-c_{\alpha} \tau \|\xi\|^{\alpha}}.
\end{equation}
Applying  the Poisson Summation Formula leads to   the following Fourier expansion on the torus $\mathbb{T}^d$:
\begin{equation}\label{eq:poisson_theta_refined}
    \theta_{\alpha}(x, \tau) = \sum_{m \in \mathbb{Z}^d} \hat{f}_{\alpha}(2\pi m, \tau) e^{-i 2\pi m \cdot x} = \sum_{m \in \mathbb{Z}^d} e^{-c_{\alpha} \tau \|2\pi m\|^{\alpha}} e^{-i 2\pi m \cdot x}.
\end{equation}

\textit{Proof of Part (i) (Small $\tau$ regime):} We revert to the spatial summation. Separating the $k=0$ term yields
\begin{equation}\label{eq:theta_spatial_split}
    \theta_{\alpha}(x, \tau) = \tau^{-\frac{d}{\alpha}} f_{\alpha}(x \tau^{-\frac{1}{\alpha}}) + \tau^{-\frac{d}{\alpha}} \sum_{k \in \mathbb{Z}^d \setminus \{0\}} f_{\alpha}((x+k)\tau^{-\frac{1}{\alpha}}).
\end{equation}
The $\alpha$-stable density satisfies the tail bound $f_{\alpha}(z) \le C(1 + \|z\|)^{-(d+\alpha)}$. For any $x \in \mathbb{T}^d$ of form $ (-1/2,1/2]^d$  and $k \neq 0$, we have $\|x+k\| \ge c\|k\|$. Thus, the remainder in \eqref{eq:theta_spatial_split} is bounded by
\begin{equation*}
    \tau^{-\frac{d}{\alpha}} \sum_{k \neq 0} \frac{C}{1 + (c\|k\|\tau^{-\frac{1}{\alpha}})^{d+\alpha}} \le C \tau^{-\frac{d}{\alpha}} \cdot \tau^{1+\frac{d}{\alpha}} \sum_{k \neq 0} \frac{1}{\|k\|^{d+\alpha}} = O(\tau).
\end{equation*}
 This confirms that $\theta_{\alpha}$ is localized around the $\alpha$-stable density as $\tau \to 0$.

\textit{Proof of Part (ii) (Large $\tau$ regime):} In \eqref{eq:poisson_theta_refined}, the term for $m=0$ is exactly $1$. For $\tau \geq 1$, the sum over $m \neq 0$ is dominated by the terms with $\|m\|=1$. Since there are only finitely many such terms and the rest decay even faster, we have
\begin{equation*}
    |\theta_{\alpha}(x, \tau) - 1| \le \sum_{m \in \mathbb{Z}^d \setminus \{0\}} e^{-c_{\alpha} (2\pi)^{\alpha} \tau \|m\|^{\alpha}} \le C e^{-c_{\alpha} (2\pi)^{\alpha} \tau},
\end{equation*}
which proves the exponential convergence to the uniform distribution.

\textit{Proof of Part (iii) (Uniform bound):}
For $\tau \ge 1$, $|\theta_{\alpha}(x, \tau)| \le 1 + \epsilon$ from Part (ii). For $\tau < 1$, the estimate in Part (i) together with the fact that $\|f_{\alpha}\|_{\infty} < \infty$ implies $\theta_{\alpha}(x, \tau) \le C \tau^{-d/\alpha}$. Combining these, we obtain $\theta_{\alpha}(x, \tau) \le C(1 + \tau^{-d/\alpha})$ for all $\tau > 0$.
\end{proof}

\section{Markov chain comparisons}\label{sec:markov_chain}

\subsection{Markov Chain comparison: random band matrices}\label{sec:band_comparison}
We impose the following assumptions on two integrable functions $f$ and $\tilde{f}$.
\begin{definition}\label{ass:ft}
There exist constants $ \alpha \in (0, 2]$, $c_0 > 0$, $\varepsilon,\delta > 0$, and $0 < \rho < 1$ such that the following assumptions are satisfied:
\begin{enumerate}
  \item (Normalization) Both  $f$ and $\widetilde{f}$ are probability densities
  \begin{equation}\label{eq:prob-density}
    f(x),~\widetilde{f}(x)\ge 0, \quad \int_{\mathbb R^d} f(u)\,du \;= \int_{\mathbb R^d} \widetilde{f}(u)\,du=1;
  \end{equation}
  \item (Small-frequency control) For all $|\xi|\le \varepsilon$,
  \begin{equation}\label{eq:small-freq}
    |\widehat f\!(\xi)|,~|\widehat {\tilde{f}}(\xi)| \le \exp(-c_0|\xi|^{\alpha}),\quad\Bigl|\widehat f\!(\xi)-\widehat {\tilde{f}}(\xi)\Bigr|\le C|\xi|^{\alpha+\delta}.
  \end{equation}
  \item (High-frequency contraction) For all $|\xi| > \varepsilon$,
  \begin{equation}\label{eq:high-freq}
    |\widehat f(\xi)|,|\widehat{ \tilde{f}}(\xi)| \;\le\; \rho < 1;
  \end{equation}
  \item (Long-range decay) For some   constants $T, K > d$ and $C_T, C_{K} > 1$,
  \begin{equation}\label{equ:D1}
      f(x),\tilde{f}(x)\le \frac{C_T}{(1+|x|)^{T}},
    \qquad |\widehat f(\xi)|, |\widehat {\tilde{f}}(\xi)| \;\le\; \frac{C_K}{(1+|\xi|)^{K}}
      \qquad  x,\xi \in \mathbb R^d.
  \end{equation}
\end{enumerate}
\end{definition}

\begin{proposition}[$\ell^{\infty}$-difference]\label{thm:clt-upper}
Let $p_n(x,y)$ and $\widetilde{p}_n(x,y)$ be the $n$-step transition probability given by the Markov chain in Definition \ref{defmodel}, with profile functions $f$ and $\widetilde{f}$  satisfying the assumptions in Definition~\ref{ass:ft}.
If $nW^{-K}=o(1)$, then  for all $x \in \Lambda_L$,
\begin{equation}\label{equ:band_mixing}
    \Bigl|\, p_n(0,x) - \tilde{p}_n(0,x)\,\Bigr|
    = O(W^{-d} n^{-\tfrac{d+\delta}{\alpha}})  +\; O( n W^{-K}).
\end{equation}
\end{proposition}

\begin{proof}
We begin with the discrete Fourier representation
\begin{equation}\label{equ:D2}
p_n(0,x)=\frac{1}{N}\sum_{k\in\Lambda_L}\bigl(\widehat p(k)\bigr)^n e^{\frac{2\pi i k\cdot x}{L}},
\end{equation}
where the discrete Fourier transform of the variance profile
\begin{align}
    \widehat{p}(k)&=\frac{1}{M}\sum_{x\in \Lambda}\sum_{m\in \mathbb{Z}^d}
    f\!\Big(\frac{x+mL}{W}\Big)e^{-\frac{2\pi i k\cdot x}{L}}
    =\frac{1}{M}\sum_{x\in \mathbb{Z}^d} f\!\Big(\frac{x}{W}\Big)e^{-\frac{2\pi i k\cdot x}{L}}\notag\\
    &=\frac{W^d}{M}\sum_{m\in \mathbb{Z}^d}\widehat{f}\!\Big(W\bigl(\tfrac{k}{L}+m\bigr)\Big)
    =\frac{W^d}{M}\widehat{f}\!\Big(\tfrac{Wk}{L}\Big)+O(W^{-K}),
\end{align} where in the last two equalities  we have used {Poisson summation} and dropped $|m|>0$ terms via \eqref{equ:D1}. In the Poisson summation step, both series converge absolutely by \eqref{equ:D1}, hence uniformly.   {Furthermore,  they define bounded continuous functions. By Plancherel identity  they agree in the sense of $L^2$ distance, and since two continuous functions that are equal a.e. must coincide everywhere on $\mathbb{R}^d$, the identities hold pointwise.}

Taking absolute values in \eqref{equ:D2} and applying the triangle inequality, we obtain
\begin{equation}\label{equ:D5}
|p_n(0,x)-\tilde{p}_n(0,x)| \le {{\frac{W^d}{M}}} \frac{1}{N}\sum_{k\in\Lambda_L}\Bigl|(\widehat f\!\Bigl(\tfrac{Wk}{L}\Bigr))^n-(\widehat {\tilde{f}}\!\Bigl(\tfrac{Wk}{L}\Bigr))^n\Bigr| \;+\; O\!\left(nW^{-K}\right),
\end{equation}
where the factor $n$ in the error term comes from expanding $(\widehat f+O(W^{-K}))^n$.
Next, we estimate the main sum. For this, we obtain
\begin{align}
    &\frac{1}{N}\sum_{k\in\Lambda_L}\Bigl|(\widehat f\!\Bigl(\tfrac{Wk}{L}\Bigr))^n-(\widehat {\tilde{f}}\!\Bigl(\tfrac{Wk}{L}\Bigr))^n\Bigr|
    \quad\le \frac{1}{L^d}\sum_{k\in\mathbb{Z}^d}\Bigl|(\widehat f\!\Bigl(\tfrac{Wk}{L}\Bigr))^n-(\widehat {\tilde{f}}\!\Bigl(\tfrac{Wk}{L}\Bigr))^n\Bigr|\notag\\
    &\quad=\Biggl(\sum_{|\tfrac{Wk}{L}|<\epsilon} \;+\;
      \sum_{\epsilon\le |\tfrac{Wk}{L}|<100 C_K} \;+\;
      \sum_{100 C_K\le |\tfrac{Wk}{L}|}\Biggr) \frac{1}{L^d}
      \Bigl|(\widehat f\!\Bigl(\tfrac{Wk}{L}\Bigr))^n-(\widehat {\tilde{f}}\!\Bigl(\tfrac{Wk}{L}\Bigr))^n\Bigr|\notag \\
    &\quad:= I_1+I_2+I_3.
\end{align}

For $I_1$,
\begin{equation}
\begin{aligned}
        \Bigl|(\widehat f\!\Bigl(\tfrac{Wk}{L}\Bigr))^n-(\widehat {\tilde{f}}\!\Bigl(\tfrac{Wk}{L}\Bigr))^n\Bigr| &\le n \max\Big\{\Bigl|\widehat f\!\Bigl(\tfrac{Wk}{L}\Bigr)\Bigl|^{n-1},~\Bigl|\widehat {\tilde{f}}\!\Bigl(\tfrac{Wk}{L}\Bigr)\Bigr|^{n-1}\Big\} \Bigl|\widehat f\!\Bigl(\tfrac{Wk}{L}\Bigr)-\widehat {\tilde{f}}\!\Bigl(\tfrac{Wk}{L}\Bigr)\Bigr|.
\end{aligned}
\end{equation}
Thus, taking $\xi=\frac{Wk}{L}$, we have
\begin{equation}
\begin{aligned}
        I_1 &\le \frac{n}{L^d}\frac{L^d}{W^d}\int_{|\xi|\le \epsilon}C|\xi|^{\alpha+\delta}\exp\!\bigl(-c_0 n|\xi|^\alpha\bigr)\,d\xi\le \; C W^{-d} n^{-\tfrac{d+\delta}{\alpha}}=O\big(W^{-d} n^{-\tfrac{d+\delta}{\alpha}}\big).
\end{aligned}
\end{equation}

For $I_2$, Assumption~\ref{ass:ft}(3) gives
\begin{equation}
\begin{aligned}
    I_2 & =\;
      \frac{1}{L^d}\sum_{\epsilon\le |\tfrac{Wk}{L}|<100 C_K} \Bigl|(\widehat f\!\bigl(\tfrac{Wk}{L}\bigr)\bigr)^n-\bigl(\widehat {\tilde{f}}\!\bigl(\tfrac{Wk}{L}\bigr)\bigr)^n\Bigr|\\
      &\le \frac{1}{L^d}\sum_{\epsilon\le |\tfrac{Wk}{L}|<100 C_K} \Bigl|(\widehat f\!\bigl(\tfrac{Wk}{L}\bigr)\bigr)^n\Bigr|+\Bigl|\bigl(\widehat {\tilde{f}}\!\bigl(\tfrac{Wk}{L}\bigr)\bigr)^n\Bigr|\\
      &\le \frac{1}{L^d}\sum_{\epsilon\le |\tfrac{Wk}{L}|<100 C_K}2\rho^n\\
      &\;\le\; \frac{1}{L^d}\Bigl(\tfrac{100C_K L}{W}\Bigr)^d \rho^n \;\le\; C'_K W^{-d}\rho^n
    =O(W^{-d} n^{-\tfrac{d+\delta}{\alpha}}).
\end{aligned}
\end{equation}

For $I_3$, we have
\begin{equation}
\begin{aligned}
    &\sum_{100 C_K\le |\tfrac{Wk}{L}|} \Bigl|(\widehat f\!\bigl(\tfrac{Wk}{L}\bigr)\bigr)^n-\bigl(\widehat {\tilde{f}}\!\bigl(\tfrac{Wk}{L}\bigr)\bigr)^n\Bigr|
    \le \sum_{100 C_K\le |\tfrac{Wk}{L}|}\Bigl|\widehat f\!\Bigl(\tfrac{Wk}{L}\Bigr)\Bigr|^n+\Bigr|\widehat {\tilde{f}}\!\Bigl(\tfrac{Wk}{L}\Bigr)\Bigr|^n.
\end{aligned}
\end{equation}
Using Assumption~\ref{ass:ft}(4) and taking $\xi=\frac{Wk}{L}$, we derive
\begin{equation}
    I_3 \;\le\; C W^{-d}\int_{|\xi|>50C_K}\Bigl(\tfrac{C_K}{(1+|\xi|)^{K}}\Bigr)^n d\xi
    \;=O(W^{-d} n^{-\tfrac{d+\delta}{\alpha}}).
\end{equation}

Combining these bounds with \eqref{equ:D5}, we conclude the proof.
\end{proof}

\begin{proposition}[$\ell^1$-difference]\label{prop:L1_difference}If $f$ and $\widetilde{f}$ satisfy the assumptions in Definition  ~\ref{ass:ft}, then as $n\rightarrow\infty$,
\begin{equation}\label{equ:band_mixing_L_infinity}
    \sum_{x\in \Lambda_L}\Bigl|\, p_n(0,x) - \tilde{p}_n(0,x)\,\Bigr|
    = o(1).
\end{equation}
\end{proposition}
\begin{proof}
    Let \(\Delta\) tend to infinity very slowly. Split the sum into inner and outer parts:
    \begin{align}
        \sum_{x}\Bigl| p_n(0,x) - \tilde{p}_n(0,x) \Bigr| & = \sum_{|x|\le n^{\frac{1}{\alpha}} W \Delta} \Bigl| p_n(0,x) - \tilde{p}_n(0,x) \Bigr|
           + \sum_{|x| > n^{\frac{1}{\alpha}} W \Delta} \Bigl| p_n(0,x) - \tilde{p}_n(0,x) \Bigr| \notag \\
        &= \sum_{|x|\le n^{\frac{1}{\alpha}} W \Delta} o\!\left(W^{-d} n^{-\frac{d}{\alpha}}\right)
           + \sum_{|x| > n^{\frac{1}{\alpha}} W \Delta} \Bigl| p_n(0,x) - \tilde{p}_n(0,x) \Bigr| \tag{by \eqref{equ:band_mixing}} \\
        &= \sum_{|x| > n^{\frac{1}{\alpha}} W \Delta} \Bigl| p_n(0,x) - \tilde{p}_n(0,x) \Bigr| + o(1).
    \end{align}
    By the triangle inequality,
    \begin{align}
        &\sum_{|x| > n^{\frac{1}{\alpha}} W \Delta} \Bigl| p_n(0,x) - \tilde{p}_n(0,x) \Bigr|
        \le \sum_{|x| > n^{\frac{1}{\alpha}} W \Delta} \bigl( p_n(0,x) + \tilde{p}_n(0,x) \bigr) \notag \\
        &= 2 - \sum_{|x|\le n^{\frac{1}{\alpha}} W \Delta} \bigl( p_n(0,x) + \tilde{p}_n(0,x) \bigr) \notag\\
        &\le 2 - 2\sum_{|x|\le n^{\frac{1}{\alpha}} W \Delta} \tilde{p}_n(0,x)
           + \sum_{|x|\le n^{\frac{1}{\alpha}} W \Delta} \bigl| p_n(0,x) - \tilde{p}_n(0,x) \bigr|=o(1).
    \end{align}
    Here for the stable transition probability \(\tilde{p}_n\), the sum over the ball \(|x|\le n^{\frac{1}{\alpha}} W \Delta\) tends to \(1\) as \(\Delta\to\infty\) (since the total mass is \(1\) and the tail is negligible).
    Moreover, the inner sum of differences is already \(o(1)\) by \eqref{equ:band_mixing}. Consequently, we obtain the desired estimate and complete the proof.
\end{proof}

\subsection{Markov chain comparison: Wegner model}

For the block Wegner model, we index the blocks by the discrete torus $\mathbb{T}_D^d := (\mathbb{Z}/D\mathbb{Z})^d$. Let $M$ be the size of each block, such that the total number of states is $N = M \cdot D^d$. For each $x \in [N]$, we denote its corresponding block index by $\alpha(x) \in \mathbb{T}_D^d$, and let $\mathcal{I}_{\alpha} = \{x : \alpha(x) = \alpha\}$ be the set of microstates within block $\alpha$.

The transition probability $p_n(x,y)$ of the associated Markov chain is determined by a transition kernel that is uniform within each block. Specifically, let $\overline{P}$ be the transition matrix of a lazy random walk on $\mathbb{T}_D^d$ with single-step probabilities:
\begin{equation}
    \overline{p}(\alpha,\beta) = \begin{cases}
        1-\lambda, & \text{if } \alpha=\beta, \\
        \frac{\lambda}{2d}, & \text{if } |\alpha-\beta|=1, \\
        0, & \text{otherwise}.
    \end{cases}
\end{equation}

\begin{lemma}\label{lem:reduction}
Let $p_n(x,y)$ be the $n$-step transition probability of the Wegner model. For any $x \in \mathcal{I}_{\alpha}$ and $y \in \mathcal{I}_{\beta}$, we have
\begin{equation}
    p_n(x,y) = \frac{1}{M} \overline{p}_n(\alpha, \beta),
\end{equation}
where $\overline{p}_n(\alpha, \beta)$ is the $n$-step transition probability associated with $\overline{P}$.
\end{lemma}
\begin{proof}
By the block-tridiagonal structure of the variance profile, the transition probability $\sum_{y \in \mathcal{I}_\beta} p(x,y)$ depends only on the block index $\alpha(x)$. Since the entries are uniform within each block, a simple induction on $n$ shows that the probability is distributed uniformly over the microstates of the target block, yielding the factor $1/M$.
\end{proof}

The asymptotic behavior of $p_n(x,y)$ is governed by the effective coupling scale $n\lambda$. Let $\mathcal{H}_t(\alpha, \beta)$ denote the continuous-time random walk kernel (heat kernel) on $\mathbb{T}_D^d$, and let $\mathcal{G}_{\sigma^2}(u)$ denote the Gaussian density on $\mathbb{R}^d$ with variance $\sigma^2$.

\begin{proposition}\label{prop:wegner_asymptotics}
For $x \in \mathcal{I}_{\alpha}$ and $y \in \mathcal{I}_{\beta}$, the following asymptotic regimes hold as $n, N \to \infty$:
\begin{enumerate}
    \item \textbf{Frozen regime} ($n\lambda \ll 1$):
    The transition probability is dominated by the initial state:
    \begin{equation}
        \left| p_n(x,y) - \frac{1}{M}\delta_{\alpha\beta} \right| = O(n\lambda).
    \end{equation}

    \item \textbf{Skellam regime} ($n\lambda \to \mu \in (0,\infty)$):
    The transition probability converges to the continuous-time random walk kernel:
    \begin{equation}
        \left| p_n(x,y) - \frac{1}{M}p^{\mathrm{(Skellam)}}(\beta-\alpha,\mu) \right| = O(n^{-1} M^{-1}).
    \end{equation}
    Here $p^{\mathrm{(Skellam)}}$ was defined in \eqref{equ:skellam}.
    \item \textbf{Diffusive regime} ($n\lambda \to \infty$):
    The transition probability satisfies a local central limit theorem. For $\sigma^2 = n\lambda/d$, we have
    \begin{equation}
        \sup_{x,y} \left| p_n(x,y) - \frac{1}{M}\mathcal{G}_{n\lambda/d}(\alpha-\beta) \right| = o\Big( \frac{1}{M (n\lambda)^{d/2}} \Big).
    \end{equation}
\end{enumerate}
\end{proposition}
\begin{proof}
The results for the first two regimes follow from the Poisson Limit Theorem for the number of non-lazy jumps (see e.g., \cite[Theorem 3.6.1]{durrett2019}). The third regime follows from the local central limit theorem for random walks (see e.g., \cite[Theorem 3.4.10]{durrett2019}).
\end{proof}
Similar to Proposition \ref{prop:L1_difference}, we can obtain the $\ell^1$-difference for Wegner orbital model.
\begin{corollary}
Let $\widetilde{p}$ be the reference kernel defined in Theorem \ref{prop:alpha_llt} with $\alpha=2$. In the diffusive regime ($n\lambda \to \infty$), we have
\begin{equation}
    \sup_{x,y} |p_n(x,y) - \widetilde{p}_n(x,y)| = o\Big( \frac{1}{M (n\lambda)^{d/2}} \Big),
\end{equation}
and
\begin{equation}
    \sup_{x} \sum_{y \in [N]} |p_n(x,y) - \widetilde{p}_n(x,y)| = o(1).
\end{equation}
\end{corollary}

\end{document}